\documentclass[12pt]{amsart}
\usepackage{amssymb,amsmath,amsthm,amscd}
\usepackage[all]{xy}
\usepackage{tikz}
\usepackage[inline]{enumitem}

\setlist{nolistsep}
\allowdisplaybreaks

\newtheorem{theorem}{Theorem}
\newtheorem{lemma}[theorem]{Lemma}
\newtheorem{corollary}[theorem]{Corollary}
\newtheorem{proposition}[theorem]{Proposition}
\theoremstyle{remark}
\newtheorem{remark}{Remark}
\numberwithin{theorem}{section} \numberwithin{equation}{section}

\newcommand{\chr}[4]
{\begin{bmatrix}
 #1 & #2\\[0.2em]
 #3 & #4
\end{bmatrix}}

\newcommand{\emptyproof}{\vspace*{-0.3cm}\hfill $\square$\smallskip}

\begin{document}
\title[$(1,2)$-polarized Kummer surfaces]{On the geometry of $(1,2)$-polarized Kummer surfaces}
\author{Adrian Clingher and Andreas Malmendier}

\address{Dept.~\!of Mathematics and Computer Science, University of Missouri -- St.~\!Louis, St.~\!Louis, MO 63121}
\email{clinghera@umsl.edu}

\address{Dept.~\!of Mathematics and Statistics, Utah State University, Logan, UT 84322}
\email{andreas.malmendier@usu.edu}

\begin{abstract}
We discuss several geometric features of a Kummer surface associated with a $(1,2)$-polarized abelian 
surface defined over the field of complex numbers. 
In particular, we show that any such Kummer surface can be modeled as the double cover of the projective plane 
branched along six lines, three of which meet a common point.
The proof uses certain explicit pencils of plane quartic bielliptic genus-three curves
whose associated Prym varieties are naturally $(1,2)$-polarized abelian surfaces.
\end{abstract}


\maketitle

\section{Introduction and Summary or results}
A rich source of examples for the theory of complex algebraic K3 surfaces is given by the so-called double sextic surfaces, 
i.e. surfaces obtained after taking double covers of the projective plane branched over a projective sextic curve. A specially 
interesting class within this realm corresponds to the case when the sextic curve in question is reducible and consists of six 
distinct lines. This case, which generally leads to K3 surfaces whose generic Picard rank is $16$, has two particular limit situations. 
One limit is given by the case when the six lines are simultaneously tangent to a common conic. The other corresponds 
to the situation when three of the six lines meet at a common point.   

The first limit is rather famous, as the double sextic construction leads in this situation to Kummer surfaces associated to 
Jacobians of genus-two curves, classical objects in algebraic geometry whose properties have been analyzed in numerous works 
over the last more than a hundred years. The second limit case above has received considerably less scrutiny and is the main 
focus of the present paper. We show that the K3 surfaces obtained here are also of Kummer type and are associated to abelian surfaces 
with a special non-principal polarization.   

Kummer surfaces have been traditionally studied in connection with abelian surfaces with principal polarizations. 
In this case, the rational map associated with the twice theta-divisor linear system provides a map to complex projective 
three-space that factors via an embedding of the Kummer surface. This symmetry allows one to obtain a Kummer surface 
associated with a principally polarized
abelian surface as the minimal resolution of a singular quartic surface, or Kummer quartic, 
with the maximal number of simple nodes which -- for a quartic in $\mathbb{P}^3$ -- is sixteen. 
By a suitable projection from a node of the Kummer quartic, one then obtains the double sextic realization of the Kummer surface 
from the first limit above, i.e. as a double cover of the projective plane branched along the union of six lines 
all tangent to a conic. These geometric features traditionally allow one to establish a normal form for such Kummer quartics
with coefficients and coordinates given in terms of modular and Jacobi forms of genus two.

The aim of this paper is to study certain features of $(1,2)$-polarized abelian surfaces and
their associated Kummer surfaces. We shall give explicit descriptions -- analogous to descriptions of the Kummer surfaces
as branched double cover, quartic surface in three-space, and Jacobian variety of a genus-two curve in the case of a principal polarization --
for the $(1,2)$-polarized Kummer surface.  

In Theorem~\ref{prop1}, we shall prove that a $(1,2)$-polarized Kummer surface is birationally
equivalent to the quotient of the symmetric square of a bielliptic and hyperelliptic genus-three curve by the action of the bielliptic involution on each factor.
This result, in turn, allows us to derive a normal form for a $(1,2)$-polarized Kummer surface as 
an otic surface in weighted projective three-space. The proposition is based on a particular feature of genus-three curves
(see for instance \cite{MR1816214}): a genus-three curve that is both hyperelliptic and bielliptic is always the double
cover of a genus-two curve. This double cover then lifts to a rational degree-two map from the Kummer surface with $(1,2)$-polarization
onto a principally polarized Kummer surface and determines the coefficients defining the Kummer surface with $(1,2)$-polarization
in terms of modular forms of genus two and level two.

On the other hand, it is known (see  \cite{MR946234}) that an embedding of a smooth bielliptic genus-three curve into an abelian 
surface is uniquely determined (up to a two-torsion translation) 
by the requirement that the abelian involution restricts to the bielliptic involution.
In Theorem~\ref{cor:pencil} we shall determine an explicit pencil of quartic bielliptic genus-three curves whose associated 
Prym variety is naturally isomorphic to the $(1,2)$-polarized abelian surface. 

The kernel of multiplication by two on a principally polarized abelian surface is a group of order 16.
But only certain subgroups of order four appear as kernels of a two-isogeny with another principally polarized
abelian surface, the so-called maximal isotropic subspace of the two-torsion points.
Taking the quotient of a principally polarized abelian surface by a maximal isotropic subspace is again a principally polarized abelian surface,
the $(2,2)$-isogenous abelian surface. Conversely, a second maximal isotropic subspace on the $(2,2)$-isogenous abelian surface allows one to construct a 
dual $(2,2)$-isogeny that maps back to the original principally polarized abelian surface. When the principally polarized abelian surfaces are Jacobians of 
generic genus-two curves with full level-two structure, their moduli can be described explicitly in terms of genus-two theta-functions.
Equations~(\ref{relations_RosRoots}) relate the moduli of a generic genus-two curve and one of its $(2,2)$-isogenous curve.

By taking the above arguments to the level of Kummer surfaces, one sees that the $(2,2)$-isogenies on principally polarized abelian surfaces 
induce algebraic maps between the corresponding Kummer surfaces. In Theorems~\ref{cor:psi} and~\ref{cor:hat_psi} we shall prove that 
the morphisms between Kummer surfaces can be factored into a sequence of two rational double covers that act either
by base transformations or by fiberwise two-isogenies relative to three particular Jacobian elliptic fibrations.
In fact, the translations by generators of the maximal isotropic subgroups descend to projective automorphisms of the corresponding Kummer
surfaces and can be understood as Van Geemen-Sarti involutions and Nikulin involutions covering a degree-two base transformation.
The $(1,2)$-polarized Kummer surfaces appear then as the intermediate spaces in this factorization.
In Theorem~\ref{cor:equiv} we shall prove that there is a natural isomorphism between the constructed $(1,2)$-polarized Kummer surface
and its $(2,2)$-isogenous Kummer surface. 

The above special Jacobian elliptic fibrations on the Kummer surface associated with the Jacobian of a generic genus-two curve
give rise to Jacobian elliptic fibrations on the associated $(1,2)$-polarized Kummer surface. One of them
is particularly simple: it is the quadratic twist of a rational elliptic surface. This leads, in Theorem~\ref{prop:double_cover}, to a simple geometric 
description of a $(1,2)$-polarized Kummer surface: the $(1,2)$-polarized Kummer surface
is the double cover of the projective plane branched along a six lines, exactly three of which have a common point.

This article is structured as follows: in Section~\ref{sec:background} we discuss some of the background
and needed related work, most importantly the notions of rational double cover of a Kummer surface and an even eight.
In Section~\ref{genus-two}, we provide the necessary background to describe the moduli and $(2,2)$-isogenies
of principally polarized abelian surfaces that are Jacobians of generic genus-two curves with full level-two structure. 
In Section~\ref{sec:kummers} we prove the main results of the article.
A short discussion on the non-generic Picard rank $18$ case is also included.
Section~\ref{sec:summary} gives an overview of all constructed fibrations, morphisms, and involutions.

It should be noted that the elliptic fibration used in Theorem~\ref{prop:double_cover} has an interpretation in the context of string theory: 
by Sen's construction, it is the elliptic fibration that provides -- in the presence of non-zero masses -- an embedding of a  
supersymmetric Yang-Mills theory into F-theory. This makes the moduli in Equations~(\ref{relations_RosRoots})
observables of both an F-theoretic description of a type II-B string background and  the massive Seiberg-Witten curve for a $\mathcal{N}=2$ 
supersymmetric $SU(2)$ Yang-Mills theory with $N_F=4$ flavor hypermultipletts
embedded into this string theory. As we will describe in forthcoming work, this picture can be generalized and has important applications in string theory.

\section*{Acknowledgments}
\setcounter{section}{1}
The first author acknowledges support from the Simons Foundation through grant no.~208258.

\section{Background}
\label{sec:background}
Let $\mathbf{A}$ denote an abelian surface and let $-\mathbb{I}: a \mapsto -a$ be the involution automorphism.
The quotient, $\mathbf{A}/\langle -\mathbb{I} \rangle$, is a singular surface with sixteen ordinary double points. 
Its minimum resolution, $\operatorname{Kum}(\mathbf{A})$, is a special K3 surface called the Kummer surface associated to $\mathbf{A}$.

\subsection{Polarizations of abelian surfaces}
For an abelian surface $\mathbf{A}\cong \mathbb{C}^2/\Lambda$, a polarization is 
a positive definite hermitian form $H$ on $\mathbb{C}^2$,  satisfying $E = \operatorname{Im} \, H(\Lambda,\Lambda) \subset \mathbb{Z}$. 
Such a positive definite hermitian form corresponds to the class of an ample line bundle in the N\'eron-Severi group $\operatorname{NS}(\mathbf{A})$.
It is known that there always exists a basis $\langle e_1, e_2, f_1 f_2\rangle$ of $\Lambda$ such that $E$ is given 
by the matrix $\bigl(\begin{smallmatrix}
0&D\\ -D&0
\end{smallmatrix} \bigr)$
with $D=\bigl(\begin{smallmatrix}
d_1&0\\ 0&d_2
\end{smallmatrix} \bigr)$ where $d_1, d_2 \ge 0 $ and $d_1$ divides $d_2$, i.e., $d_1 \mid d_2$. Therefore, polarizations are labelled by such pairs 
of integers $(d_1, d_2)$.

The following facts are well known: If $\mathbf{A}=\operatorname{Jac}(\mathcal{C})$ is the Jacobian of a smooth curve $\mathcal{C}$ of genus two, then the hermitian 
form associated to the divisor class $[\mathcal{C}]$ is a polarization of type $(1, 1)$, also called a principal polarization. A curve $\mathcal{C}$ of genus-two will be called 
generic if $\operatorname{NS}(\operatorname{Jac} \mathcal{C})=\mathbb{Z}[\mathcal{C}]$. Conversely, over the complex numbers a principally polarized abelian surface is 
either the Jacobi variety of a smooth curve of genus two with the theta-divisor or the product of two complex elliptic curves with the product polarization. 
Moreover, every polarization is induced by a principal polarization via an isogeny of abelian surfaces~\cite[Sec.~4]{MR2062673}.
For an abelian surface $\mathbf{A}$ with polarization $\mathcal{L}$, a well known result of Lefschetz states that when $d_1 \ge 3$ then the theta-functions, 
which are sections of the associated rational map $\Phi_\mathcal{L}: \mathbf{A} \to \mathbb{P}^N$ provide an embedding of the abelian surface into projective space.
Moreover, theorems in \cite[Sec.~4]{MR2062673}  give a complete answer to the question whether $\Phi_\mathcal{L}$ is an embedding if $d_1=2$.
If $2 \mid d_1$ and $d_1 \ge 4$ or $3 \mid d_1$, Riemann's theta relations or cubic theta relations, respectively, describe the image
of  $\Phi_\mathcal{L}$ in $\mathbb{P}^N$ ~\cite[Sec.~7]{MR2062673}. If $d_1=1$ and $d_2 \ge 5$, Reider's theorem provides an explicit criterion for $\Phi_\mathcal{L}$
to be an embedding \cite[Sec.~10]{MR2062673}. The image of $\Phi_\mathcal{L}$ in the case of polarizations of type $(1,4)$ and $(2,2)$ was also described in \cite[Sec.~10]{MR2062673}. In \cite{MR1602020,MR1827859}, the equations of projectively embedded abelian surfaces with polarizations of type $(1,d)$ for $d\ge 5$ were given.

\subsubsection{$(1,2)$-polarizations}
The case of $(1,2)$-polarized abelian surfaces has been studied to a lesser extent -- to the authors' knowledge only in \cite{MR946234}, \cite[Sec.~10]{MR2062673}, \cite[Sec.~2.4]{MR2363136orig}, and \cite[Sec.~6]{MR2888217orig}.  Let us review some of the basic facts regarding $(1,2)$-polarized abelian surfaces:
if $\mathbf{B}$ is an abelian surface with a line bundle $\mathcal{L}$ of self-intersection $\mathcal{L}\cdot \mathcal{L}=4$,
the Riemann-Roch theorem then states that $\chi(\mathcal{L})=\mathcal{L}\cdot \mathcal{L}/2=2$. Moreover, the line bundle $\mathcal{L}$
is ample if and only if the Hodge numbers satisfy $h^i(\mathcal{L})=0$ for $i=1,2$ (and $\mathcal{L}\cdot \mathcal{L}=2>0$). On the other hand,
if the ample line bundle $\mathcal{L}$ is of type $(d_1,d_2)$ we have $h^0(\mathcal{L})=d_1d_2$ and the associated rational
map is denoted by $\Phi_{\mathcal{L}}: \mathbf{B} \to \mathbb{P}^{d_1d_2-1}$. Therefore,
it follows that an ample line bundle with self-intersection $\mathcal{L}\cdot \mathcal{L}=4$ necessarily defines a polarization
of type $(1,2)$ on $\mathbf{B}$ and the associated rational map is $\Phi_{\mathcal{L}}: \mathbf{B} \to \mathbb{P}^1$.
However, since $h^0(\mathcal{L})=\mathcal{L}\cdot \mathcal{L}/2=2>0$ the map $\Phi_{\mathcal{L}}$ is not a fibration
and the linear system $|\mathcal{L}|$ induces a linear pencil on $\mathbf{B}$.

Effective divisors on $\mathbf{B}$ are curves, and for any curve $\mathcal{B}$ on $\mathbf{B}$
the arithmetic genus is given by $p_a(\mathcal{B})=1-\chi(\mathcal{O}_{\mathcal{B}})$, and is by the adjunction formula
related to the self-intersection number by $2 p_a(\mathcal{B}) -2 =\mathcal{B}\cdot \mathcal{B}$.
Therefore, any curve $\mathcal{B}$ in the
linear system $|\mathcal{L}|$ of the ample line bundle $\mathcal{L}$ of type $(1,2)$ has $\mathcal{B}\cdot \mathcal{B}=4$ and 
$p_a(\mathcal{B})=3$.

Barth \cite[Sec.~1.1]{MR946234} gave an effective criterion for the existence of an irreducible curve $\mathcal{B} \in |\mathcal{L}|$ and proves
that if there exists such an irreducible element in $|\mathcal{L}|$, then the general member in $|\mathcal{L}|$ is irreducible and smooth, and the linear 
pencil has exactly four distinct base points $\{q_0, q_1, q_2, q_3\}$. 
We assume hereafter that any element of $|\mathcal{L}|$ is irreducible. 
In particular, since $p_a(\mathcal{B})=3$  a curve $\mathcal{B}\in |\mathcal{L}|$ is either smooth of genus three or an irreducible curve of geometric genus two 
with a double point. For the natural isogeny $T(\mathcal{L}): \mathbf{B} \to \operatorname{Pic}^0( \mathbf{B} )$ given by $t \mapsto \mathcal{L}^{-1}
\otimes t^*\mathcal{L}$, Barth proves that $\{q_0, q_1, q_2, q_3\} \cong  \ker T(\mathcal{L}) \cong (\mathbb{Z}/2)^2$.
Hence, the points $\{q_0, q_1, q_2, q_3\}$ are of order-two, i.e., fixed points by the $(-\mathbb{I})$-involution on $\mathbf{B}$,
and no curves $\mathcal{B} \in |\mathcal{L}|$ are singular at any of the base points.
Moreover, any curve $\mathcal{B} \in |\mathcal{L}|$ passing through the 12 remaining order-two points $\{q_4, \dots, q_{15}\}$ is singular, and
therefore must be an irreducible curve of geometric genus two with one node.

Another important result of \cite[Sec.~1.1]{MR946234} is that the line bundle $\mathcal{L}$ is symmetric and the $(-\mathbb{I})$-involution on $\mathbf{B}$
restricts to an involution $\imath^{\mathcal{B}}$ on each curve $\mathcal{B} \in |\mathcal{L}|$. If $\mathcal{B} \in |\mathcal{L}|$ is smooth, then 
by the Riemann-Hurwitz formula the quotient $\mathcal{F}=\mathcal{B}/\langle \imath^{\mathcal{B}} \rangle$ is an elliptic curve. 
The curve $\mathcal{B}$ is called a \emph{bielliptic curve}
since it admits a so-called elliptic involution $\imath^{\mathcal{B}}$ covering a degree-two morphism 
$\phi^{\mathcal{B}}_{\mathcal{F}}$ onto the elliptic curve $\mathcal{F}$. The four branch points of $\phi^{\mathcal{B}}_{\mathcal{F}}$ 
are exactly $\{q_0, q_1, q_2, q_3\}$.  Conversely, it was proved that for any smooth genus-three curve $\mathcal{B}$ 
the property of admitting a bielliptic involution is equivalent to admitting an embedding
into an abelian surface $\mathbf{B}$ \cite[Prop.~1.8]{MR946234}. By the universality of Jacobian varieties, 
it then follows that the embedding $\mathcal{B} \hookrightarrow\mathbf{B}$
is up to a two-torsion translation uniquely determined by the property that the abelian involution restricts to the
bielliptic involution, i.e., $\imath^{\mathcal{B}} = -\mathbb{I} \mid_{\mathcal{B}}$ \cite[Prop.~1.10]{MR946234}.

To any double cover of a curve one can associate a \emph{Prym variety}. 
In our situation, choosing a branch point $q_0 \in \mathcal{B}$ of $\phi^{\mathcal{B}}_{\mathcal{F}}: \mathcal{B} \to \mathcal{F}$
one obtains an embedding $\mathcal{B} \hookrightarrow \operatorname{Pic}^0( \mathcal{B} )$ via $q \mapsto \mathcal{O}_{\mathcal{B}}(q-q_0)$.
Then, there is a norm morphism $\operatorname{Nm}: \operatorname{Pic}^0( \mathcal{B} ) \to \mathcal{F}$ defined on the level of divisors
by applying $\phi^{\mathcal{B}}_{\mathcal{F}}$ to each point of a divisor class of degree zero.
The kernel of this map $\ker{(\operatorname{Nm})}$ is called the Prym variety $\operatorname{Prym}(\mathcal{B}/\mathcal{F})$.
The Prym variety $\operatorname{Prym}(\mathcal{B}/\mathcal{F})$ can also be identified 
with the quotient $\operatorname{Pic}^0( \mathcal{B} )/\phi^{\mathcal{B}\, *}_{\mathcal{F}}(\mathcal{F})$
where $\phi^{\mathcal{B}\, *}_{\mathcal{F}}(\mathcal{F})$ is the pullback of $\mathcal{F}\cong \operatorname{Pic}^0( \mathcal{F} )$ on the level of Jacobians
\cite[Lem.~6.2.1]{MR2888217orig}.

The \emph{duality theorem} \cite[Thm.~1.12]{MR946234} can then be stated as follows:
\begin{theorem}[duality theorem]
\label{thm:duality}
\begin{enumerate}
\item[]
\item For any smooth genus-three curve $\mathcal{B}$ admitting a bielliptic involution $\imath^{\mathcal{B}}$ 
covering a degree-two morphism  $\phi^{\mathcal{B}}_{\mathcal{F}}:  \mathcal{B} \to \mathcal{F}=\mathcal{F}=\mathcal{B}/\langle \imath^{\mathcal{B}} \rangle$,
it follows that $\mathcal{B}$ is embedded into $\operatorname{Prym}(\mathcal{B}/\mathcal{F})$ with self-intersection $4$ and 
$\operatorname{Prym}(\mathcal{B}/\mathcal{F})$ carries a natural $(1,2)$-polarization.

\item For any abelian surface $\mathbf{B}$ with a line bundle $\mathcal{L}$ of self-intersection $\mathcal{L}\cdot \mathcal{L}=4$
defining a $(1,2)$-polarization such that any element of $|\mathcal{L}|$ is irreducible, it follows that the general member $\mathcal{B} \in |\mathcal{L}|$
is a bielliptic curve of genus-three such that $\imath^{\mathcal{B}} = -\mathbb{I} \mid_{\mathcal{B}}$ and 
$\operatorname{Prym}(\mathcal{B}/\mathcal{F}) \cong \mathbf{B}$.
\end{enumerate}
\end{theorem}

\begin{remark}
One can generalize the notion of Prym variety to include nodes such that
the identification in Theorem~\ref{thm:duality} remains valid for the singular fibers \cite[Lem.~6.2.2]{MR2888217orig}.
\end{remark}

\subsection{Double covers of Kummers}
\label{ssec:doubles}
The Kummer surface $\operatorname{Kum}(\mathbf{A})$ of an abelian surface is closely related to another K3 surface, called K3 surface with Shioda-Inose structure.
A K3 surface $X$ has a Shioda-Inose structure if it admits a rational map  of degree two
 $\phi: X \dashrightarrow \operatorname{Kum}(\mathbf{A})$ 
to a Kummer surface which induces a Hodge isometry between the transcendental lattices $\operatorname{T}(X)(2)$ and 
$\operatorname{T}(\operatorname{Kum} \mathbf{A})$. 
The transcendental lattice is $\operatorname{T}(X) = \operatorname{NS}(X)^{\perp} \in H^2(X,\mathbb{Z})\cap H^{1,1}(X)$, and
the notation $\operatorname{T}(X)(2)$ indicates that the bilinear pairing on the transcendental lattice $\operatorname{T}(X)$ 
is multiplied by $2$. Morrison proved that a K3 surface $X$ admits a Shioda-Inose structure if and only if there exists a Hodge isometry
between the transcendental lattices  $\operatorname{T}(X) \cong \operatorname{T}(\mathbf{A})$. 

 In general, a rational map of degree two between a K3 surface and a Kummer surface, does not need to induce a Hodge isometry of their transcendental lattices.
  Mehran proved  in \cite{MR2306633}  that for a K3 surface $X$ whose transcendental lattice admits an embedding $\operatorname{T}(X) \hookrightarrow \operatorname{T}(\mathbf{A})$ into 
 the period lattice of an abelian surface $\mathbf{A}$ with the property $\operatorname{T}(\mathbf{A})/\operatorname{T}(X) \cong (\mathbb{Z}/2)^\alpha$ for $0 \le \alpha \le 4$, there exist a rational map of degree two $\phi: X \dashrightarrow \operatorname{Kum}(\mathbf{A})$. Mehran also proved that there exist only finitely many isomorphism classes of such K3 surfaces 
 admitting a rational map of degree two 
 onto $\operatorname{Kum}(\mathbf{A})$ by giving a complete list of all possible choices for the transcendental lattices $\operatorname{T}(X)$
 and the number of isomorphism classes for each choice
 (cf.~\cite[Cor.~3.3]{MR2306633}). 
 
 A subset of this list is comprised by K3 surfaces $X$ that admit an involution $\theta: X \to X$ such that the induced map on
 the N\'eron-Severi lattice $\operatorname{NS}(X)$ and the transcendental lattice $\operatorname{T}(X)$ acts by $\mathbb{I}$ and $-\mathbb{I}$, respectively.
Mehran proved in \cite[Cor.~3.3]{MR2306633} that the transcendental lattice $\operatorname{T}(X)$ of the K3 surface $X$ with Picard number seventeen 
has to be one of the following three cases: 
 \begin{enumerate*}
 \item the transcendental lattice is $H^2 \oplus \langle -2 \rangle$ in which case there is only one isomorphism class; this case recovers the case of K3 surfaces with a Shioda-Inose structure by Morrison's criterion;
\item the transcendental lattice is $H(2) \oplus H \oplus \langle -2 \rangle$ in which case there are 
15 isomorphism classes;
\item the transcendental lattice is  $H(2)^2\oplus \langle -2 \rangle$ in which case there are 15 
isomorphism classes. 
\end{enumerate*}
Here,  $H$ is the hyperbolic rank-two lattice $\mathbb{Z}^2$ with the quadratic form $2 xy$.
 This article is dedicated to a comprehensive investigation of the third case.

\subsection{Even eights on K3 and Kummer surfaces}
On an arbitrary K3 surface $Y$, an even eight is a set of eight disjoint smooth rational curves $\Delta=\lbrace {C}_1, \dots ,{C}_8\rbrace$ 
such that ${C}_1 + \dots + {C}_8 \in 2 \, \operatorname{NS}(Y)$.
For any such even eight, there is a double cover $\phi: Z \dashrightarrow Y$ branched on ${C}_1 + \dots + {C}_8$. 
If ${E}_i$ denotes the inverse image of ${C}_i$ for $1\le i \le 8$, then $\phi^*({C}_i) = 2 \, {E}_i$ and ${E}_i^2 =-1$. 
Hence, one can blow down the ${E}_i$'s to obtain a new surface $X=X(Y,\Delta)$. It turns out that this surface $X$ is again a K3 surface and the covering involution is symplectic and fixes the eight special points obtained in the blow-down, also called a Nikulin involution.

On the other hand, on any Kummer surface there are sixteen smooth disjoint rational curves -- obtained from the resolution of the ordinary double points --
that we can select even eights from. It follows from the work by Nikulin that there exist exactly 30 even eights made up from these exceptional divisors. Mehran~\cite{MR2804549} proved that by applying 
the above construction to any of these even eights on $Y=\operatorname{Kum}(\mathbf{A})$, one obtains again Kummer surfaces: in fact, if $\phi:X \dashrightarrow \operatorname{Kum}(\mathbf{A})$ is 
the double cover associated to an even eight  $\Delta$ of nodes on $Y=\operatorname{Kum}(\mathbf{A})$, then the K3 surface $X$  will contain again sixteen disjoint smooth rational curves. The reason is that  the eight curves complimentary to the even eight on a Kummer surface do not intersect the branch locus of the  double cover $\phi: Z \dashrightarrow Y$. Thus, they split under the map $\phi$ and define sixteen disjoint smooth rational curves on $Z$ that are mapped by the blow-down to sixteen curves on $X$. Thus, $X$ contains sixteen disjoint smooth rational curves and hence is a Kummer surface
$X=\operatorname{Kum}(\mathbf{B})$ for some abelian surface $\mathbf{B}$. Two Kummer surfaces $X(Y,\Delta_1)$ and  $X(Y,\Delta_2)$ constructed from two even eights $\Delta_1$ 
and $\Delta_2$ are isomorphic if and only if the Nikulin lattices generated by $\Delta_1$ and  $\Delta_2$ are related by an automorphism on $Y$. 
Moreover, one obtains the exact same Kummer surface  whether one takes the double cover branched along an even eight or its complement. 

For a Kummer surface $Y=\operatorname{Kum}(\operatorname{Jac} C)$ of the Jacobian of a generic curve $\mathcal{C}$ 
of genus two, Mehran labels the fifteen different even eights (up to taking complements) by $\Delta_{ij}$ with $1 \le i < j \le 6$, and proves in \cite[Prop.~4.2]{MR2804549} 
that Kummer surfaces $X(Y,\Delta_{ij})$ and $X(Y,\Delta_{i'j'})$ are isomorphic if and only if $i=i'$ and $j=j'$. 
Moreover, she proved that there is an abelian surface $\mathbf{B}_{ij}$ associated with $X(Y,\Delta_{ij})$  such that 
$X=\operatorname{Kum}(\mathbf{B}_{ij})$ and the rational map $\phi_{ij}: X(Y,\Delta_{ij}) \dashrightarrow \operatorname{Kum}(\mathbf{A})$ is induced by an isogeny $\varphi_{ij}: \mathbf{B}_{ij} \to \mathbf{A}$ 
of abelian surfaces of degree two~\cite{MR2804549}.  The 
Kummer surfaces $\operatorname{Kum}(\mathbf{B}_{ij})=X(\operatorname{Kum}(\operatorname{Jac} \mathcal{C}),\Delta_{ij})$ realize the 15 isomorphism classes of 
rational degree-two covers of $\operatorname{Kum}(\operatorname{Jac} \mathcal{C})$ with transcendental lattice $\operatorname{T}(X) \cong H(2)\oplus H(2) \oplus \langle -2 \rangle$ from Mehran's classification list mentioned above. 

\subsection{Elliptic fibrations on Kummer surfaces}
Any K3 surface with Picard number $X$ bigger than 4 is guaranteed to admit an elliptic fibration, i.e., a regular map into $\mathbb{P}^1$
where the general fiber is a smooth connected curve of genus one. Moreover, Sterk showed in \cite{MR786280} that any K3 surface $X$ does only have finitely many elliptic fibrations up to $\operatorname{Aut}(X)$. 

On the Kummer surface $\operatorname{Kum}(\mathcal{E}_1 \times \mathcal{E}_2)$ for the product  of two non-isogenous generic complex 
elliptic curves, the associated Kummer surface has Picard number eighteen, or even nineteen if the two elliptic curves are 
mutually isogenous. Oguiso classified all eleven inequivalent elliptic fibrations on the Kummer surface for the product of two non-isogenous 
generic elliptic curves in \cite{MR1013073}. 
The simplest of such elliptic  fibrations are the ones induced by the projection of $\mathcal{E}_1 \times \mathcal{E}_2$ onto its first or second factor and are called the 
first or second Kummer pencil.  Kuwata and Shioda determined in \cite{MR2409557} explicitly elliptic parameters and Weierstrass equations for all eleven different fibrations that appear.

On the Kummer surface $\operatorname{Kum}(\operatorname{Jac} \mathcal{C})$ for a generic curve  $\mathcal{C}$ of genus two, there are always two sets of 
sixteen $(-2)$-curves, called nodes and tropes, which are either the exceptional divisors corresponding to blow-up of the 16 two-torsion points or they arise from the
embedding of the polarization divisor as symmetric theta divisors. These two sets of smooth rational curves have a rich symmetry, the so-called $16_6$-configuration 
where each node intersects exactly six tropes and vice versa \cite{MR1097176}.  Using curves in the $16_6$-configuration, 
one can find all elliptic fibrations since all irreducible components of a reducible fiber in an elliptic fibration are $(-2)$-curves \cite{MR0184257} and
the rank of the N\'eron-Severi group is seventeen. All inequivalent elliptic fibrations where determined explicitly by Kumar in \cite{MR3263663}. In particular, Kumar computed elliptic parameters and 
Weierstrass equations for all twenty five different fibrations that appear, and analyzed the reducible fibers and  Mordell-Weil lattices.

The Shioda-Inose partner of $\operatorname{Kum}(\operatorname{Jac} C)$ admits exactly two elliptic fibrations realizing the two inequivalent orthogonal 
complements of $H \oplus H \oplus \langle -2 \rangle$ in the K3 lattice $H^3 \oplus E_8(-1)^2$. Here, $E_8(-1)$ is the negative definite lattice associated with the exceptional root systems of $E_8$.
These two elliptic fibrations were described in \cite{MR2427457, MR2935386} and applied to F-theory and gauge theory in \cite{MR3366121,MR2854198,Malmendier:2016aa}. 
Mehran proved in \cite{MR2804549} that the Kummer surface $\operatorname{Kum}(\mathbf{B}_{12})$ admits an elliptic fibration with exactly twelve singular fibers 
of Kodaira-type $I_2$. The existence of this particular Jacobian elliptic fibration will play a key role in the proof of Theorem~\ref{cor:pencil}.

  \subsection{Jacobian elliptic fibrations}
A surface is called a Jacobian elliptic fibration if it is a (relatively) minimal elliptic surface $\pi: \mathcal{X} \to \mathbb{P}^1$ over $\mathbb{P}^1$
with a distinguished section $\mathsf{S}_0$. The complete list of possible singular fibers has been determined by Kodaira~\cite{MR0184257}. 
It encompasses two infinite families $(I_n, I_n^*, n \ge0)$ and six exceptional cases $(II, III, IV, II^*, III^*, IV^*)$.
To each Jacobian elliptic fibration $\pi: \mathcal{X} \to \mathbb{P}^1$ there is an associated Weierstrass model $\bar{\mathcal{X}\,}\to \mathbb{P}^1$ 
obtained by contracting all components of fibers not meeting $\mathsf{S}_0$. 
The fibers of $\bar{\mathcal{X}\,}$ are all irreducible, the singularities are all rational double points, and $\mathcal{X}$ is the minimal desingularization.
If we choose $t \in \mathbb{C}$ as a local affine coordinate on $\mathbb{P}^1$, the Weierstrass normal form is given by
\begin{equation}
\label{Eq:Weierstrass}
 Y^2 = 4 \, X^3 - g_2(t) \, X - g_3(t) \;,
\end{equation}
where $g_2$ and $g_3$ are polynomials in $t$ of degree four and six, respectively, because $\mathcal{X}$ is a K3 surface. It is of course well known
how the type of singular fibers is read off from the orders of vanishing of the functions $g_2$, $g_3$ and the discriminant $\Delta= g_2^3 - 27 \, g_3^2$ 
at the various singular base values. Note that the vanishing degrees of $g_2$ and $g_3$ are always less or equal three and five, respectively,
as otherwise the singularity of $\bar{\mathcal{X}\,}$ is not a rational double point.

For a family of Jacobian elliptic surfaces $\pi: \mathcal{X} \to \mathbb{P}^1$, the two classes in N\'eron-Severi lattice $\operatorname{NS}(\mathcal{X})$ associated 
with the elliptic fiber and section  span a sub-lattice $\mathcal{H}$ isometric to the standard hyperbolic lattice $H$, and we have the following decomposition 
as a direct orthogonal sum
\begin{equation*}
 \operatorname{NS}(\mathcal{X}) = \mathcal{H} \oplus \mathcal{W} \;.
\end{equation*}
There is a sub-lattice $\mathcal{W}^{\mathrm{root}} \subset \mathcal{W}$ spanned by the roots, i.e., the algebraic classes of
self-intersection $-2$ inside $\mathcal{W}$. The singular fibers of the Weierstrass model $\bar{\mathcal{X}\,} \to \mathbb{P}^1$ 
determine $\mathcal{W}^{\mathrm{root}}$ uniquely up to permutation. Moreover, by Nishiyama's lemma 
there exists a canonical group isomorphism identifying $\mathcal{W}/\mathcal{W}^{\mathrm{root}}$ with the Mordell-Weil 
group of sections, i.e.,
\begin{equation}
 \mathcal{W}/\mathcal{W}^{\mathrm{root}} \stackrel{\cong} \longrightarrow \operatorname{MW}({\pi}, \mathsf{S}_0) \;.
\end{equation}

 An important tool in the investigation of fibrations is the so-called 
 discriminant form $(D_\mathrm{L},q_\mathrm{L})$ of the transcendental lattice $\mathrm{L}$.   
For a non-degenerate even lattice $\mathrm{L}$ over $\mathbb{Z}$ with a quadratic form $Q$ and dual lattice $\mathrm{L}^\vee$, 
the discriminant group is defined by $D_\mathrm{L}=\mathrm{L}^\vee/\mathrm{L}$. The natural projection is denoted by 
$\Psi: \mathrm{L}^{\vee} \to D_\mathrm{L}$. A non-degenerate quadratic form $q_{\mathrm{L}}$ on $D_\mathrm{L}$ with 
values in $\mathbb{Z}/2 \mod{2}$ is given by $q_{\mathrm{L}}(x) = Q(x') \!\!\mod{2}$ where $x' \in \mathrm{L}^\vee$ such that $\Psi(x')=x$.
There is a bijection between the set of isotopic subgroups of the discriminant group and the set
of over-lattices of $\mathrm{L}$, i.e., the integral sub-lattices of $\mathrm{L}^\vee$ containing $\mathrm{L}$:
\begin{theorem}[Nikulin \cite{MR525944}]
\label{thm:Nikulin}
If $\operatorname{V} \subset D_\mathrm{L}$ is a subgroup isotopic with respect
to $q_{\mathrm{L}}$, then $\mathrm{M}=\Psi^{-1}(\operatorname{V})$ is an even over-lattice of $\mathrm{L}$, and the discriminant form
of $\mathrm{M}$ is isomorphic to
\begin{equation}
  \Big( D_\mathrm{M}, q_{\mathrm{M}} \Big) =\Big( D_\mathrm{L}, q_{\mathrm{L}} \Big)\Big|_{\operatorname{V}^\perp/\operatorname{V}} \;.
\end{equation}
The correspondence $\operatorname{V} \to \mathrm{M}$ is a bijection.
\end{theorem}
 It follows that the discriminant group and discriminant form of the transcendental lattice $\operatorname{T}(\mathcal{X})$
 can be computed from the lattice of roots using
\begin{equation}
\label{eqn:transL}
 \Big( D_{ \operatorname{T}(\mathcal{X})}, q_{\, \operatorname{T}(\mathcal{X})} \Big) \cong \Big( D_\mathcal{W}, -q_{\mathcal{W}} \Big) \;.
\end{equation}
The discriminant group of the transcendental lattice of a K3 surface is an essential part in Mehran's proof:
 It follows from the work of Nikulin \cite[Prop.~4.1]{MR2306633} that
 an involution $\theta: X \to X$ -- such that the induced map on the N\'eron-Severi lattice $\operatorname{NS}(X)$ 
 and the transcendental lattice $\operatorname{T}(X)$ acts by $\mathbb{I}$ and $-\mathbb{I}$, respectively -- 
 exists if and only if the discriminant group $D_{\operatorname{T}(X)}$ of the transcendental lattice is two-elementary, i.e., isomorphic 
 to $(\mathbb{Z}/2)^\alpha$ for some $\alpha$. For a K3 surface with Shioda-Inose structure we have
$D_{\operatorname{T}(X)}\cong \mathbb{Z}/2$, whereas for the transcendental lattice $\operatorname{T}(X)=H(2)^2\oplus \langle -2 \rangle$
we find the discriminant group $D_{\operatorname{T}(X)}\cong (\mathbb{Z}/2)^5$. 
   
 \section{Genus-two curves, Jacobians, and two-isogenies}
\label{genus-two}

\subsection{The Siegel three-fold and theta-functions}
In this section we give a brief review of genus-two theta-function as they relate to our work. Most of this material
can be found in any standard reference on theta-functions, in particular \cite{MR0141643, MR0168805}.

The Siegel three-fold is a quasi-projective variety of dimension $3$ obtained from the Siegel upper 
half-plane of degree two which by definition is the set of two-by-two symmetric matrices over $\mathbb{C}$ whose imaginary part is positive definite, i.e.,
\begingroup\makeatletter\def\f@size{10}
 \begin{align}
\label{Siegel_tau}
 \mathbb{H}_2 = \left. \left\lbrace \underline{\tau} = \left( \begin{array}{cc} \tau_{11} & \tau_{12}  \\ \tau_{12}  & \tau_{22} \end{array} \right) \right|
 \tau_{11}, \tau_{22}, \tau_{12} \in \mathbb{C},\, \operatorname{Im}{(\tau_{11})} \,  \operatorname{Im}{(\tau_{22}}) > \operatorname{Im}{(\tau_{21})}^2, \, \operatorname{Im}{(\tau_{22})} > 0 \right\rbrace ,
 \end{align}
\endgroup
quotiented out by the action of the modular transformations $\Gamma_2:=
\operatorname{Sp}_4(\mathbb{Z})$, i.e., 
\begin{equation}
 \mathcal{A}_2 =  \mathbb{H}_2 / \Gamma_2 \;.
\end{equation} 
Each $\underline{\tau} \in \mathbb{H}_2$ determines a principally polarized complex abelian surface $\mathbf{A}_{\underline{\,\tau}} = \mathbb{C}^2 / \langle \mathbb{Z}^2 \oplus \underline{\tau} \,
\mathbb{Z}^2\rangle$ with period matrix $(\mathbb{I}_2,\underline{\tau}) \in \operatorname{Mat}(2, 4;\mathbb{C})$.
The principal polarization  $\mathcal{L}$ associated to  $\underline{\tau}$ is given by the Riemann form 
$E( x_1 +  \underline{\tau}x_2, y_1  +  \underline{\tau}y_2)=x_1^t\cdot y_2 - y_1^t\cdot x_2$ on  $\mathbb{Z}^2 \oplus \underline{\tau} \, \mathbb{Z}^2$.
Two abelian surfaces $\mathbf{A}_{\underline{\,\tau}}$ 
and $\mathbf{A}_{\underline{\,\tau}'}$ are isomorphic if and only if there is a symplectic matrix 
\begin{equation}
M= \left(\begin{array}{cc} A & B \\ C & D \end{array} \right) \in \Gamma_2
\end{equation}
such that $\underline{\tau}' = M (\underline{\tau}):=(A\underline{\tau}+B)(C\underline{\tau}+D)^{-1}$.
Since $M$ preserves the Riemann form, it follows that the Siegel three-fold $\mathcal{A}_2$ is the set of isomorphism classes of principally polarized abelian surfaces.

For a elliptic variable $\underline{z} \in \mathbb{C}^2$ and modular variable $\underline{\tau}\in \mathbb{H}_2$, Riemann's theta-function is defined by setting
\[ \theta \big( \underline{z} , \underline{\tau} \big) = \sum_{\underline{u}\in \mathbb{Z}^2} e^{\pi i\, ( \underline{u}^t \cdot \underline{\tau} \cdot  \underline{u} + 2 \underline{u}^t \cdot \underline{z} )  } . \]
The theta-function is holomorphic on $\mathbb{C}^2\times
\mathbb{H}_2$. For $\underline{a}, \underline{b} \in \mathbb{Q}^2$, the theta-functions with rational characteristics are defined by setting
\[ 
\theta   \begin{bmatrix} \underline{a}^t \\ \underline{b}^t \end{bmatrix}( \underline{z} , \underline{\tau}) 
=  \sum_{\underline{u}\in \mathbb{Z}^2} e^{\pi i \, \big((\underline{u}+\underline{a})^t \cdot \underline{\tau} 
\cdot  (\underline{u}+\underline{a}) + 2 \, (\underline{u}+\underline{a})^t \cdot (\underline{z}+\underline{b})\big)  }. \]
If the coefficients of the vectors $\underline{a}$ and $\underline{b}$ are in $\{ 0,\frac{1}{2}\}$, the
characteristics are called half-integer characteristics. There are 16 half-integer characteristics, 10 are even and 6 are odd
according to whether the associated theta-characteristics is even or odd, i.e.,
\begin{equation}
\theta   \begin{bmatrix} \underline{a}^t \\ \underline{b}^t \end{bmatrix}( -\underline{z} , \underline{\tau})   
= (-1)^{4 \, \underline{a}^t\cdot \underline{b}} \; \theta   \begin{bmatrix} \underline{a}^t \\ \underline{b}^t \end{bmatrix}(\underline{z} , \underline{\tau}) \;.
\end{equation}
A scalar obtained by evaluating a theta-characteristic at $\underline{z}=0$ is called a theta-constant. 
We denote the even half-integer characteristics by 
\begin{small}
\[
\begin{split}
\theta_1 = \theta\chr {0}{0}{0}{0} ,  \,  
\theta_2 =  \theta\chr {0}{0}{\frac{1}{2}} {\frac{1}{2}} ,  \,
\theta_3 &=  \theta\chr {0}{0}{\frac{1}{2}}{0} , \, 
\theta_4 =  \theta\chr {0}{0}{0}{\frac{1}{2}} , \,   \;
\theta_5 =  \theta\chr{\frac{1}{2}}{0}{0}{0} ,\\
  \theta_6  =  \theta\chr {\frac{1}{2}}{0}{0}{\frac{1}{2}} , \, 
  \theta_7 =  \theta\chr{0}{\frac{1}{2}} {0}{0} , \, 
  \theta_8 &=  \theta\chr{\frac{1}{2}}{\frac{1}{2}} {0}{0} , \, 
  \theta_9 =  \theta\chr{0}{\frac{1}{2}} {\frac{1}{2}}{0} , \, 
  \theta_{10} =  \theta\chr{\frac{1}{2}}{\frac{1}{2}} {\frac{1}{2}}{\frac{1}{2}} ,
\end{split}
\]
\end{small}
and write
\begin{equation}
\label{Eqn:theta_short}
 \theta_i(\underline{z}) \quad \text{instead of} \quad 
 \theta   \begin{bmatrix} \underline{a}^{(i)\, t} \\ \underline{b}^{(i)\, t} \end{bmatrix}  (\underline{z} , \underline{\tau}) 
  \quad \text{where $i=1,\dots ,10$,}
\end{equation}
and $\theta_i =\theta_i(0)$. Under duplication of the modular variable by isogeny,
we want the theta-functions $\theta_1, \theta_5, \theta_7, \theta_8$ to play a role dual to $\theta_1, \theta_2, \theta_3, \theta_4$.
We renumber the former by interchanging the roles of $\underline{a}^{(i)}$ and $\underline{b}^{(i)}$ in
Equation~(\ref{Eqn:theta_short}) and use the symbol $\Theta$ to mark the fact that they are evaluated at the isogenous abelian variety.
Therefore, we will denote the theta-characteristics with doubled modular variable by 
\begin{equation}
\label{Eqn:Theta_short}
 \Theta_i(\underline{z})  \quad \text{instead of} \quad \theta   \begin{bmatrix} \underline{b}^{(i)\, t} \\ \underline{a}^{(i)\, t} \end{bmatrix} (\underline{z} , \, 2 \underline{\tau}) \quad \text{where $i=1,\dots ,10$.}
\end{equation}

\subsection{Abelian surfaces from genus-two curves}
Let $\mathcal{C}$ be an irreducible, smooth, projective curve of genus two, defined over $\mathbb{C}$, 
and let $\mathcal{M}_2$, be the coarse moduli space of smooth curves of genus two. We denote by $[\mathcal{C}]$ the isomorphism class of $\mathcal{C}$, i.e., 
the corresponding point in $\mathcal{M}_2$. For a genus-two curve $\mathcal{C}$ with $[\mathcal{C}] \in \mathcal{M}_2$ given as sextic $Y^2 = f_6(X,Z)$
in weighted projective space $\mathbb{WP}(1,3,1)$, we send three roots $\lambda_4, \lambda_5, \lambda_6$ to $1, 0, \infty$ to get an isomorphic curve
in Rosenhain normal form, i.e., 
\begin{equation}
\label{Eq:Rosenhain}
 \mathcal{C}: \quad Y^2 = X\,Z \, \big(X-Z\big) \, \big( X- \lambda_1 Z\big) \,  \big( X- \lambda_2 Z\big) \,  \big( X- \lambda_3 Z\big) \;.
\end{equation} 
The ordered tuple $(\lambda_1, \lambda_2, \lambda_3)$ where the $\lambda_i$ are all distinct and different from $0, 1, \infty$
determines a point in $\mathcal{M}_2(2)$, the moduli space of genus-two curves together with a level-two structure.
As functions on $\mathcal{M}_2(2)$, the Rosenhain invariants generate its coordinate ring $\mathbb{C}(\lambda_1, \lambda_2, \lambda_3)$.

 Torelli's theorem states that the map sending a curve $\mathcal{C}$ to its Jacobian variety 
$\operatorname{Jac}(\mathcal{C}) = \operatorname{Pic}^0(\mathcal{C})$, i.e.,  the connected component of 
the identity in the Picard group of $\mathcal{C}$, is injective and defines a birational map $\mathcal{M}_2 \dashrightarrow \mathcal{A}_2$.
Moreover, the three $\lambda$-parameters in the Rosenhain normal form of a genus-two curve $\mathcal{C}$ in Equation~(\ref{Eq:Rosenhain}) 
determine a level-two structure on the corresponding Jacobian variety and generate the function field of $\mathcal{A}_2(2)$, 
that is the three-dimensional moduli space of principally polarized abelian surfaces with level-two structure.

The three $\lambda$-parameters in the Rosenhain normal form of a genus-two curve $\mathcal{C}$ in Equation~(\ref{Eq:Rosenhain}) 
can be expressed as ratios of even theta-constants. There are 720 choices for such expressions since
the forgetful map $\mathcal{M}_2(2) \to \mathcal{M}_2$ is a Galois covering of degree $720 = |S_6|$ where $S_6$ acts on the 
roots of $\mathcal{C}$ by permutations. Any of the 720 choices may be used, we picked the one from \cite{MR2367218}.
We therefore have the following lemma:
\begin{lemma} \label{ThomaeLemma}
Given a genus-two curve $\mathcal{C}$,
the Rosenhain roots $\lambda_1, \lambda_2, \lambda_3$  can be written as follows:
\begin{equation}\label{Picard}
\lambda_1 = \frac{\theta_1^2\theta_3^2}{\theta_2^2\theta_4^2} \,, \quad \lambda_2 = \frac{\theta_3^2\theta_8^2}{\theta_4^2\theta_{10}^2}\,, \quad \lambda_3 =
\frac{\theta_1^2\theta_8^2}{\theta_2^2\theta_{10}^2}\,.
\end{equation}
Conversely, given three distinct complex numbers $(\lambda_1, \lambda_2, \lambda_3)$ different from $0, 1, \infty$ there is 
complex abelian surface $\mathbf{A}_{\underline{\,\tau}}$ with period matrix $(\mathbb{I}_2,\underline{\tau})$ and $\underline{\tau} \in\mathcal{A}_2$
such that $\mathbf{A}_{\underline{\,\tau}}=\operatorname{Jac} \, \mathcal{C}$.
\end{lemma}
We have the following lemma:
\begin{lemma}
\label{rem:l_var}
The Rosenhain roots $\lambda_1, \lambda_2, \lambda_3$ are expressible as rational functions in the variables 
$\lbrace \Theta_1^2 , \Theta_2^2, \Theta_3^2, \Theta_4^2\rbrace$, and the rational function
\begin{equation}
\label{def_l}
 l = \dfrac{(\Theta_1 \Theta_2 - \Theta_3 \Theta_4) (\Theta_1^2 + \Theta_2^2 + \Theta_3^2 + \Theta_4^2)(\Theta_1^2 - \Theta_2^2 - \Theta_3^2 + \Theta_4^2)}
 {(\Theta_1 \Theta_2 + \Theta_2 \Theta_4) (\Theta_1^2 - \Theta_2^2 + \Theta_3^2 - \Theta_4^2)(\Theta_1^2 + \Theta_2^2 - \Theta_3^2 - \Theta_4^2)}
\end{equation}
satisfies $l^2= \lambda_1 \lambda_2 \lambda_3$.
\end{lemma}
\begin{proof}By direct computation.
\end{proof}

\subsection{Two isogenies of abelian surfaces}
\label{ssec:2isog}
For a principally polarized abelian surface $\mathbf{A}$, the polarization of $\mathbf{A}$ induces a symplectic form on $\mathbf{A}[2]$, 
the space of two-torsion points of $\mathbf{A}$. Therefore,  $\mathbf{A}[2]\cong ( \mathbb{Z}/2)^4$ is a symplectic vector space of 
dimension four over the finite field $\mathbb{F}_2$. For any isotropic two-dimensional subspace $\mathsf{k}$ of $\mathbf{A}[2]$, that is, the symplectic form vanishes on $\mathsf{k}$, 
it follows from CM-theory that the quotient $\hat{\mathbf{A}}=\mathbf{A}/\mathsf{k}$ is again a principally polarized abelian surface.  Moreover, a maximal 
isotropic subgroup $\mathsf{k}$ of the two-torsion is always 
isomorphic to $( \mathbb{Z}/2)^2$. One defines a $(2,2)$-isogeny to be an 
isogeny\footnote{An isogeny is a surjective homomorphism with finite kernel.} $\psi_{\mathsf{k}}: \mathbf{A} \to \hat{\mathbf{A}}$ between principally
polarized abelian surfaces such that the kernel is a two-dimensional isotropic subspace $\mathsf{k}$ of $\mathbf{A}[2]$. 
Concretely, given any choice of such isotropic subspace $\mathsf{k}$ the $(2,2)$-isogeny between complex tori is given by
\begin{equation}
\begin{split}
\psi_{\mathsf{k}}: \; \mathbf{A} = \mathbb{C}^2 / \langle \mathbb{Z}^2 \oplus \underline{\tau} \,\mathbb{Z}^2\rangle & \to  \hat{\mathbf{A}} =\mathbb{C}^2 / \langle \mathbb{Z}^2 \oplus 2\underline{\tau} \,\mathbb{Z}^2\rangle \\
\underline{z} & \mapsto 2 \, \underline{z} \;.
\end{split}
\end{equation}
  The group of linear transformations of $\mathbf{A}[2]$ preserving the symplectic form can be identified 
with the permutation group of the set of 6 Weierstrass points on the curve $\mathcal{C}$. This classical way to describe the 15 inequivalent $(2,2)$-isogenies 
on the Jacobian $\mathbf{A}=\operatorname{Jac}(\mathcal{C})$  of a generic curve  $\mathcal{C}$ of genus-two is called \emph{Richelot isogeny}.   
If we choose for $\mathcal{C}$ a sextic equation $Y^2 = f_6(X,Z)$, then any factorization $f_6 = A\cdot B\cdot C$ into 
three degree-two polynomials $A, B, C$ defines a genus-two 
curve $\hat{\mathcal{C}\,}\!$ given by
\begin{equation}
\label{Richelot}
 \Delta \cdot Y^2 = [A,B] \, [A,C] \, [A,C] 
\end{equation}
where we have set $[A,B] = A'B - AB'$ with $A'$ denoting the derivative of $A$ with respect to $X$, and $\Delta$ 
is the determinant of $A, B, C$ with respect to the basis $X^2, XZ, Z^2$.
It was proved in~\cite{MR970659} that $\operatorname{Jac}(\mathcal{C})$ and $\operatorname{Jac}(\hat{\mathcal{C}\,}\!)$ are $(2,2)$-isogenous, and that there are exactly 
$15$ different curves $\hat{\mathcal{C}\,}\!$ that are obtained this way. It follows that this construction yields all principally polarized abelian surfaces
$(2,2)$-isogenous to $\mathbf{A} = \operatorname{Jac}(\mathcal{C})$ as Jacobians
$\hat{\mathbf{A}} = \operatorname{Jac} \hat{\mathcal{C}\,}\!$. 

\subsubsection{$16_6$-configuration}
 On $\mathbf{A}=\operatorname{Jac}(\mathcal{C})$ the 16 theta-divisors together
with the 16 order-two points in $\mathbf{A}[2]$ form a $16_6$-configuration.
For a generic genus-two curve  $\mathcal{C}$ given by Equation~(\ref{Eq:Rosenhain}), we denote the Weierstrass points by $\mathfrak{p}_i: [X:Y:Z]=[1:0:\lambda_i]$ for $1 \le i \le 5$ and $\mathfrak{p}_6:  [X:Y:Z]=[0:0:1]$. The 16 order-two points of $\mathbf{A}[2]$ are obtained using the embedding of the curve into  the connected component of the identity in the Picard group, i.e., $\mathcal{C} \hookrightarrow \mathbf{A} \cong \operatorname{Pic}^0(\mathcal{C})$ with $\mathfrak{p} \mapsto [\mathfrak{p} -\mathfrak{p}_6]$. In particular, we obtain all elements of $\mathbf{A}[2]$ as
 \begin{equation}
 \label{oder2points}
  \mathfrak{p}_{i6} = [ \mathfrak{p}_i - \mathfrak{p}_6] \; \text{for $1 \le i \le 6$}, \qquad 
  \mathfrak{p}_{ij}=[ \mathfrak{p}_i + \mathfrak{p}_j - 2 \, \mathfrak{p}_6]  \; \text{for $1 \le i < j \le 5$}, 
 \end{equation}
and we set $\mathfrak{p}_{0} =\mathfrak{p}_{66}$.
For $\{i,j,k,l,m,n\}=\{1, \dots6\}$  the group law on $\mathbf{A}[2]$ is given by the relations
 \begin{equation}
 \label{group_law}
    \mathfrak{p}_0 +  \mathfrak{p}_{ij} =  \mathfrak{p}_{ij}, \quad  \mathfrak{p}_{ij} +  \mathfrak{p}_{ij} =  \mathfrak{p}_{0}, \quad 
    \mathfrak{p}_{ij} +  \mathfrak{p}_{kl} =  \mathfrak{p}_{mn}, \quad  \mathfrak{p}_{ij} +
      \mathfrak{p}_{jk} =  \mathfrak{p}_{ik}.
 \end{equation}
 Moreover, we will denote the theta-divisors on $\mathbf{A}$ by $\mathfrak{t}$ to avoid any confusion with the theta-functions.
 The standard theta-divisor $\mathfrak{t} = \mathfrak{t}_6 = \{ [\mathfrak{p} - \mathfrak{p}_0] \, \mid \mathfrak{p} \in \mathcal{C}\}$
 contains the six order-two points $\mathfrak{p}_0, \mathfrak{p}_{i6}$ for $1 \le i \le 5$. 
 Likewise for $1 \le i \le 5$, the five translates $\mathfrak{t}_i = \mathfrak{t} + \mathfrak{p}_{i6}$ contain $\mathfrak{p}_0, \mathfrak{p}_{i6}, \mathfrak{p}_{ij}$ with $j \not = i, 6$,
 and the ten translates $\mathfrak{t}_{ij6} = \mathfrak{t}+ \mathfrak{p}_{ij}$ for $1 \le i < j \le 5$
 contain $\mathfrak{p}_{ij}, \mathfrak{p}_{i6}, \mathfrak{p}_{j6}, \mathfrak{p}_{kl}$ with $k,l \not = i,j,6$ and $k<l$.
 Conversely, each point lies on exactly six of the divisors, namely
 \begin{align}
   \mathfrak{p}_{0} & \in \mathfrak{t}_i \phantom{ , \mathfrak{t}_6, \mathfrak{t}_{ij6} } \quad  \text{for $i=1, \dots, 6$,}\\
   \mathfrak{p}_{i6} & \in \mathfrak{t}_i, \mathfrak{t}_6, \mathfrak{t}_{ij6} \quad  \text{for $i=1, \dots, 5$ with $j \not = i,6$,}\\
   \mathfrak{p}_{ij} & \in \mathfrak{t}_i, \mathfrak{t}_j, \mathfrak{t}_{kl6} \quad  \text{for $1 \le i < j \le 5$ with $k,l \not = i,j,6$ and $k<l$.}
 \end{align}
 The automorphism group of the $16_6$-configuration on $\mathbf{A}$ is $\mathbb{F}_2^4 \rtimes \operatorname{Sp}(4, \mathbb{F}_2)$, i.e., given by
 translations by order-two points and rotations preserving the symplectic form on $\mathbf{A}[2]$ because they induce an action of 
 the permutation group of the set of 6 theta-divisors containing a fixed two-torsion point. 
 
\subsubsection{Computation of dual $(2,2)$-isogenies}
\label{computation}
We will label the 15 inequivalent maximal isotropic subgroups of $\mathbf{A}[2]\cong ( \mathbb{Z}/2)^4$ and quotients by $\mathsf{k}_{ij}$ and 
$\hat{\mathbf{A}}_{ij}=\mathbf{A}/\mathsf{k}_{ij}$ with $1 \le i < j \le 6$. 
Similarly, we will denote the Richelot isogenous sextic curves by $\hat{\mathcal{C}\,}\!_{ij}$ such that $\hat{\mathbf{A}}_{ij} =\operatorname{Jac}(\hat{\mathcal{C}\,}\!_{ij})$.
 We have the following lemma:
 \begin{lemma}
 \label{lem:decomp}
 Let $\mathsf{k}$ and $\!\hat{\,\mathsf{k}}$ be two maximal isotropic subgroup of $\mathbf{A}[2]$ such that 
 $\!\hat{\,\mathsf{k}}+\mathsf{k}=\mathbf{A}[2]$, $\!\hat{\,\mathsf{k}} \cap \mathsf{k}=\{ \mathfrak{p}_0\}$.
 Set $\hat{\mathbf{A}}=\mathbf{A}/\mathsf{k}$, and denote the image of $\!\hat{\,\mathsf{k}}$ in $\hat{\mathbf{A}}$
 by $\mathsf{K}$. Then it follows that $\hat{\mathbf{A}} /\mathsf{K} \cong \mathbf{A}$,
 and the composition of $(2,2)$-isogenies $\hat{\psi}_\mathsf{K}\circ\psi_{\mathsf{k}}$ is multiplication by two on $\mathbf{A}$, i.e.,
$(\underline{z}, \underline{\tau}) \mapsto (2  \underline{z}, \underline{\tau})$. 
  \end{lemma}
 \begin{proof}
 By construction $\mathsf{k}$ is a finite subgroup of $\mathbf{A}$, $\hat{\mathbf{A}}=\mathbf{A}/\mathsf{k}$ a complex torus,
 and the natural projection $\psi: \mathbf{A} \to \hat{\mathbf{A}}\cong \mathbf{A}/\mathsf{k}$ and isogeny. The order of the kernel is two, hence it is a degree-four  isogeny.
 The same applies to the map $\hat{\psi}: \hat{\mathbf{A}} \to \hat{\mathbf{A}}/\mathsf{K}$.
 Therefore, the composition $\hat{\psi}\circ \psi$ is an isogeny with kernel $\!\hat{\,\mathsf{k}}+\mathsf{k}=\mathbf{A}[2]$.
 Thus, $\hat{\mathbf{A}}/\mathsf{K}\cong \mathbf{A}$ and the map $\hat{\psi}\circ \psi$
 is the group homomorphism $\underline{z} \mapsto 2 \underline{z}$ whose kernel are the two-torsion points.
 \end{proof} 
 
We will use theta-functions to determine explicit formulas relating the Rosenhain roots for $\mathcal{C}$ to the roots for a particular 
choice of $(2,2)$-isogenous curve, say $\hat{\mathcal{C}\,}\!_{12}$, as follows: let the dual abelian surface $\hat{\mathbf{A}}_{12} = \operatorname{Jac}(\hat{\mathcal{C}\,}\!_{12})$
be given by the sextic curve 
\begin{equation}
\label{Eq:Rosenhain2}
 \hat{\mathcal{C}\,}\!_{12}: \quad y^2 = x\,z\,\big(x-z\big) \, \big( x- \Lambda_1z\big) \,  \big( x- \Lambda_2 z\big) \,  \big( x- \Lambda_3 z\big) \;,
\end{equation} 
with Rosenhain roots
\begin{equation}\label{Picard_sq}
\Lambda_1 = \frac{\Theta_1^2\Theta_3^2}{\Theta_2^2\Theta_4^2} \,, \quad \Lambda_2 = \frac{\Theta_3^2\Theta_8^2}{\Theta_4^2\Theta_{10}^2}\,, \quad \Lambda_3 =
\frac{\Theta_1^2\Theta_8^2}{\Theta_2^2\Theta_{10}^2}\,.
\end{equation}
Thus, the $(2,2)$-isogeny $\psi_{12}: \mathbf{A}_{12}=\operatorname{Jac}(\mathcal{C}) \to \hat{\mathbf{A}}_{12}=\operatorname{Jac}(\hat{\mathcal{C}\,}\!_{12})$ 
is realized by the change of moduli in Equations~(\ref{Picard_sq_b}) and defines a maximal isotropic subgroup $\mathsf{k}_{12}$ 
such that $\hat{\mathbf{A}}_{12}=\mathbf{A}/\mathsf{k}_{12}$. We have the following:
\begin{lemma}
\label{lem:R2isog}
Taking the quotient by the maximal isotropic subgroup $\mathsf{k}_{12} \subset \mathbf{A}[2]$ 
corresponds to the Richelot isogeny obtained from pairing the linear factors according to
$(1,\lambda_1)$, $(\lambda_2,\lambda_3)$, $(0,\infty)$ for the genus-two curve $\mathcal{C}$ in Equation~(\ref{Eq:Rosenhain})
\end{lemma}
\begin{proof}
We compute the Richelot-isogeny in Equation~(\ref{Richelot}) obtained from pairing the roots according to
$(\lambda_1,\lambda_4=1)$, $(\lambda_2,\lambda_3)$, $(\lambda_5=0,\lambda_6=\infty)$. For this new
curve we compute its Igusa-invariants which are in fact rational functions only of the theta-functions
appearing in $[\theta_1: \theta_2: \theta_3: \theta_4]$. We then compute the Igusa invariants for the quadratic twist $\hat{\mathcal{C}\,}\!_{12}^{(\mu)}$ 
of the curve in Equation~(\ref{Eq:Rosenhain2}). They again are  rational functions of only the theta-functions
appearing in $[\theta_1: \theta_2: \theta_3: \theta_4]$ since the Rosenhain roots of $\hat{\mathcal{C}\,}\!_{12}$ are determined by the theta 
nulls $[\theta_1: \theta_2: \theta_3: \theta_4]$ using the equations
\begin{equation}\label{Picard_sq_b}
\begin{split}
\Lambda_1 & = \frac{(\theta_1^2 + \theta_2^2 + \theta_3^2 + \theta_4^2)(\theta_1^2 - \theta_2^2 - \theta_3^2 + \theta_4^2)}{(\theta_1^2 + \theta_2^2 - \theta_3^2 - \theta_4^2)(\theta_1^2 - \theta_2^2 + \theta_3^2 - \theta_4^2)}, \\
\Lambda_2 & =  \frac{(\theta_1^2 - \theta_2^2 - \theta_3^2 + \theta_4^2)(\theta_1^2\theta_2^2+\theta_3^2\theta_4^2+2 \theta_1\theta_2\theta_3\theta_4)}{(\theta_1^2 - \theta_2^2 + \theta_3^2 - \theta_4^2)(\theta_1^2\theta_2^2-\theta_3^2\theta_4^2)} ,\\
\Lambda_3 & =  \frac{(\theta_1^2 + \theta_2^2 + \theta_3^2 + \theta_4^2)(\theta_1^2\theta_2^2+\theta_3^2\theta_4^2+2 \theta_1\theta_2\theta_3\theta_4)}{(\theta_1^2 + \theta_2^2 - \theta_3^2 - \theta_4^2)(\theta_1^2\theta_2^2-\theta_3^2\theta_4^2)} \;.
\end{split}
\end{equation}
The two sets of Igusa invariants are identical for 
\[
\mu = \frac{(\theta_1\theta_2-\theta_3\theta_4)^2(\theta_1^2 + \theta_2^2 - \theta_3^2 - \theta_4^2)(\theta_1^2 - \theta_2^2 + \theta_3^2 - \theta_4^2)}{4 \, \theta_1\theta_2\theta_3\theta_4(\theta_1^2 + \theta_2^2 + \theta_3^2 + \theta_4^2)(\theta_1^2 - \theta_2^2 - \theta_3^2 + \theta_4^2)} \;.
\]
\end{proof}

To see the relation between the moduli of the isogenous curves $\mathcal{C}$ and $\hat{\mathcal{C}\,}\!_{12}$ directly, we introduce the new moduli
$\lambda_1' = ( \lambda_1 + \lambda_2 \lambda_3)/l$,  $\lambda_2' = (\lambda_2 + \lambda_1 \lambda_3)/l$, 
and  $\lambda_3' = (\lambda_3 + \lambda_1 \lambda_2)/l$, and similarly $\Lambda_1', \Lambda_2', \Lambda_3'$
with $l^2 = \lambda_1 \lambda_2 \lambda_3$ and $L^2=  a_1 \Lambda_2 \Lambda_3$.
One checks by explicit computation the following lemma:
\begin{lemma}
The sets of parameters $\{\lambda_i'\}$ and $\{\Lambda_i'\}$ can each be expressed as rational functions 
entirely in terms of the squares of theta-constants contained in $\lbrace \theta_1^2 , \theta_2^2, \theta_3^2, \theta_4^2\rbrace$ 
or $\lbrace \Theta_1^2 , \Theta_2^2, \Theta_3^2, \Theta_4^2\rbrace$.
\end{lemma}
\begin{proof}
By direct computation.
\end{proof}

Moreover, we have the following relations:
\begin{proposition}
\label{lem:2isog_curve}
The moduli of the genus-two curve $\mathcal{C}$ in Equation~(\ref{Eq:Rosenhain}) and the $(2,2)$-isogenous genus-two curve $\hat{\mathcal{C}\,}\!_{12}$ in Equation~(\ref{Eq:Rosenhain2})
are related by
\begin{equation}
\label{relations_RosRoots}
 \begin{split}
 \begin{array}{rl}
   \Lambda_1' & = 2 \, \frac{2 \lambda_1' - \lambda_2'-\lambda_3'}{\lambda_2'-\lambda_3'} \,,\\[0.6em]
   \Lambda_2' - \Lambda_1' & = - \frac{4(\lambda_1'-\lambda_2')(\lambda_1'-\lambda_3')}{(\lambda_1'+2)(\lambda_2'-\lambda_3')} \,,\\[0.6em]
   \Lambda_3' - \Lambda_1' & = - \frac{4(\lambda_1'-\lambda_2')(\lambda_1'-\lambda_3')}{(\lambda_1'-2)(\lambda_2'-\lambda_3')} \,,
  \end{array}
  & \qquad
 \begin{array}{rl}
   \lambda_1' & = 2 \, \frac{2 \Lambda_1' - \Lambda_2'-\Lambda_3'}{\Lambda_2'-\Lambda_3'} \,,\\[0.6em]
   \lambda_2' - \lambda_1' & = - \frac{4(\Lambda_1'-\Lambda_2')(\Lambda_1'-\Lambda_3')}{(\Lambda_1'+2)(\Lambda_2'-\Lambda_3')} \,,\\[0.6em]
   \lambda_3' - \lambda_1' & = - \frac{4(\Lambda_1'-\Lambda_2')(\Lambda_1'-\Lambda_3')}{(\Lambda_1'-2)(\Lambda_2'-\Lambda_3')} \,.
  \end{array}   
  \end{split}
\end{equation}
\end{proposition}
\begin{proof}
The proof follows by direct computation using the so-called \emph{principal transformations of degree two}~\cite{MR0141643, MR0168805}
that relate the theta-constants of $\{\theta_1, \theta_2, \theta_3, \theta_4\}$ 
and $\{\Theta_1, \Theta_2, \Theta_3, \Theta_4\}$ by the equations
\begin{equation}
\label{Eq:degree2doubling}
\begin{array}{lllclll}
 \theta_1^2 & = & \Theta_1^2 + \Theta_2^2 + \Theta_3^2 + \Theta_4^2 \,, &\qquad
 \theta_2^2 & =&  \Theta_1^2 + \Theta_2^2 - \Theta_3^2 - \Theta_4^2 \,, \\[0.2em]
 \theta_3^2 & = &  \Theta_1^2 - \Theta_2^2 - \Theta_3^2 + \Theta_4^2 \,, &\qquad
 \theta_4^2 &= &  \Theta_1^2 - \Theta_2^2 + \Theta_3^2 - \Theta_4^2 \,,
\end{array}
\end{equation}
and
\begin{equation}
\label{Eq:degree2doublingR}
\begin{array}{lllclll}
 \theta_5^2 & = & 2 \, \big( \Theta_1 \Theta_3 + \Theta_2  \Theta_4 \big) \,, &\qquad
 \theta_6^2 & =&  2 \, \big( \Theta_1 \Theta_3 - \Theta_2  \Theta_4 \big)\,, \\[0.4em]
 \theta_7^2 & = & 2 \, \big( \Theta_1 \Theta_4 + \Theta_2  \Theta_3 \big) \,, &\qquad
 \theta_8^2 &= &  2 \, \big( \Theta_1 \Theta_2 + \Theta_3  \Theta_4 \big)\,, \\[0.4em]
 \theta_9^2 & = & 2 \, \big( \Theta_1 \Theta_4 - \Theta_2  \Theta_3 \big)\,, &\qquad
 \theta_{10}^2 &= & 2 \, \big( \Theta_1 \Theta_2 - \Theta_3  \Theta_4 \big) \,.
\end{array}
\end{equation}
\end{proof}

Similarly, we can express the Rosenhain roots in Equation~(\ref{Picard}) in terms of the theta nulls $[\Theta_1: \Theta_2: \Theta_3: \Theta_4]$.
This defines a maximal isotropic subgroup $\!\hat{\,\mathsf{k}}_{12}$ and its image
$\mathsf{K}_{12}$. We have the following corollary:
\begin{corollary}
\label{cor:symmetric}
The quotient of the abelian surface $\hat{\mathbf{A}}_{12}$ by the maximal isotropic subgroup $\mathsf{K}_{12}$
equals $\hat{\mathbf{A}}_{12}/\mathsf{K}_{12}=\mathbf{A}$.
Moreover, using the maximal isotropic subgroups $\mathsf{k}_{12}$ and $\mathsf{K}_{12}$ the relation between $\mathbf{A}=\mathcal{C}$ and $\hat{\mathbf{A}}_{12}=
\hat{\mathcal{C}\,}\!_{12}$ and the relation between the $(2,2)$-isogenies $\psi_{12}: \operatorname{Jac}(\mathcal{C}) \to \operatorname{Jac}(\hat{\mathcal{C}\,}\!_{12})$ and 
$\hat{\psi}_{12}: \operatorname{Jac}(\hat{\mathcal{C}\,}\!_{12}) \to \operatorname{Jac}(\mathcal{C})$
are symmetric. 
\end{corollary}
\begin{proof}
By construction the maximal isotropic subgroups $\mathsf{k}_{12}$ and 
$\mathsf{K}_{12}$ give a decomposition described in Lemma~\ref{lem:decomp}.
We use Equations~(\ref{Eq:degree2doubling}) and  Equations~(\ref{Eq:degree2doublingR}) to express Equations~(\ref{Picard}) in terms of $\lbrace \Theta_1, \Theta_2, \Theta_3, \Theta_4\rbrace$.
The resulting equations are identical with Equations~(\ref{Picard_sq_b}) with 
$\Lambda_i \leftrightarrow \lambda_i$ and $\theta_i \leftrightarrow \Theta_i$ interchanged.
\end{proof}

\section{Normal form of $(1,2)$-polarized Kummer surface}
\label{sec:kummers}
\subsection{Bielliptic and hyperelliptic genus-three curves}
Let $\mathcal{D}$ be an irreducible, smooth, projective curve of genus three, defined over $\mathbb{C}$.
Within the (coarse) moduli space of curves of genus three $\mathcal{M}_3$ we denote the hyperelliptic locus by $\mathcal{M}_3^h$. 
It is known that  $\mathcal{M}_3^h$ is an irreducible five-dimensional sub-variety\footnote{
The hyperelliptic involution on an irreducible, smooth, projective curve of genus $g$ is unique if $g \ge 2$.} of $\mathcal{M}_3$.
We also define the bielliptic locus
\[
 \mathcal{M}_3^{b} = \left\lbrace [\mathcal{D}] \in \mathcal{M}_3 \Big| \; \mathcal{D} \; \text{is bielliptic} \right\rbrace \;,
 \]
 where bielliptic means that $\mathcal{D}$ admits a degree-two morphism $\pi^{\mathcal{D}}_b: \mathcal{D} \to \mathcal{E}$ onto an elliptic curve,
 and we denote by $[\mathcal{D}]\in \mathcal{M}_3$ the isomorphism class of $\mathcal{D}$. 
 We denote by $\imath^{\mathcal{D}}_b$ the involution, i.e., the element of $\operatorname{Aut}(\mathcal{D})$ which interchanges the sheets of $\pi^{\mathcal{D}}_b$, so that  $\mathcal{E} \cong 
 \mathcal{D} /\langle\imath^{\mathcal{D}}_b \rangle$.  We recall from \cite{MR932781} that $\mathcal{M}_3^{b}$ is an irreducible four-dimensional sub-variety of  $\mathcal{M}_3$, and 
 it is the unique component of maximal dimension of the singular locus\footnote{By the Castelnuovo-Severi inequality it follows
 that curves of genus $g \ge 6$ admit precisely one bielliptic structure and that bielliptic curves of genus $g \ge 4$ cannot be hyperelliptic.} of $\mathcal{M}_3$.
 The following theorem was proved in \cite{MR932781}:
 \begin{theorem}
  The generic bielliptic curve of genus 3 carries exactly one bielliptic structure.
 \end{theorem} 
 \label{thm1}
 Moreover, the following theorem was proved in \cite{MR1816214}:
 \begin{theorem}
 \label{thm:genus3-hyperelliptc-bielliptic}
 \begin{enumerate}
 \item[]
 \item $[ \mathcal{D}] \in \mathcal{M}_3^{b} \cap \mathcal{M}_3^h$ iff $\mathcal{D}$ is a double cover of a genus-two curve $\mathcal{C}$.
 \item $\mathcal{M}_3^{b} \cap \mathcal{M}_3^h$ is an irreducible, 3-dimensional, rational sub-variety of $\mathcal{M}_3^{b}$.
 \end{enumerate}
  \end{theorem} 
  Every genus-two curve $\mathcal{C} \in \mathcal{M}_2$ has a canonical hyperelliptic structure, and hence it  can be described
  as double cover of $\mathbb{P}^1$ branched at $6$ points. We chose $4$ out 
  of these six points to be the image of 4 pairs of points on $\mathcal{D}\in \mathcal{M}_3^{b} \cap \mathcal{M}_3^h$ such that all eight points in the preimage 
  are fixed under the hyperelliptic involution and each pair is 
  kept fixed by the bielliptic involution. It follows that there are exactly $15=\binom{6}{4}$ inequivalent hyperelliptic and bielliptic genus-three curves
  $\mathcal{D}_{ij}$ for $1 \le i < j \le 6$ that double cover the genus-two curve $\mathcal{C}$.  In fact, it is obvious that the choice of a double cover $\mathcal{D}$ 
   is determined by the choice of Richelot isogeny in Equation~(\ref{Richelot}) and vice versa.
  
  We now determine a normal form for a hyperelliptic and bielliptic genus-three curve $\mathcal{D}$ 
  that is the double cover of the genus-two curve $\mathcal{C}$. 
  We think of this particular genus-three curve
  as a section of $\mathcal{M}_3^{b} \cap \mathcal{M}_3^h \to \mathcal{M}_2$. To do so, we consider the pull-back of the identification $\mathcal{M}_3^{b} \cap \mathcal{M}_3^h \to \mathcal{M}_2$
  along the Galois covering $\mathcal{M}_2(2) \to \mathcal{M}_2$ where $\mathcal{M}_2(2)$ is the moduli space of genus-two curves with level-two structure.
  For a genus-two curve $\mathcal{C}$ given as sextic $Y^2 = f_6(X,Z)$, a class in $\mathcal{M}_2(2)$ is given by the ordered tuple $(\lambda_1, \lambda_2, \lambda_3)$
  after we sent the three remaining roots $\lambda_4, \lambda_5, \lambda_6$ to $1, 0, \infty$. We then choose the points $(\lambda_1, \lambda_2, \lambda_3, 1)$ to
  be the images of the eight ramification points of $\mathcal{D}$.
 
 \begin{remark} 
  Following \cite{MR2166928} we will describe all curves (and fibrations) in the form $Y^2=F_{2n}(X,Z)$ where $F$ is a homogeneous polynomial having $2n$ distinct roots. 
  This defines a hypersurface in $\mathbb{WP}(1,n,1)$ which is the general hyperelliptic curve of genus $g=n-1$. Because of the monomial $Y^2$ 
  the hypersurface does not pass through the cone point $[0:1:0]$. The union of two affine charts given by
  $X=1$ and $Z=1$ glued together is the double cover with $2n$ branch points.
 \end{remark}
  
  \subsubsection{A normal form}
  We assume that the genus-two hyperelliptic curve $\mathcal{C}$ in Theorem~\ref{thm:genus3-hyperelliptc-bielliptic} is in Rosenhain normal form, 
 i.e., for $[X:Y:Z] \in \mathbb{WP}(1,3,1)$ the genus-two curve $\mathcal{C}$ is given by
\begin{equation}
\label{Eq:Rosenhain_g2}
  Y^2 = X\,Z \, \prod_{i=1}^4 \big(X-\lambda_i \, Z\big)  \;,
\end{equation} 
with the hyperelliptic involution given by $\imath^{\mathcal{C}}: [X:Y:Z] \mapsto [X:-Y:Z]$ and the corresponding hyperelliptic map $\pi^{\mathcal{C}}: 
\mathcal{C} \to \mathbb{P}^1$ given by $[X:Y:Z]  \mapsto [X:Z]$. The hyperelliptic involution on $\mathcal{C}$ has the 6 fixed points 
$[\lambda_i:0:1]$ for $i=1, \dots,5$ with $\lambda_4=1$, $\lambda_5=0$, and $[1:0:0]$. We also choose $l$ with $l^2=\lambda_1 \lambda_2\lambda_3$.

By Theorem~\ref{thm1}, generically a 
hyperelliptic and bielliptic genus-three curve $\mathcal{D}$ in the preimage of $\mathcal{M}_3^{b} \cap \mathcal{M}_3^h \to \mathcal{M}_2$ is given by
\begin{equation}
\label{Eq:Rosenhain_g3}
 y^2 = \big(x^2-z^2\big) \, \big( x^2- \lambda_1\, z^2\big) \,  \big( x^2- \lambda_2 \, z^2\big) \,  \big( x^2- \lambda_3 \, z^2\big)  
\end{equation} 
for $[x:y:z] \in \mathbb{WP}(1,4,1)$. The hyperelliptic involution $\imath^{\mathcal{D}}_{h}: [x:y:z] \mapsto [x:-y:z]$, 
and the corresponding hyperelliptic map $\pi^{\mathcal{D}}: \mathcal{D} \to \mathbb{P}^1$ is
$[x:y:z]  \mapsto [x:z]$. The hyperelliptic involution on $\mathcal{D}$ has the 8 fixed points $[\pm \sqrt{\lambda_i}:0:1]$ for $i=1, \dots,4$. A bielliptic involution on $\mathcal{D}$ 
is $\imath^{\mathcal{D}}_b: [x:y:z] \mapsto [-x:y:z] = [x:y:-z]$ and has the 4 fixed points given by $[0:\pm l:1]$ and $[1:\pm1:0]$.

The composition $\imath^{\mathcal{D}}_b \circ \imath^{\mathcal{D}}_h$ is fixed-point-free. An \'etale double cover $\pi^\mathcal{D}_\mathcal{C}: \mathcal{D}  \to \mathcal{C}$ is then given by
\begin{equation}
\label{mapQC}
\begin{split}
  \pi^\mathcal{D}_\mathcal{C}: \quad [x:y:z] & \mapsto  [X:Y:Z]=[x^2:xyz:z^2] \;.
 \end{split} 
 \end{equation}
 The images of the 4 pairs of hyperelliptic fixed points $[\pm \sqrt{\lambda_i}:0:1]$ and the two pairs of bielliptic fixed points $[0:\pm l:1]$ and $[1:\pm1:0]$
 under $\pi^\mathcal{D}_\mathcal{C}$ are exactly the six hyperelliptic fixed points of the genus-two curve $\mathcal{C}$.
 
The quotient curve $\mathcal{E} \cong  \mathcal{D} /\langle\imath^{\mathcal{D}}_b \rangle$ obtained from the bielliptic involution is given by 
\begin{equation}
\label{Eq:Rosenhain_g1}
  Y^2 = \big(X-Z\big) \, \big( X- \lambda_1\, Z\big) \,  \big( X- \lambda_2 \, Z\big) \,  \big( X- \lambda_3 \, Z\big) 
\end{equation} 
with $[X:Y:Z] \in \mathbb{WP}(1,2,1)$, and the double cover $\pi^{\mathcal{D}}_{\mathcal{E}}: \mathcal{D} \to \mathcal{E}$ is given by
\[
 \pi^{\mathcal{D}}_{\mathcal{E}}: \quad [ x:y: \pm z] = [- x:y: \mp z] \mapsto  [X:Y:Z]=[x^2:y:z^2] \;.
 \]
The four branch points of $ \pi^{\mathcal{D}}_{\mathcal{E}}$ are the bielliptic fixed points $[0:\pm l:1]$ and $[1:\pm1:0]$.
The hyperelliptic involution on $\mathcal{E}$ is $\imath^{\mathcal{E}}_{h}: [X:Y:Z] \mapsto [X:-Y:Z]$, 
and the corresponding hyperelliptic map $\pi^{\mathcal{E}}: \mathcal{E} \to \mathbb{P}^1$ is given by  $[X:Y:Z]  \mapsto [X:Z]$. 

We also introduce a smooth plane quartic curve $\mathcal{Q}$ given by
\begin{equation}
\label{Eq:Rosenhain_g3b}
  \hat{Y}^4  = \big(X-Z\big) \, \big( X- \lambda_1\, Z\big) \,  \big( X- \lambda_2 \, Z\big) \,  \big( X- \lambda_3 \, Z\big) 
\end{equation} 
for $[X:\hat{Y}:Z] \in \mathbb{P}^3$, and the double cover $\pi^{\mathcal{Q}}_{\mathcal{E}}: \mathcal{Q} \to \mathcal{E}$ is given by
\[
 \pi^{\mathcal{Q}}_{\mathcal{E}}: \quad [ X: \pm \hat{Y}: Z] \mapsto  [X:y:Z]=[X:\hat{Y}^2:Z] \;.
 \]
 The four branch points of $\pi^{\mathcal{Q}}_{\mathcal{E}}$ are the bielliptic fixed points $[\lambda_i:0:1]$ for $1\le i \le 4$.
By the genus-degree formula $\mathcal{Q}$ has genus three.

The situations is depicted in Figure~\ref{Relations_Curves}. The morphism $\pi^{\mathcal{Q}}_{\mathcal{E}}$ and $\pi^{\mathcal{D}}_{\mathcal{E}}$ have
the ramification given by the points $[\lambda_i:0:1]$ for $1\le i \le 4$ and the points $[0:\pm l:1]$ and $[1:\pm1:0]$, respectively. 
The hyperelliptic map $\pi^{\mathcal{E}}: \mathcal{E} \to \mathbb{P}^1$ projects these points to $[X:Z]=[\lambda_i:1]$ and $[X:Z]=[1:0]$ and $[0:1]$,
i.e., precisely the image of the six Weierstrass points of $\mathcal{C}$ under the projection $\pi^{\mathcal{C}}: \mathcal{C} \to \mathbb{P}^1$.
 \begin{figure}[ht]
 \scalebox{0.75}{
\centerline{
\xymatrix{
& \mathcal{D}^{(g=3)}: y^2 = \prod_{i=1}^4 \left(x^2 - \lambda_i z^2\right)  \ar[dd]_{\pi^{\mathcal{D}}_{\mathcal{E}}} \ar[rd]_{\pi^{\mathcal{D}}_{\mathcal{C}}} \\
\mathcal{Q}^{(g=3)}: \hat{Y}^4 = \prod_{i=1}^4 \left(Z - \lambda_i Z\right)  \ar[rd]_{\pi^{\mathcal{Q}}_{\mathcal{E}}}
& & \mathcal{C}^{(g=2)}: Y^2 = X  Z \, \prod_{i=1}^4 \left(X - \lambda_i Z\right)\ar[dd]_{\pi^{\mathcal{C}}} \\
& \mathcal{E}^{(g=1)}: y^2 = \prod_{i=1}^4 \left(X - \lambda_i Z\right) \ar[dr]_{\pi^{\mathcal{E}}}   \\
&&  \mathbb{P}^1 \ni [X:Z]
}}}
 \caption{\label{Relations_Curves}}
\end{figure}
 
\subsection{Kummer surface from the genus-two curve $\mathcal{C}$}
We start with two copies of the genus-two curve in Equation~(\ref{Eq:Rosenhain}) and form the symmetric product 
$\mathcal{C}^{(2)} = (\mathcal{C}\times\mathcal{C})/\langle \sigma_{\mathcal{C}^{(2)} } \rangle$ where $\sigma_{\mathcal{C}^{(2)}}$ interchanges
the copies of $\mathcal{C}$.
The variety $\mathcal{C}^{(2)}/\langle \imath^\mathcal{C} \times  \imath^\mathcal{C} \rangle$ is given in terms of the variables
$[z_1:z_2:z_3:\tilde{z}_4] \in \mathbb{WP}(1,1,1,3)$ with
\begin{equation}
\label{transfo_kummer}
 z_1=Z^{(1)}Z^{(2)},\; z_2=X^{(1)}Z^{(2)}+X^{(2)}Z^{(1)}, \; z_3=X^{(1)}X^{(2)}, \; \tilde{z}_4=Y^{(1)}Y^{(2)}
 \end{equation}
by the equation
\begin{equation}
\label{kummer_middle}
  \tilde{z}_4^2 = z_1  z_3 \,  \prod_{i=1}^4 \big( \lambda_i^2 \, z_1  -  \lambda_i \, z_2 +  z_3 \big) \;.
\end{equation}

We will use the notation $p_{ij}$ with $1\le i < j \le 6$ to label the $15$ singular points of Equation~(\ref{kummer_middle}), known as nodes.
The nodes descend from the elements $\mathfrak{p}_{ij} \in \mathbf{A}[2]$ in Equation~(\ref{oder2points}).
The ten nodes $p_{ij}$  with $1\le i < j \le  5$ are given by setting
$$
 [z_1: z_2 : z_3 : \tilde{z}_4] = [1: \lambda_i+\lambda_j : \lambda_i \lambda_j: 0] \;,
$$
where we have used $\lambda_4 =1$, $\lambda_5=0$.
These points are obtained by combining Weierstrass points $\mathfrak{p}_{i}$ and $\mathfrak{p}_{j}$  on $\mathcal{C}$ given by
\begin{equation}
\label{pencilC_DP}
[X^{(1)}:Y^{(1)}:Z^{(1)}]=[\lambda_i: 0 : 1] \; \text{and} \; [X^{(2)}:Y^{(2)}:Z^{(2)}]=[\lambda_j: 0 : 1] 
\end{equation}
in Equation~(\ref{transfo_kummer}).
Similarly, five nodes  $p_{i6}$ for $1\le i \le 5$ are given by
$$
 [z_1: z_2 : z_3 : \tilde{z}_4] = [0: 1 : \lambda_i : 0] ,
$$
and obtained by combining Weierstrass points $\mathfrak{p}_{i}$ and $\mathfrak{p}_{6}$ on $\mathcal{C}$ given by
\begin{equation}
\label{pencilC_DP_b}
[X^{(1)}:Y^{(1)}:Z^{(1)}]=[\lambda_i: 0 : 1] \; \text{and} \; [X^{(2)}:Y^{(2)}:Z^{(2)}]=[1: 0 : 0] 
\end{equation}
in Equation~(\ref{transfo_kummer}).
 In Table~\ref{tab:nodes} we listed the nodes $p_{ij}$ corresponding to the
 points $\mathfrak{p}_{ij} \in \mathbf{A}[2]$. We relate these nodes to the notation used in Kumar~\cite{MR3263663} 
 and Mehran~\cite{MR2804549} in Table~\ref{tab:nodes}.  Mehran uses the notation $e_{i'j'}$ for the nodes.

\begin{table}
\scalebox{0.9}{
\begin{tabular}{c|c|c|l}
$p$ & $n$ & $e$  & $[z_1: z_2 : z_3 : \tilde{z}_4]$\\
\hline
$p_0$ 	& $n_0$ 		& $e_0$ 	& --\\
$p_{56}$	& $n_1$ 		& $e_{12}$ &  $[0: 1 : 0 : 0]$\\
$p_{16}$	& $n_3$ 		& $e_{14}$ &  $[0: 1 : \lambda_1 : 0]$\\
$p_{26}$	& $n_4$ 		& $e_{15}$ &  $[0: 1 : \lambda_2 : 0]$\\
$p_{36}$	& $n_5$ 		& $e_{16}$ &  $[0: 1 : \lambda_3 : 0]$\\
$p_{46}$	& $n_2$ 		& $e_{13}$ &  $[0: 1 : 1 : 0]$\\
$p_{15}$	& $n_{13}$ 	& $e_{24}$ &  $[1: \lambda_1 : 0 : 0]$\\
$p_{25}$	& $n_{14}$ 	& $e_{25}$ &  $[1: \lambda_2 : 0 : 0]$\\
$p_{35}$	& $n_{15}$ 	& $e_{26}$ &  $[1: \lambda_3 : 0 : 0]$\\
$p_{45}$	& $n_{12}$ 	& $e_{23}$ &  $[1: 1 : 0 : 0]$\\
$p_{14}$	& $n_{23}$ 	& $e_{34}$ &  $[1: \lambda_1+1 : \lambda_1 : 0]$\\
$p_{24}$	& $n_{24}$ 	& $e_{35}$ &  $[1: \lambda_2+1 : \lambda_2 : 0]$\\
$p_{34}$	& $n_{25}$ 	& $e_{36}$ &  $[1: \lambda_3+1 : \lambda_3 : 0]$\\
$p_{13}$	& $n_{35}$ 	& $e_{46}$ &  $[1: \lambda_1+\lambda_3 : \lambda_1\lambda_3 : 0]$\\
$p_{23}$	& $n_{45}$ 	& $e_{56}$ &  $[1: \lambda_2+\lambda_3 : \lambda_2\lambda_3 : 0]$\\
$p_{12}$	& $n_{34}$ 	& $e_{45}$ &  $[1: \lambda_1+\lambda_2 : \lambda_1\lambda_2 : 0]$
\end{tabular}}
\bigskip
\caption{Nodes on a generic Jacobian Kummer surface}
\label{tab:nodes}
\end{table}

Six tropes, i.e., hypersurfaces that intersect the quartic surface in a conic, are now easily found by inspection: for a given integer 
$i$ with $1 \le i \le 5$, the nodes $p_{ij}$ or $p_{ji}$ all lie on the plane
\begin{equation}
 \mathsf{T}_i: \quad \lambda_i^2 z_1 - \lambda_i z_2 + z_3 = 0 \;,
\end{equation}
and we set $\mathsf{T}_6: z_1=0$. Thus, Equation~(\ref{kummer_middle}) becomes
\[
 \tilde{z}_4^2 = \mathsf{T}_1 \mathsf{T}_2 \mathsf{T}_3 \mathsf{T}_4 \mathsf{T}_5 \mathsf{T}_6 \;,
\]
which is a double cover of $\mathbb{P}^2$ branched along a reducible plane sextic curve (see Fig.~\ref{6Lines}) -- the union of six lines all tangent to a conic. In fact,
the trope $\mathsf{T}_i$ is tangent to the conic $K_2=z_2^2 - 4 \, z_1 \, z_3=0$ 
at the non-singular points $p_{ii}: [1: 2\lambda_i : \lambda_i^2 :0]$ for $i=1,\dots,5$, and $\mathsf{T}_6$ is tangent to 
$K_2=0$ at the point $p_{66}: [0:0:1:0]$. These points of tangency are obtained by combining the Weierstrass points on $\mathcal{C}$ in Equation~(\ref{pencilC_DP}) 
for $i=j=1, \dots 5$ or two copies of $\mathfrak{p}_6 \in \mathcal{C}$, respectively, in Equation~(\ref{transfo_kummer}).
A $16_6$-configuration descends from the $16_6$-configuration on $\mathbf{A}=\operatorname{Jac}(\mathcal{C})$ with
the remaining $10$ tropes $\mathsf{T}_{ijk}$ with $1\le i < j < k \le 6$ descending from the theta-divisors $\mathfrak{t}_{ijk}$ .

Conversely, if we start with a singular quartic surface in $\mathbb{P}^3$, called Kummer quartic, 
with the maximal number of simple nodes which -- for a quartic in $\mathbb{P}^3$ -- is 16. 
For the singular point $p_0$ on this singular quartic surface in $\mathbb{P}^3$, we identify the lines in $\mathbb{P}^3$ through the point $p_0$ with $\mathbb{P}^2$ 
and map any line in the tangent cone of $p_0$ to itself. In this way one obtains a double cover of $\mathbb{P}^2$ branched along 
a plane curve of degree six  where all the nodes of the quartic surface different from $p_0$ map to nodes of the sextic.
By the genus-degree formula, the maximal number of nodes on a sextic curve is obtained when the curve is a union of six lines, in which case we have fifteen
remaining nodes apart from $p_0$. In particular, the branch locus of the double cover to $\mathbb{P}^2$ is a reducible plane sextic curve, namely the union of six lines all tangent to a conic. 

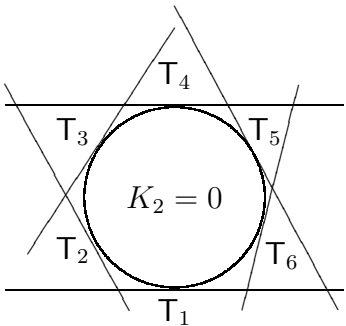
\begin{figure}[ht]
$$
  \begin{xy}
   <0cm,0cm>;<1.5cm,0cm>:
    (2,-.3)*++!D\hbox{$\mathsf{T}_1$},
    (1.1,0.25)*++!D\hbox{$\mathsf{T}_2$},
    (1.1,1.3)*++!D\hbox{$\mathsf{T}_3$},
    (2,1.8)*++!D\hbox{$\mathsf{T}_4$},
    (2.8,1.3)*++!D\hbox{$\mathsf{T}_5$},
    (2.95,0.2)*++!D\hbox{$\mathsf{T}_6$},
    (2,0.7)*++!D\hbox{$K_2=0$},
    (.5,0.18);(3.5,0.18)**@{-},
    (0.5,2);(1.6,0)**@{-},
    (3.1,2);(2.6,0)**@{-},
    (0.7,.5);(2,2.5)**@{-},
    (0.5,1.82);(3.5,1.82)**@{-},
    (2,2.7);(3.45,0)**@{-},
    (2,1)*\xycircle(.8,.8){},
  \end{xy}
  $$
\caption{Double cover branched along reducible sextic}
\label{6Lines}
\end{figure}

\begin{remark}
Let $\mathcal{L}$ be the ample line bundle on the abelian surface $\mathbf{A}=\operatorname{Jac}(\mathcal{C})$ defining its principal polarization 
and consider the rational map $\Phi_{\mathcal{L}^2}: \mathbf{A} \to \mathbb{P}^3$ associated with the line bundle $\mathcal{L}^2$. 
The map $\Phi_{\mathcal{L}^2}$ factors via an embedding of $\mathbf{A}/\langle -\mathbb{I} \rangle$ into $\mathbb{P}^3$ \cite[Sec.~10.2]{MR2062673}.
\end{remark}

\subsubsection{An elliptic fibration}
\label{sec:elliptic_fibration_middle}
We use the rational transformation $[z_1:z_2:z_3:\tilde{z}_4]=[T_1z: x: T_2 z: y z]$ 
with $[(T_1, T_2) , (x,y,z)] \in \mathbb{F}^{3,2}_{-1}$ where we have introduced 
\begingroup\makeatletter\def\f@size{10}
 \begin{align}
 \label{HSB}
  \mathbb{F}^{l,m}_{-k} = \left\lbrace  \begin{array}{c} \Big( (T_1, T_2) , (x,y,z) \Big) \\ \in \mathbb{C}^2_{\not = 0} \times \mathbb{C}^3_{(x,z)\not = 0} \end{array} \Big|\;  \forall \lambda \in \mathbb{C}^\star:\!\!   \begin{array}{l} 
  \big((T_1, T_2) , (x,y,z)\big)  \sim  \big(T_1, T_2) , (\lambda x, \lambda^m y, \lambda z)\big) \\
  \big((T_1, T_2) , (x,y,z)\big)  \sim  \big((\lambda T_1, \lambda T_2) , (\lambda^k x, \lambda^l y, z)\big)
 \end{array}
 \right\rbrace
\end{align}
\endgroup
such that $\pi^{l,m}_{-k}: \mathbb{F}^{l,m}_{-k} \to \mathbb{F}_{-k}$ with $[(T_1, T_2) , (x,y,z)]  \mapsto  [(T_1, T_2) , (x,z)]$
is the projection onto the Hirzebruch surface $\mathbb{F}_{-k}\cong\mathbb{F}_{1+k}$.
Equation~(\ref{kummer_middle}) becomes the equation of a genus-one fibration $\bar{\mathcal{Y}\,}\!_1$ over $\mathbb{P}^1$ given by
\begin{equation}
\label{kummer_middle_ell}
   y^2 = T_1 T_2 \,  \prod_{i=1}^4 \left( \lambda_i \, x  - \big(\lambda_i^2 \, T_1  +T_2 \big) \, z \right)\;.
\end{equation}
It is easy to check that the projection $\pi^{ \mathcal{Y}_1}:  \bar{\mathcal{Y}\,}\!_1\to \mathbb{P}^1$ with $([T_1:T_2],[x:y:z]) \mapsto [T_1:T_2]$ 
defines a K3 fibration with two fibers of Kodaira-type $I_0^*$ and
six fibers of Kodaira-type $I_2$. Four sections are given by  $\mathsf{E}_i: [x:y:z]=[\lambda_i^2 \, T_1  +T_2:0:\lambda_i]$
for $1\le i \le4$. We choose the point given by $\mathsf{E}_4$ to be the neutral element of the Mordell-Weil group of sections,
turning Equation~(\ref{kummer_middle_ell}) into a Jacobian elliptic fibration. The Mordell-Weil group is generated by the 4 two-torsion sections 
$\mathsf{E}_i$ for $1\le i \le4$ and the section $\mathsf{P}$ of infinite order given by
\begin{equation}
\mathsf{P}: \left\lbrace\begin{array}{rl}
 x & = (T_1 + T_2) ( T_2 -\lambda_1\lambda_2\lambda_3 T_1) +  T_1 T_2 \prod_{i=1}^3 (\lambda_i -1)  \,,\\[0.2em]
 y & = T_1 T_2 \prod_{i=1}^3 (\lambda_i -1) (T_2 - T_1 \! \prod_{k=1, k\not= i }^3 \! \lambda_k)\,,\\[0.4em]
 z & =  T_2 -\lambda_1\lambda_2\lambda_3 T_1 \;.
\end{array}\right.
\end{equation}

The six nodes $p_{ij}$ with $1\le i < j \le 4$ 
are the singular points of elliptic fibration on $ \bar{\mathcal{Y}\,}\!_1$ in Equation~(\ref{kummer_middle_ell}) given by
\begin{equation}
\label{sings_middle}
  \Big([T_1:T_2],[x:y:z]\Big)  = \Big([1: \lambda_i \lambda_j  ], [\lambda_i+\lambda_j : 0 : 1 ]\Big).
 \end{equation}
Specifically, in the singular fibers $\mathcal{Y}_{1,t_{ij}}$ over the six points $t_{ij}: [T_1:T_2]=[1: \lambda_i \lambda_j]$ for $1\le i < j \le 4$, 
the two-torsion sections $\mathsf{E}_i=\mathsf{E}_j$ coincide at the point $[x:y:z]= [\lambda_i+\lambda_j : 0 : 1 ]$. Each fiber of type $I_2$ over $t_{ij}$ has the form
\begin{equation}
\label{kummer1_ell_sing}
  y^2 =   \lambda_i \lambda_j \,  \left( x  - \big(\lambda_i   + \lambda_j \big) \, z\right)^2 \! \prod_{k=1, k\not= i,j}^4 \left( \lambda_k \, x  - \big(\lambda_k^2   + \lambda_i \lambda_j \big) \, z \right).
\end{equation}
Moreover, the nodes $p_{i6}$ and  $p_{i5}$ for $1\le i \le 4$ are the points of intersection of the sections $\mathsf{E}_i$ with
the fibers of type $I_0^*$ over $[T_1:T_2]=[1:0]$  at $[x:y:z]=[\lambda_i:0:1]$ and over $[T_1:T_2]=[0:1]$  at $[x:y:z]=[1:0:\lambda_i]$, respectively.

In terms of the elliptic fibration~(\ref{kummer_middle_ell}), the 4 points of tangency $p_{ii}$ for $1\le i \le 4$
are given by
\[
  \Big([T_1:T_2],[x:y:z]\Big)  = \Big([1: \lambda_i^2 ], [2 \, \lambda_i : 0 : 1 ]\Big)
 \] 
and are located in the smooth fibers over the points $[T_1:T_2]=[1: \lambda_i^2 ]$ 
on the two-torsion section $\mathsf{E}_i$.  The two remaining points of tangency $p_{55}$ and $p_{66}$ are given by the points $[x:y:z]=[0:0:1]$ 
over $[T_1:T_2]=[1:0]$ and $[T_1:T_2]=[0:1]$, respectively.

\subsubsection{The elliptic fibration under a special projective automorphism}
There is an extra involution acting on the surface~(\ref{kummer_middle_ell}). To see this, one applies fiberwise a   onal change of coordinates.
In each fiber given by Equation~(\ref{kummer_middle_ell}), we apply a linear transformation that maps the points $\mathsf{E}_i: [x:y:z]=[\lambda_i^2 \, T_1  +T_2:0:\lambda_i]$ 
for $1 \le i \le 4$ to the points $\mathsf{E}'_i: [x':y':z']=[\lambda_i' T_1T_2:0:1]$ for $1\le i\le3$ and 
$\mathsf{E}'_4: [x':y':z']=[T_2^2 + l^2 T_1^2:0:1]$. The transformation is given by
 \begin{equation}
 \label{transfo_kummer_middle}
 x'  = z' \, T_1 T_2  \, \frac{A \, x + B\, z}{C \, x + D \, z}\; , \quad
 y'  = (z')^2 \, y \, T_1 T_2 \,  \frac{\prod_{i=1}^3 (\lambda_i-1) (T_2-T_1\prod_{k=1, k\not= i}^{3}\lambda_k)}{(C \, x + D \, z)^2} ,
 \end{equation}
where $A(T_1,T_2), C(T_1,T_2)$ and $B(T_1,T_2), D(T_1,T_2)$ are homogeneous polynomials in $[T_1:T_2]$ of degree $1$ and $2$, respectively, that satisfy
\[
 \det \left( \begin{array}{cc} A & B \\ C & D \end{array} \right) = \prod_{i=1}^3 (\lambda_i-1) (T_2-T_1\!\!\prod_{k=1, k\not= i}^{3}\lambda_k) \;.
 \]
The map~(\ref{transfo_kummer_middle}) transforms Equation~(\ref{kummer_middle_ell}) into the equation
\begin{equation}
\label{kummer_middle_ell_p}
   (y')^2 = \left( x' - \big(T_2^2 + l^2 T_1^2\big) \, z' \right) \,  \prod_{i=1}^3 \left( x'  - \lambda_i' \, l \, T_1 T_2 \, z' \right)\;.
\end{equation}
The projection $([T_1:T_2],[x':y':z']) \mapsto [T_1:T_2]$ 
defines the same elliptic K3 fibration as Equation~(\ref{kummer_middle_ell}). The Mordell-Weil group is now generated by the 4 two-torsion sections $\mathsf{E}'_i$
for $1\le i \le4$ and a section $\mathsf{P}'$ of infinite order given by $[x':y':z']=[1:1:0]$.
It is obvious that Equation~(\ref{kummer_middle_ell_p}) admits the additional involution
\begin{equation}
\label{C2_involutions_middle}
\begin{split}
  \imath^{\mathcal{Y}_1}: \quad  \big([T_1:T_2],[x':y':z']\big)  \mapsto  \big([T_2:l^2 T_1],[l^2 x':l^4y':z']\big) .
\end{split}
\end{equation}
We move the point given by $\mathsf{E}_4'$ to infinity and
convert Equation~(\ref{kummer_middle_ell_p}) to Weierstrass form using a transformation defined over $\mathbb{C}(t)$ with $t=T_2/(l T_1)$. 
After translation by two-torsion the Weierstrass fibration becomes
\begin{equation}
\label{kummer_middle_ell_p_W}
\begin{split}
Y^2  = \; X & \big( X- t \, \left( t^2 - t \, \lambda'_3  +1 \right)   \left( \lambda_1'-\lambda_2' \right) \big) \, \big( X- t \,\left( t^2 - t \, \lambda'_2  +1 \right)    \left( \lambda_1'-\lambda_3' \right) \big) \;,
\end{split}
\end{equation}
with a two-torsion section given by $(X,Y)=(0,0)$ and a distinguished section $\mathsf{S}_0$ given by the point at infinity.
We will also denote this Weierstrass fibration by $Y^2= X^3 - p_2(t) \, t \, X^2+p_1(t)\, t^2 X$ with
\begin{equation}
\label{polynomials_left}
\begin{split}
 p_1(t) &=   \left( \lambda_1'-\lambda_2' \right)  \left( \lambda_1'-\lambda_3' \right) \left( t^2 - t \, \lambda'_2  +1 \right)  \left( t^2 - t \, \lambda'_2  +1 \right)  ,\\
 p_2(t) & = (2\lambda_1'-\lambda_2'-\lambda_3')  \, t^2 - (\lambda_1'(\lambda_2'+\lambda_3') - 2 \lambda_2' \lambda_3') \, t +(2\lambda_1'-\lambda_2'-\lambda_3') \;.
 \end{split}
 \end{equation}
 The discriminant of the elliptic fiber is given by
 \begin{equation}
 \label{discrim_M}
 \begin{split}
  \Delta_{\mathcal{Y}_1} = &\;  \left( \lambda_1'-\lambda_2' \right)^2  \left( \lambda_1'-\lambda_3' \right)^2  \left( \lambda_2'-\lambda_3' \right)^2 \\
  \times & \;  t^6  \left( t^2 - t \, \lambda'_1  +1 \right)^2  \left( t^2 - t \, \lambda'_2  +1 \right)^2 \left( t^2 - t \, \lambda'_3  +1 \right)^2 \;.
 \end{split} 
 \end{equation}
Acting on Equation~(\ref{kummer_middle_ell_p_W}) is the involution $\imath^{\mathcal{Y}_1}$ as
\begin{equation}
\label{C2_involutions_middle_b}
\begin{split}
  \imath^{\mathcal{Y}_1}: \quad  \big(t, X, Y\big)  \mapsto  \left(\frac{1}{t}, \frac{X}{t^4},-\frac{Y}{t^6} \right) ,
\end{split}
\end{equation}
which leaves the holomorphic two-form $dt\wedge dX/Y$ invariant.
 We collect the properties of the Jacobian elliptic K3 fibration in Equation~(\ref{kummer_middle_ell_p_W}) in the following:
\begin{lemma}
\label{lem:EllMiddle}
Equation~(\ref{kummer_middle_ell_p_W}) defines a Jacobian elliptic K3 fibration $\pi^{ \mathcal{Y}_1}:  \bar{\mathcal{Y}\,}\!_1\to \mathbb{P}^1$ of Picard rank $17$ with 
two singular fibers of type $I_0^*$, 6 singular fibers of type $I_2$, a determinant of the discriminant form equal to  $2^6$, and 
a Mordell-Weil group of sections given by $\operatorname{MW}(\pi^{ \mathcal{Y}_1}, \mathsf{S}_0)=(\mathbb{Z}/2)^2\oplus \langle 1 \rangle$. 
\end{lemma}

\begin{remark}
The divisor class of the fiber in Equation~(\ref{kummer_middle_ell_p_W}) is easily found to be $2\mathsf{T}_6 + {P}_{16} +{P}_{26} + {P}_{36} + {P}_{46}$
where ${P}_{ij}$ are the exceptional divisors obtained from resolving the nodes $p_{ij}$ in Table~\ref{tab:nodes}. This corresponds to the divisor class
$H - P_0 -P_{56}$ where $H$ is the hyperplane class. Therefore, in addition to the node $p_0$ the elliptic fibration depends on the choice 
of an additional node, say $p_{56}$, for which there are $15$ choices.
\end{remark}

The fibration in Lemma~\ref{lem:EllMiddle} is the elliptic fibration {\tt (1)} in the list of all possible elliptic fibrations on a generic Kummer surface determined 
by Kumar in~\cite[Thm.~2]{MR3263663}. We have the following corollary:
\begin{corollary}
\label{cor:genus_two}
The K3 surface $\mathcal{Y}_1$  is the Kummer surface $\operatorname{Kum}(\operatorname{Jac} \mathcal{C})$ of the principally polarized abelian surface
$\operatorname{Jac}(\mathcal{C})$ with $\mathcal{C}$ given in Equation~(\ref{Eq:Rosenhain_g2}).
\end{corollary}

The additional involution in Equation~(\ref{C2_involutions_middle_b}) comes from a special birational automorphism known as Kummer translation \cite{MR1458752}.
 We have the following:
\begin{proposition}
\label{Lem:BT_Y1}
The Nikulin involution $\imath^{\mathcal{Y}_1}$ in Equation~(\ref{C2_involutions_middle}) descends from the
translation of $\mathbf{A}=\operatorname{Jac}(\mathcal{C})$ by the two-torsion point $\mathfrak{p}_{56}=\mathfrak{e}_{12} \in \mathbf{A}[2]$
to an automorphism of $\mathcal{Y}_1$. In particular, the translation is linear, i.e.,
induced by a projective automorphism of $\mathbb{P}^3$.
\end{proposition}
\begin{proof}
It is easy to check that the involution $\imath^{\mathcal{Y}_1}$ in Equation~(\ref{C2_involutions_middle})
interchanges the roots of each factor $t^2 - t \, \lambda'_i  +1$ for $1\le i \le 3$ in the discriminant or, equivalently,
the singularities $p_{ij}$ and $p_{k4}$ in Equation~(\ref{sings_middle}) with $\{i,j,k\}=\{1,2,3\}$.
Similarly the involution interchanges the components of the $I_0^*$-fibers, i.e., $p_{k5}$ and $p_{k6}$ for $k=1, \dots, 4$.
From the group law for the order-two points in Equation~(\ref{group_law}) it follows that the automorphism corresponds to translation
by $p_{56}$. We have already checked that $\imath^{\mathcal{Y}_1}$ leaves the holomorphic two-form $dt\wedge dX/Y$ invariant.
\end{proof}
 
\subsubsection{A Van Geemen-Sarti involution}
The translation by the two-torsion section $(X,Y)=(0,0)$ in Equation~(\ref{kummer_middle_ell_p_W}) is a Nikulin involution $\imath^{\mathcal{Y}_1}_F$ on $\mathcal{Y}_1$,
which for $X\not = 0$ is given by
\begin{equation}
\label{VGS_involution_middle}
 \imath^{\mathcal{Y}_1}_F: \; (X,Y)  \mapsto  (X',Y') = (X,Y) \overset{.}{+}  (0,0) = \left( \frac{p_1(t)\, t^2}{X}, - \frac{p_1(t)\, t^2 Y}{X^2} \right)
\end{equation}
and leaves the holomorphic two-form $dt\wedge dX/Y$ invariant.
We obtain a new K3 surface $\mathcal{X}_1$ by resolving the eight nodes of  $\mathcal{Y}_1/\imath^{\mathcal{Y}_1}_F$. In fact, 
fiberwise translations by a section of order two 
in a Jacobian elliptic fibration is also known as \emph{Van Geemen-Sarti involution}.
The isogeny  $\phi^{\mathcal{Y}_1}_{\mathcal{X}_1}: \mathcal{Y}_1 \dashrightarrow \mathcal{X}_1$ is the rational quotient map  
with a kernel generated by the sections given by $(X,Y)=(0,0)$ and the point at infinity.  We obtain for $\bar{\mathcal{X}\,}\!_1$ the Weierstrass model 
\begin{equation}
\label{kummer_middle_ell_dual_W}
 y^2 = x  \, \Big( x^2 +2\, p_2(t) \, t \, x  + \left(p_2^2(t) - 4 \, p_1(t)\right)  t^2 \Big)  \;,
\end{equation}
with a two-torsion section given by $(x,y)=(0,0)$, a distinguished section $\mathsf{s}_0$ given by the point at infinity, and
a discriminant given by
 \begin{equation}
 \label{discrim_MU}
 \begin{split}
  \Delta_{\mathcal{X}_1} =  \; & 16 \, \left( \lambda_1'-\lambda_2' \right)  \left( \lambda_1'-\lambda_3' \right)  \left( \lambda_2'-\lambda_3' \right)^4\\
  \times \; & t^6  \left( t^2 - t \, \lambda'_1  +1 \right)^4  \left( t^2 - t \, \lambda'_2  +1 \right) \left( t^2 - t \, \lambda'_3  +1 \right) \;.
 \end{split} 
 \end{equation}
The explicit formulas for the isogeny and the dual isogeny of a Van Geemen-Sarti involution are well known and given by
\begin{equation}
\label{Eq:isogeny}
{\phi}^{\mathcal{Y}_1}_{\mathcal{X}_1}: \quad \bar{\mathcal{Y}\,}\!_1 \dashrightarrow \bar{\mathcal{X}\,}\!_1\,, \quad  (X,Y) \mapsto \left( \frac{Y^2}{X^2}, \frac{Y \, \big(p_1(t)\, t^2 - X^2\big)}{X^2}\right)\,,
\end{equation}
and
\begin{equation}
\label{eqn:dual_isog}
{\phi}^{\mathcal{X}_1}_{\mathcal{Y}_1}:  \quad \bar{\mathcal{X}\,}\!_1 \dashrightarrow \bar{\mathcal{Y}\,}\!_1\,, \quad  (x,y) \mapsto \left( \frac{y^2}{4 \,x^2}, \frac{y \, \big((p_2^2(t) -4 \, p_1(t)) \, t^2- x^2 \big)}{8 \, x^2}\right)\,.
\end{equation}
The translation by the two-torsion section $(x,y)=(0,0)$ in Equation~(\ref{kummer_middle_ell_dual_W}) is a Nikulin involution $\imath^{\mathcal{X}_1}_F$ on $\mathcal{X}_1$
covering ${\phi}^{\mathcal{X}_1}_{\mathcal{Y}_1}$.
We have the following:
\begin{lemma}
\label{lem:EllMiddleTop}
Equation~(\ref{kummer_middle_ell_dual_W}) defines a  Jacobian elliptic K3 fibration $\pi^{ \mathcal{X}_1}:  \bar{\mathcal{X}\,}\!_1\to \mathbb{P}^1$ of Picard rank 17
with two singular fibers of type $I_0^*$, two fibers of type $I_4$, four fibers of type $I_1$, a determinant of the discriminant form equal to $2^5$,
and a Mordell-Weil group of sections $\operatorname{MW}(\pi^{ \mathcal{X}_1}, \mathsf{s}_0)=\mathbb{Z}/2 \oplus \langle \frac{1}{2} \rangle$.
\end{lemma}

The dual fiberwise isogeny ${\phi}^{\mathcal{X}_1}_{\mathcal{Y}_1}$ has a geometric interpretation: 
in the reducible fibers of the Jacobian elliptic K3 surface $\mathcal{Y}_1$ 
eight disjoint smooth rational curves $\lbrace C_1, \dots, C_8\rbrace$ consisting only of exceptional divisors, or nodes, can be identified that form an even eight $\Delta = C_1 + \dots + C_8 
\in 2 \, \operatorname{NS}(\mathcal{Y}_1)$.  One then obtains the surface $\mathcal{X}_1$ as double cover of $\mathcal{Y}_1$ branched along $\Delta$ 
after the inverse images of the curves $C_i$ for $1\le i \le 8$ are blown down. That is, the double branched cover map is precisely the dual two-isogeny
$\phi^{\mathcal{X}_1}_{\mathcal{Y}_1}: \mathcal{X}_1 \dashrightarrow\mathcal{Y}_1$.

Mehran labels the fifteen different even eights consisting of exceptional curves (up to taking complements) on the Kummer surface 
$\operatorname{Kum}(\operatorname{Jac} \mathcal{C})$ of the Jacobian of a generic genus-two curve $\mathcal{C}$ 
by $\Delta_{ij}$ with $1 \le i < j \le 6$, and proves that these 15 even eights give rise to all 15 isomorphism classes of 
rational double 
covers of $\operatorname{Kum}(\operatorname{Jac} \mathcal{C})$ with transcendental lattice $H(2)\oplus H(2) \oplus \langle -2 \rangle$  \cite[Prop.~4.2]{MR2804549}.
The even eights $\Delta_{ij}$ are given by 
\[
 \Delta_{ij} = E_{1i} + \dots + \widehat{E_{ij}} + \dots + E_{i6} + E_{1j} + \dots + \widehat{E_{ij}} + \dots + E_{j6} \;,
 \]
where $E_{11}=0$, and $E_{ij}$ are the exceptional divisor obtained by resolving the nodes $e_{ij}$ in Table~\ref{tab:nodes}.
Moreover, Mehran proved that each rational map $\phi_{ij}: \operatorname{Kum}(\mathbf{B}_{ij}) \dashrightarrow \mathcal{Y}_1=\operatorname{Kum}(\operatorname{Jac} \mathcal{C})$ 
is induced by an isogeny $\varphi_{ij}: \mathbf{B}_{ij} \to \operatorname{Jac}(\mathcal{C})$ of abelian surfaces of degree two~\cite{MR2804549}. 
It follows that the surfaces $\mathbf{B}_{ij}$ are $(1,2)$-polarized abelian surfaces.

Using Mehran's notation we have the following:

\begin{lemma}
The K3 surface $\mathcal{X}_1$  is the Kummer surface $\operatorname{Kum}(\mathbf{B}_{34})$ of a $(1,2)$-polarized abelian surface
$\mathbf{B}_{34}$ that covers $\varphi_{34}: \mathbf{B}_{34} \to \operatorname{Jac}(\mathcal{C})$ by an isogeny of degree two
such that the induced rational map $\phi_{34}: \mathcal{X}_1 = \operatorname{Kum}(\mathbf{B}_{34}) \dashrightarrow \mathcal{Y}_1=\operatorname{Kum}(\operatorname{Jac} \mathcal{C})$
associated with the even eight $\Delta_{34}$ realizes the fiberwise two-isogeny $\phi^{\mathcal{X}_1}_{\mathcal{Y}_1}: \mathcal{X}_1 \dashrightarrow\mathcal{Y}_1$ 
in Equation~(\ref{eqn:dual_isog}).
\end{lemma}

\begin{proof}
We show that an even eight $\Delta $ can be constructed within the reducible fibers of $\mathcal{Y}_1$ such that
the double cover branched along the even eight realizes the fiberwise two-isogeny. Therefore,
the singular fibers $6 I_2 + 2I_0^*$ must be mapped to the singular fibers $2 I_4 + 4 I_1 + 2I_0^*$. To turn a fiber of type $I_0^*$ into another fiber of type $I_0^*$
by double cover, two rational curves different from the central fiber must be chosen to be part of the branch locus. The remaining four disjoint rational curves of the even eight
must be picked from the six fibers of type $I_2$.  The four reducible fibers of type $I_2$ with marked components are then mapped to fibers of type $I_1$, whereas
in the remaining two fibers of type $I_2$ the components are doubled turning them into fibers of type $I_4$. 

From comparing the discriminants in Equation~(\ref{discrim_M}) and Equation~(\ref{discrim_MU}), we observe that the reducible fibers over $t^2 - t \, \lambda'_2  +1$
and $t^2 - t \, \lambda'_3  +1$ change from type $I_2$ to $I_1$. Therefore, the reducible components in these $I_2$-fibers make up four components of 
the even eight. Using Table~\ref{tab:nodes}, it follows that in the notation of Mehran~\cite{MR2804549} these components are the divisors 
$E_{35}=P_{24}, E_{46}=P_{13}$ and $E_{36}=P_{34}, E_{45}=P_{12}$
which are all nodes. Therefore, the even eight is $\Delta_{34}$ which also contains the divisors $E_{13}=P_{46}, E_{14}=P_{16}$ and $E_{23}=P_{45}, E_{24}=P_{15}$.
It is easy to check that these divisors are in fact rational components of the fibers of type $I_0^*$ over $t=0$ and $t=\infty$, respectively.
\end{proof}

\begin{proposition}
\label{Lem:VGSI_Y1}
The Van Geemen-Sarti involution $\imath^{\mathcal{Y}_1}_F$ in Equation~(\ref{VGS_involution_middle}) descends from the 
translation of $\mathbf{A}=\operatorname{Jac}(\mathcal{C})$ by the two-torsion point $\mathfrak{p}_{14}=\mathfrak{e}_{34} \in \mathbf{A}[2]$
to an automorphism of $\mathcal{Y}_1$. 
\end{proposition}
\begin{proof}
The involution must map the set of nodes given by $\{E_{35}, E_{46}, E_{36}, E_{45}\}=\{ P_{12}, P_{13}, P_{24}, P_{34}\}$
into itself. Moreover, it must interchange the two marked components in each fiber of type $I_0^*$, that is
$E_{13}=P_{46} \leftrightarrow E_{14} = P_{16}$ and $E_{23}=P_{45} \leftrightarrow E_{24}=P_{15}$. It follows that the involution is induced by
translation by $p_{14}$.
\end{proof}

\subsubsection{The action of $(2,2)$-isogeny}
\label{ssec:action_isog}
A $(2,2)$-isogeny $\psi: \operatorname{Jac}(\mathcal{C}) \to \operatorname{Jac}(\hat{\mathcal{C}\,}\!)$ induces an algebraic map
between the corresponding Kummer surfaces  \cite[Sec.~3]{MR1406090} which we denote by 
\begin{equation}
\label{2_2_isog}
  \Psi: \quad \operatorname{Kum}(\operatorname{Jac} \mathcal{C}) \dashrightarrow   \operatorname{Kum}(\operatorname{Jac} \hat{\mathcal{C}\,}\!).
\end{equation}
For $\hat{\mathcal{C}\,}\!=\hat{\mathcal{C}\,}\!_{12}$ the morphism $\Psi$ is realized by the change in modular parameters derived in Section~\ref{ssec:2isog}.
Since the Weierstrass equation~(\ref{kummer_middle_ell_p_W}) only depends on $\{\lambda_1', \lambda_2', \lambda_2' \}$ it follows that $\Psi_{12}:
\mathcal{Y}_1 \to \hat{\mathcal{Y}\,}\!_1=\operatorname{Kum}(\operatorname{Jac} \hat{\mathcal{C}\,}\!_{12}),$ preserves the elliptic fibration in Lemma~\ref{lem:EllMiddle}. Therefore, for the even eight $\Delta_{12}$ on $\mathcal{Y}_1$, one 
naturally obtains an even eight  $\hat{\Delta}_{12}$
on $\hat{\mathcal{Y}\,}\!_1$, and in turn a $(1,2)$-polarized abelian surface $\hat{\mathbf{B}}_{12}$ with  two-isogeny $\hat{\varphi}_{12}: \hat{\mathbf{B}}_{12} \to \operatorname{Jac} \hat{\mathcal{C}\,}\!_{12}$
inducing a rational map $\hat{\phi\,}\!_{12}: \operatorname{Kum}(\hat{\mathbf{B}}_{12}) \dashrightarrow \hat{\mathcal{Y}\,}\!_1$.
Using the fact that the rational map $\hat{\phi}_{12}$ realizes a fiberwise isogeny with a dual isogeny $\phi^{\hat{\mathcal{Y}\,}\!_1}_{\hat{\mathcal{X}\,}\!_1}$, we obtain an algebraic map
\[
  \Psi'_{12}= \phi^{\hat{\mathcal{Y}\,}\!_1}_{\hat{\mathcal{X}\,}\!_1} \circ \Psi_{12} \circ \phi^{\mathcal{X}_1}_{\mathcal{Y}_1}: \quad \mathcal{X}_1 = \operatorname{Kum}(\mathbf{B}_{12}) \dashrightarrow  \hat{\mathcal{X}\,}\!_1 = \operatorname{Kum}(\hat{\mathbf{B}}_{12}) 
\]
preserving the Jacobian elliptic fibration in Lemma~\ref{lem:EllMiddleTop}

\subsection{Kummer surface from the genus-three curve $\mathcal{D}$}
We start with two copies of the genus-three curve in Equation~(\ref{Eq:Rosenhain_g3}) and the symmetric product $\mathcal{D}^{(2)} = (\mathcal{D}\times\mathcal{D})/\langle \sigma_{\mathcal{D}^{(2)} } \rangle$.
The variety $\mathcal{D}^{(2)}/\langle \imath^\mathcal{D}_b \times  \imath^\mathcal{D}_b \rangle$ is given in terms of the variables
$[Z_1:Z_2:Z_3:\tilde{Z}_4] \in \mathbb{WP}(1,2,1,4)$ with
\begin{equation}
\label{transfo_kummer2}
 Z_1=z^{(1)}z^{(2)}, \;Z_2=( x^{(1)} z^{(2)} )^2+(x^{(2)}z^{(1)})^2, \;Z_3=x^{(1)}x^{(2) }, \;\tilde{Z}_4=y^{(1)}y^{(2)}
 \end{equation}
by the equation
\begin{equation}
\label{kummer3}
  \tilde{Z}_4^2 =  \prod_{i=1}^4 \big( \lambda_i^2 \, Z_1^2  -  \lambda_i \, Z_2 +  Z_3^2 \big) \;.
\end{equation}

We determine the singular points on the octic surface given by Equation~(\ref{kummer3}). 
There are 12 singular points $q_{ij}^{\pm}$ of Equation~(\ref{kummer3}) with $1 \le i < j \le 4$ given by
$$
 [Z_1: Z_2 : Z_3 : \tilde{Z}_4] = [1: \lambda_i+\lambda_j : \pm \sqrt{\lambda_i \lambda_j }: 0] .
$$
These points are obtained by combining the Weierstrass points on $\mathcal{D}$ given by
\begin{equation}
\label{pencil_DP}
[x^{(1)}:y^{(1)}:z^{(1)}]=[\pm \sqrt{\lambda_i}: 0 : 1] \; \text{and} \; [x^{(2)}:y^{(2)}:z^{(2)}]=[\pm \sqrt{\lambda_j}: 0 : 1] 
\end{equation}
in Equation~(\ref{transfo_kummer2}).
Next we use the two pairs of bielliptic fixed points on $\mathcal{D}$ given by $[0:\pm l:1]$ and $[1:\pm1:0]$.
Combining Weierstrass points with a pair of bielliptic fixed points on $\mathcal{D}$ in Equation~(\ref{transfo_kummer2}), i.e., 
\begin{equation}
\label{pencil_DP2a}
[x^{(1)}:y^{(1)}:z^{(1)}]=[\pm \sqrt{\lambda_i}: 0 : 1] \; \text{and} \; [x^{(2)}:y^{(2)}:z^{(2)}]=[1: \pm 1 : 0] ,
\end{equation}
gives the four singular points $q_{i6}$ for $1 \le i \le 4$ with
$$
 [Z_1: Z_2 : Z_3 : \tilde{Z}_4] = [0: 1:  \sqrt{\lambda_i}: 0] .
$$
Similarly, combining Weierstrass points with the second pair of bielliptic fixed points on $\mathcal{D}$ in Equation~(\ref{transfo_kummer2}), i.e., 
\begin{equation}
\label{pencil_DP2b}
[x^{(1)}:y^{(1)}:z^{(1)}]=[\pm \sqrt{\lambda_i}: 0 : 1] \; \text{and} \; [x^{(2)}:y^{(2)}:z^{(2)}]=[0: \pm l: 1] ,
\end{equation}
gives the four singular points $q_{i5}$ for $1 \le i \le 4$ with
$$
 [Z_1: Z_2 : Z_3 : \tilde{Z}_4] = [1: \lambda_i: 0 : 0] .
$$

\subsubsection{An elliptic fibration}
\label{ssec:ell_fib_X2}
Using the rational transformation $[Z_1:Z_2:Z_3:\tilde{Z}_4]=[S_1Z: XZ: S_2 Z: Y Z^2]$ with $[(S_1, S_2) , (X,Y,Z)] \in \mathbb{F}^{4,2}_{-2}$,
Equation~(\ref{kummer3}) becomes the equation of a genus-one fibration $\bar{\mathcal{X}\,}\!_2$ over $\mathbb{P}^1$ given by
\begin{equation}
\label{kummer_left}
   Y^2 =   \prod_{i=1}^4 \left( \lambda_i \, X  - \big(\lambda_i^2 \, S_1^2  +S_2^2 \big) \, Z \right)\;.
\end{equation}
The projection $\pi^{ \mathcal{X}_2}:  \bar{\mathcal{X}\,}\!_2\to \mathbb{P}^1$ with $([S_1:S_2],[X:Y:Z]) \mapsto [S_1:S_2]$ 
defines a K3 fibration with twelve fibers of Kodaira-type $I_2$. 
Four sections are given by  $\mathsf{e}_i: [X:Y:Z]=[\lambda_i^2 \, S_1^2  +S_2^2:0:\lambda_i]$ for $1\le i \le4$. 
We choose the point given by $\mathsf{e}_4$ to be the neutral element of the Mordell-Weil group of sections,
turning Equation~(\ref{kummer_left}) into a Jacobian elliptic fibration. The Mordell-Weil group contains the 4 two-torsion sections 
$\mathsf{e}_i$ for $1\le i \le4$ and three sections $\mathsf{p}, \mathsf{q}, \mathsf{r}$ of infinite order that are given by
\begin{equation}
\mathsf{p}: \left\lbrace\begin{array}{rl}
 X & = (S_1^2 + S_2^2) ( S_2^2 -\lambda_1\lambda_2\lambda_3 S_1^2) +  S_1^2 S_2^2 \prod_{i=1}^3 (\lambda_i -1)  \,,\\[0.2em]
 Y & =  S_1 S_2 \prod_{i=1}^3 (\lambda_i -1) (S^2_2 - S^2_1 \! \prod_{k=1, k\not= i }^3 \! \lambda_k)\,,\\[0.4em]
 Z & =  S^2_2 -\lambda_1\lambda_2\lambda_3 S^2_1 \;,
\end{array}\right.
\end{equation}
$\mathsf{q}: [X:Y:Z]=[1: l :0]$, and
\begin{equation}
\mathsf{r}: \left\lbrace \begin{array}{rl}
 X & = (S_1^2 + S_2^2) ( S^4_2 -\lambda_1\lambda_2\lambda_3 S^4_1) -  \prod_{i=1}^3 (S^2_2 - \lambda_i S^2_1)  \,,\\[0.2em]
 Y & =  \prod_{i=1}^3 (S^2_2 - \lambda_i S^2_1) (S^2_2 - S^2_1 \! \prod_{k=1, k\not= i }^3 \! \lambda_k)\,,\\[0.4em]
 Z & =  S_2^4 -\lambda_1\lambda_2\lambda_3 S_1^4 \;.
\end{array}\right.
\end{equation}
The sections $\mathsf{p}$ and $\mathsf{r}$ can be decomposed further using the two minimal
sections of infinite order given by $\mathsf{s}_{\pm}: [X:Y:Z]=[\mp 2S_1 S_2: \prod_{i=1}^4 ( \lambda_i S_1 \pm S_2) : 1]$ 
such that $\mathsf{p}=\mathsf{s}_+ + \mathsf{s}_-$ and $\mathsf{r}=\mathsf{s}_+ - \mathsf{s}_-$. It is obvious that Equation~(\ref{kummer_left}) admits the involution
\begin{equation}
\label{j_involution_left}
\begin{split}
  \jmath^{\mathcal{X}_2}: \quad \big([S_1:S_2],[X:Y:Z]\big)  \mapsto  \big([-S_1:S_2],[X:-Y:Z]\big) .
\end{split}
\end{equation}
The involution acts on sections according to relations 
$\jmath^{\mathcal{X}_2}(\mathsf{p}, \mathsf{e}_i)= \mathsf{p}, \mathsf{e}_i$, $\jmath^{\mathcal{X}_2}( \mathsf{q}, \mathsf{r})= -\mathsf{q}, -\mathsf{r}$,
and $\jmath^{\mathcal{X}_2}(\mathsf{s}_{\pm})= \mathsf{s}_{\mp}$.

The twelve points $q_{ij}^{\pm}$ with $1\le i < j \le 4$ 
are the singular points of elliptic fibration on $\mathcal{X}_2$ in Equation~(\ref{kummer_middle_ell}), namely the points
\begin{equation}
\label{sings_middle_b}
  \Big([S_1:S_2],[X:Y:Z]\Big)  = \Big([1: \pm \sqrt{\lambda_i \lambda_j } ], [\lambda_i+\lambda_j : 0 : 1 ]\Big).
\end{equation}
In the singular fibers $\mathcal{X}_{2,s^{\pm}_{ij}}$ over the twelve points $s^\pm_{ij}: [1: \pm \sqrt{\lambda_i \lambda_j } ]$ for $1\le i < j \le 4$, 
the two-torsion sections $\mathsf{e}_i=\mathsf{e}_j$ coincide at the point $[X:Y:X]= [\lambda_i+\lambda_j : 0 : 1 ]$. 
The fiber of type $I_2$ over $s^{\pm}_{ij}$ has the form
\begin{equation}
\label{kummer1_ell_sing_b}
  Y^2 =   \lambda_i \lambda_j \,  \left( X  - \big(\lambda_i   + \lambda_j \big) \, Z\right)^2 \! \prod_{k=1, k\not= i,j}^4 \left( \lambda_k \, X  - \big(\lambda_k^2   + \lambda_i \lambda_j \big) \, Z \right)\;.
\end{equation}

\subsubsection{The elliptic fibration under a special projective automorphism}
There is an additional involution acting on the elliptic fibration~(\ref{kummer_left}). To see this we will apply fiberwise a birational change of coordinates.
In each fiber of Equation~(\ref{kummer_left}) we apply a linear transformation that maps the points 
$\mathsf{e}_i: [X:Y:Z]=[\lambda_i^2 \, S^2_1  +S^2_2:0:\lambda_i]$ for $1 \le i \le 4$
to the points $\mathsf{e}'_i: [X':Y':Z']=[\lambda_i' S^2_1S^2_2:0:1]$ for $1\le i\le3$ and $\mathsf{e}'_4: [X':Y':Z']=[S_2^4 + \lambda_1\lambda_2\lambda_3 S_1^4:0:1]$.
The transformation is given by
 \begin{equation}
 \label{transfo_kummer_left}
 X'  = Z'   \, \frac{a \, X + b\, Z}{c \, X + d \, Z}\; , \quad
 Y'  = (Z')^2 \, Y \, S^2_1 S^2_2 \, \frac{\prod_{i=1}^3 (\lambda_i-1) (S^2_2-S^2_1\prod_{k=1, k\not= i}^{3}\lambda_k)}{(c \, X + d\, Z)^2} ,
 \end{equation}
where the relation to the transformation in Equation~(\ref{transfo_kummer_middle}) is given by $a=A(S^2_1,S^2_2)$ and similar for $b, c, d$.
It transforms Equation~(\ref{kummer_left}) into the equation
\begin{equation}
\label{kummer_left_ell_p}
   (Y')^2 = \left( X' - \big(S_2^4 + l^2 S_1^4\big) \, Z' \right) \,  \prod_{i=1}^3 \left( X'  - \lambda_i' \, l \, S_1^2 \, S_2^2 \, Z' \right)\;.
\end{equation}
The projection $([S_1:S_2],[X':Y':Z']) \mapsto [S_1:S_2]$ 
defines the same elliptic K3 fibration as Equation~(\ref{kummer_left}). The Mordell-Weil group contains the 4 two-torsion sections $\mathsf{e}'_i$
for $1\le i \le4$ and sections $\mathsf{p}', \mathsf{q}', \mathsf{r}'$ obtained from $\mathsf{p}, \mathsf{q}, \mathsf{r}$ by transformation~(\ref{transfo_kummer_left}).
It is obvious that Equation~(\ref{kummer_left_ell_p}) admits the additional involution
\begin{equation}
\label{involution_left}
\begin{split}
  \imath^{\mathcal{X}_2}: \quad \big([S_1:S_2],[X':Y':Z']\big)  \mapsto  \big([ S_2: l \, S_1],[l^2 \, X': l^4 \, Y':Z']\big) .
\end{split}
\end{equation}
We move the point given by $\mathsf{e}_4'$ to infinity and
convert Equation~(\ref{kummer_left_ell_p}) to Weierstrass form using a transformation defined over $\mathbb{C}(s)$ with $s=S_2/ \sqrt{l} \, S_1$. 
After translation by two-torsion the Weierstrass fibration is
\begin{equation}
\label{kummer_left_ell_p_W}
\begin{split}
y^2  = \; x \; & \big( x- \left( s^4 - s^2 \, \lambda'_2  +1 \right)   \left( \lambda_1'-\lambda_2' \right) \big) \, \big( x- \left( s^4 - s^2 \, \lambda'_2  +1 \right)    \left( \lambda_1'-\lambda_3' \right) \big) \;, 
\end{split}
\end{equation}
with two-torsion section $(x,y)=(0,0)$ and a distinguished section $\mathsf{s}'_0$ given by the point at infinity.
We will write the Weierstrass fibration as $y^2= x^3 - p_2(s^2) \, x^2+p_1(s^2)\,  x$ with
polynomials $p_1, p_2$ defined in Equations~(\ref{polynomials_left}).
Acting on Equation~(\ref{kummer_left_ell_p_W}) are the involutions $\imath^{\mathcal{X}_2}$ and $\jmath^{\mathcal{X}_2}$ as
\begin{equation}
\label{C2_involutions_left_b}
\begin{split}
  \imath^{\mathcal{X}_2}: \big(s, x, y\big)  \mapsto  \left(\frac{1}{s}, \frac{x}{s^4}, -\frac{y}{s^6} \right) , \quad
  \jmath^{\mathcal{X}_2}: \big(s, x, y\big)  \mapsto  \big(-s, x, -y\big) ,
\end{split}
\end{equation}
which leave the holomorphic two-form $ds\wedge dx/y$ invariant.
  The discriminant of the elliptic fiber is given by
 \begin{equation}
 \label{discrim_L}
 \begin{split}
  \Delta_{\mathcal{X}_2} =  \; & \left( \lambda_1'-\lambda_2' \right)^2  \left( \lambda_1'-\lambda_3' \right)^2  \left( \lambda_2'-\lambda_3' \right)^2\\
  \times \; &  \left( s^4 - s^2 \, \lambda'_1  +1 \right)^2  \left( s^4 - s^2 \, \lambda'_2  +1 \right)^2 \left( s^4 - s^2 \, \lambda'_3  +1 \right)^2 \;.
 \end{split} 
 \end{equation}
 We collect the properties of the Jacobian elliptic K3 fibration in Equation~(\ref{kummer_left_ell_p_W}) in the following:
\begin{lemma}
\label{lem:EllLeft}
Equation~(\ref{kummer_left_ell_p_W}) defines a  Jacobian elliptic K3 fibration $\pi^{ \mathcal{X}_2}:  \bar{\mathcal{X}\,}\!_2\to \mathbb{P}^1$ of Picard rank $17$ with 
12 singular fibers of type $I_2$, a determinant of the discriminant form equal to $2^5$, and a Mordell-Weil group of sections 
$\operatorname{MW}(\pi^{ \mathcal{X}_2},\mathsf{s}'_0)=(\mathbb{Z}/2)^2 \oplus \langle \frac{1}{2} \rangle^{\oplus 3}$. 
\end{lemma}
\emptyproof

\subsubsection{A Nikulin involution by base transformation}
The relation between the K3 surfaces given by Equation~(\ref{kummer_middle}) and Equation~(\ref{kummer_left}) is described as follows:
the rational map 
\begin{equation}
\label{cover_map}
\begin{split}
  \pi^{\mathcal{D}^{(2)}_b}_{\mathcal{C}^{(2)}}: \quad \mathcal{D}^{(2)}/\langle \imath^\mathcal{D}_b \times  \imath^\mathcal{D}_b \rangle 
   \dashrightarrow  & \; \mathcal{C}^{(2)}/\langle \imath^\mathcal{C}_h \times  \imath^\mathcal{C}_h \rangle \\
 [Z_1:Z_2:Z_3:\tilde{Z}_4] \mapsto  & \; [z_1:z_2:z_3:\tilde{z}_4]=[Z_1^2:Z_2:Z_3^2:Z_1Z_3\tilde{Z}_4] \;,
\end{split}
\end{equation}
induces a rational map between the corresponding Jacobian elliptic K3 surfaces in Equation~(\ref{kummer_middle_ell_p_W}) and Equation~(\ref{kummer_left_ell_p_W}) given by
\begin{equation}
\label{eqn:double_cover}
\begin{split}
  \phi^{\mathcal{X}_2}_{\mathcal{Y}_1}: \quad \bar{\mathcal{X}\,}\!_2  \dashrightarrow  & \; \bar{\mathcal{Y}\,}\!_1, \quad
  \Big(s, x, y\Big) \mapsto  \Big(t,  X, Y\Big)= \Big(s^2,  s^2x, s^3y\Big).
\end{split}
\end{equation}
The morphism $\phi^{\mathcal{X}_2}_{\mathcal{Y}_1}$ maps sections according to $\phi^{\mathcal{X}_2}_{\mathcal{Y}_1} \circ  \mathsf{e}'_i = \mathsf{E}'_i$ for $1 \le i \le4$
and the same relation holds for the sections $\mathsf{p}'$ and $\mathsf{P}'$. 
The involution $\jmath^{\mathcal{X}_2}$ interchanges the sheets of $\phi^{\mathcal{X}_2}_{\mathcal{Y}_1}$,
and $\imath^{\mathcal{X}_2}$ acts equivariantly, i.e., $\phi^{\mathcal{X}_2}_{\mathcal{Y}_1}\circ \imath^{\mathcal{X}_2} = \imath^{\mathcal{Y}_1} \circ \phi^{\mathcal{X}_2}_{\mathcal{Y}_1}$. Using Mehran's notation and \cite[Prop.~5.1]{MR2804549}, we have the following:

\begin{lemma}
\label{Kum_B_12}
The K3 surface $\mathcal{X}_2$  is the Kummer surface $\operatorname{Kum}(\mathbf{B}_{12})$ of a $(1,2)$-polarized abelian surface
$\mathbf{B}_{12}$ that covers $\varphi_{12}: \mathbf{B}_{12} \to \operatorname{Jac} (\mathcal{C})$ by an isogeny of degree two
such that the induced rational map $\phi_{12}: \mathcal{X}_2 = \operatorname{Kum}(\mathbf{B}_{12}) \dashrightarrow \mathcal{Y}_1=\operatorname{Kum}(\operatorname{Jac} \mathcal{C})$
associated with the even eight $\Delta_{12}$ realizes the double cover $\phi^{\mathcal{X}_2}_{\mathcal{Y}_1}: \mathcal{X}_2 \dashrightarrow\mathcal{Y}_1$ in Equation~(\ref{eqn:dual_isog}).
\end{lemma}

\begin{proof}
The double cover in Equation~(\ref{eqn:double_cover}) has the ramification points $t=0$ and $t=\infty$.
These are the base points of fibers of type $I_0^*$ in Lemma~(\ref{lem:EllMiddle}). It is easy to check
that the disjoint rational components in these reducible fibers are -- in the notation of Mehran --  given by
$E_{13}, E_{14}, E_{15}, E_{16}$ and $E_{23}, E_{24}, E_{25}, E_{26}$, respectively.
These form the even eight $\Delta_{12}$.
\end{proof}

We have the following result:
\begin{theorem}
\label{prop1}
The Kummer surface $\operatorname{Kum}(\mathbf{B}_{12})$ of a $(1,2)$-polarized abelian surface
$\mathbf{B}_{12}$ that covers $\varphi_{12}: \mathbf{B}_{12} \to \operatorname{Jac} (\mathcal{C})$ by an isogeny of degree two
is birationally equivalent to quotient of the symmetric square of a bielliptic and hyperelliptic genus-three curve $\mathcal{D}$ 
that is the double cover of a generic genus-two curve $\mathcal{C}$.
\end{theorem}
\begin{proof}
The statement follows from Lemma~\ref{Kum_B_12} and Equations~(\ref{transfo_kummer2}) and~(\ref{kummer3}).
\end{proof}

\begin{remark}
\label{Gab1}
In~\cite{MR1871336} the authors constructed a family of K3 surfaces with Picard number fourteen whose general member $\mathcal{X}$ 
admits an elliptic fibration with twelve reducible fibers of type $I_2$ and four two-torsion sections, and has $\operatorname{T}(\mathcal{X})=H(2)^2\oplus \langle -2 \rangle^{\oplus 4}$ and $\operatorname{NS}(\mathcal{X})=D_4(-1)^{\oplus 3}\oplus \langle 2 \rangle \oplus \langle -2 \rangle$. 
The general member is the double cover with a branch locus that is the union of four divisors of bidegree $(1,1)$ 
in $\mathbb{P}^1 \times \mathbb{P}^1$.  The Kummer surface $\operatorname{Kum}(\mathbf{B}_{12})$ is a specialization of this family obtained 
in Equation~(\ref{kummer_left}) requiring that the elliptic fibration admits three more infinite-order sections with certain intersection properties, 
namely the sections $\mathsf{q}, \mathsf{s}_{\pm}$.
\end{remark}

\subsubsection{A Van Geemen-Sarti involution}
The translation by the two-torsion section $(x,y)=(0,0)$ in Equation~(\ref{kummer_left_ell_p_W}) is a Nikulin involution $\imath_F$, and we obtain a K3 surface $\mathcal{Y}_2$ by resolving the eight nodes of 
$\mathcal{X}_2/\imath_F$. The isogeny $\phi^{\mathcal{X}_2}_{\mathcal{Y}_2}: \mathcal{X}_2 \dashrightarrow \mathcal{Y}_2$ is the rational quotient map over 
$\mathbb{C}(t)$ with a kernel generated by the section made from $(x,y)=(0,0)$ and the point at infinity. 
We obtain for $\bar{\mathcal{Y}\,}\!_2$ the Weierstrass model 
\begin{equation}
\label{kummer3_ell_dual_W}
 Y^2 = X \,  \big( X^2 +2\, p_2(s^2) \, X  + p_2^2(s^2) - 4 \, p_1(s^2)  \big)  \;,
\end{equation}
with a two-torsion section given by $(x,y)=(0,0)$ and a distinguished section $\mathsf{S}'_0$ given by the point at infinity.
The explicit formulas for the isogeny and the dual isogeny are well known and given by
\begin{equation}
\label{Eq:isogeny_left}
{\phi}^{\mathcal{X}_2}_{\mathcal{Y}_2}: \quad \bar{\mathcal{X}\,}\!_2 \dashrightarrow \bar{\mathcal{Y}\,}\!_2\,, \quad  (x,y) \mapsto \left( \frac{y^2}{x^2}, \frac{y \, \big(p_1(s^2)  - x^2\big)}{x^2}\right)\,,
\end{equation}
and
\begin{equation}
\label{Eq:isogeny_left_b}
{\phi}^{\mathcal{Y}_2}_{\mathcal{X}_2}:  \quad \bar{\mathcal{Y}\,}\!_2 \dashrightarrow \bar{\mathcal{X}\,}\!_2\,, \quad  (X,Y) \mapsto \left( \frac{Y^2}{4 \,X^2}, \frac{Y \, \big(p_2^2(s^2) -4 \, p_1(s^2)- X^2 \big)}{8 \, X^2}\right)\,.
\end{equation}
We denote the Nikulin involution covering the dual two-isogeny ${\phi}^{\mathcal{Y}_2}_{\mathcal{X}_2}$ by  $\imath^{\mathcal{Y}_2}_F$.
We have the following:
\begin{lemma}
\label{lem:EllLeftTop}
Equation~(\ref{kummer3_ell_dual_W}) defines a  Jacobian elliptic K3 fibration $\pi^{ \mathcal{Y}_2}:  \bar{\mathcal{Y}\,}\!_2\to \mathbb{P}^1$ of Picard rank 17
with four singular fibers of type $I_4$, eight fibers of type $I_1$, a determinant of the discriminant form equal to $2^6$,
and a Mordell-Weil group of sections $\operatorname{MW}(\pi^{ \mathcal{Y}_2}, \mathsf{S}'_0)=\mathbb{Z}/2 \oplus \langle 1 \rangle^{\oplus 3}$.
\end{lemma}
\begin{proof}
The proof follows by comparison of the Weierstrass model in Equation~(\ref{kummer3_ell_dual_W}) with the one in \cite[Sec.~10]{MR3263663}.
\end{proof}

A direct consequence is:
\begin{corollary}
\label{lem:EllLeftTop2}
The K3 surface $\mathcal{Y}_2$  is the Kummer surface $\operatorname{Kum}(\operatorname{Jac} \hat{\mathcal{C}\,}\!_{12})$ of the principally polarized abelian surface
$\operatorname{Jac} \hat{\mathcal{C}\,}\!_{12}$ where $\hat{\mathcal{C}\,}\!_{12}$ is related by $(2,2)$-isogeny to $\mathcal{C}$ by Equations~(\ref{relations_RosRoots}).
\end{corollary}
\begin{proof}
Lemma~\ref{lem:EllLeftTop} gives the same singular fibers and group of sections as the elliptic fibration labeled {\tt (7)} does in the list of all possible elliptic fibrations on a generic Kummer surface determined 
by Kumar in~\cite[Thm.~2]{MR3263663} whence $\mathcal{Y}_2=\operatorname{Kum}(\operatorname{Jac}{\tilde{\mathcal{C}}})$. A computation shows that the cross ratio of the four base points of the fibers of type $I_4$
in Equation~(\ref{kummer3_ell_dual_W}) is given by $(\lambda_1'-2)/(\lambda_1'+2)$. 

For a genus-two curve $\tilde{\mathcal{C}}_{12}$, 
the $(2,2)$-isogeny $\tilde{\psi}_{12}: \operatorname{Jac}(\tilde{\mathcal{C}}_{12}) \to \operatorname{Jac}(\hat{\tilde{\mathcal{C}\,}}\!_{12})$
induces an isomorphism $\tilde{\Psi}_{12}$ between $\mathcal{Y}_2$ and $\hat{\mathcal{Y}\,}\!_2=\operatorname{Kum}(\operatorname{Jac}\hat{\tilde{\mathcal{C}\,}}\!_{12})$.
Since the Weierstrass equation~(\ref{kummer3_ell_dual_W}) only depends on $\{\lambda_1', \lambda_2', \lambda_2' \}$ it follows that  $\tilde{\Psi}_{12}$
preserves the Jacobian elliptic fibration in Corollary~(\ref{lem:EllLeftTop2}). On the other hand, if we plug Equations~(\ref{relations_RosRoots}) into the cross ratio of the four base points we obtain
\[
 \frac{\lambda_1'-2}{\lambda_1'+2}= \frac{(\Lambda_3-1)(\Lambda_1-\Lambda_2)}{(\Lambda_2-1)(\Lambda_1-\Lambda_3)} \;
 \]
 which is the cross ratio of the fibers of type $I_4$ in the elliptic fibration {\tt (7)} on $\operatorname{Kum}(\operatorname{Jac}{\mathcal{C}})$ by Kumar in~\cite[Thm.~2]{MR3263663}.
 Therefore, we have $\hat{\tilde{\mathcal{C}\,}}\!_{12} = \mathcal{C}$ and $\tilde{\mathcal{C}} = \hat{\mathcal{C}\,}\!_{12}$
 by Corollary~\ref{cor:symmetric}.
\end{proof}

The fiberwise isogeny $\phi^{\mathcal{X}_2}_{\mathcal{Y}_2}$ has again a geometric interpretation: 
in the reducible fibers of the Jacobian elliptic K3 surface $\mathcal{Y}_2$ an even eight can be identified 
such that the double branched cover map is precisely the two-isogeny $\phi^{\mathcal{X}_2}_{\mathcal{Y}_2}$.
We have the following:
\begin{lemma}
\label{Kum_B_dual_12}
The K3 surface $\mathcal{X}_2$  is the Kummer surface $\operatorname{Kum}(\hat{\mathbf{B}}_{12})$ of a $(1,2)$-polarized abelian surface
$\hat{\mathbf{B}}_{12}$ that covers $\hat{\varphi}_{12}: \hat{\mathbf{B}}_{12} \to \operatorname{Jac} \hat{\mathcal{C}\,}\!_{12}$ by an isogeny of degree two
such that the induced rational map $\hat{\phi}_{12}: \operatorname{Kum}(\hat{\mathbf{B}}_{12}) \dashrightarrow \operatorname{Kum}(\operatorname{Jac} \hat{\mathcal{C}\,}\!_{12})$
associated with the even eight $\hat{\Delta}_{12}$ realizes the fiberwise two-isogeny $\phi^{\mathcal{X}_2}_{\mathcal{Y}_2}: \mathcal{X}_2 \dashrightarrow\mathcal{Y}_2$ 
in Equation~(\ref{Eq:isogeny_left}).
\end{lemma}

\begin{proof}
For a double branched cover to turn each of the four fibers of type $I_4$ into a fiber of type $I_2$, two rational components different from 
the central fiber must be chosen to be in the branch locus.  The eight fibers of type $I_1$ with no marked components are turned into fibers of type $I_2$ by double cover. 
The sum of the aforementioned rational components in the four fibers of type $I_4$ -- two components from each fiber --  form in the notation of Mehran the even eight given by
 \[
 \hat{\Delta}_{12}= \hat{E}_{13} +  \hat{E}_{23} + \hat{E}_{14} + \hat{E}_{24} + \hat{E}_{15} + \hat{E}_{25}+ \hat{E}_{16} + \hat{E}_{26}  \;.
\] 
\end{proof}
The Van Geemen-Sarti involution $\imath^{\mathcal{Y}_2}_F$  is also a Kummer translation:
\begin{proposition}
\label{Lem:VGSI_Y2}
The Van Geemen-Sarti involution $\imath^{\mathcal{Y}_2}_F$ 
covering the two-isogeny ${\phi}^{\mathcal{Y}_2}_{\mathcal{X}_2}$ in Equation~(\ref{Eq:isogeny_left_b}) descends from a 
translation of $\hat{\mathbf{A}}=\operatorname{Jac} \hat{\mathcal{C}\,}\!_{12}$ by the two-torsion point $\hat{\mathfrak{p}}_{56}=\hat{\mathfrak{e}}_{12} \in \hat{\mathbf{A}}[2]$
to an automorphism on $\mathcal{Y}_2$.
\end{proposition}
\begin{proof}
The involution must map the set of nodes contained in each $I_4$-fiber into itself, i.e., map the nodes 
according to $\hat{e}_{1i}=\hat{p}_{j6} \leftrightarrow \hat{e}_{2i}=\hat{p}_{j5}$ for $i=3, 4, 5, 6$
and $j=4, 1, 2, 3$. It follows that the involution is induced by translation by $\hat{e}_{12}=\hat{p}_{56}$.
\end{proof}
The Nikulin involution $\jmath^{\mathcal{X}_2}$ in Equation~(\ref{C2_involutions_left_b}) lifts to an involution
$\jmath^{\mathcal{Y}_2}$ such that $\jmath^{\mathcal{X}_2} \circ \phi^{\mathcal{Y}_2}_{\mathcal{X}_2}
= \phi^{\mathcal{Y}_2}_{\mathcal{X}_2} \circ \jmath^{\mathcal{Y}_2}$ and leaves the elliptic fibration
in Equation~(\ref{C2_involutions_left_b}) invariant. It acts by $(s,X,Y) \mapsto (-s,X,-Y)$ and leaves the holomorphic two-form $ds\wedge dX/Y$ invariant.
We have the following:
\begin{proposition}
\label{Lem:BT_Y2}
 The Nikulin involution $\jmath^{\mathcal{Y}_2}$ descends from a 
translation of $\hat{\mathbf{A}}=\operatorname{Jac} \hat{\mathcal{C}\,}\!_{12}$ by the two-torsion point 
$\hat{\mathfrak{p}}_{12}=\hat{\mathfrak{e}}_{34} \in \hat{\mathbf{A}}[2]$ to an automorphism on $\mathcal{Y}_2$.
\end{proposition}
\begin{proof}
The involution $\jmath^{\mathcal{X}_2}$ in Equation~(\ref{j_involution_left}) lifts to an involution
$\jmath^{\mathcal{Y}_2}$ given by $s \mapsto -s$. The involution maps the nodes contained in the $I_4$-fibers according
to $\{ \hat{p}_{16}, \hat{p}_{15} \} \leftrightarrow \{ \hat{p}_{26}, \hat{p}_{25} \}$, and also
$\{ \hat{p}_{46}, \hat{p}_{45} \} \leftrightarrow \{ \hat{p}_{35} , \hat{p}_{36}\}$.
It follows that the involution is induced by translation by $\hat{e}_{34}=\hat{p}_{12}$.
\end{proof}

In the notation used above we have the following corollary:
\begin{theorem}
\label{cor:equiv}
There is an isomorphism $\operatorname{Kum}(\mathbf{B}_{12}) \cong \operatorname{Kum}(\hat{\mathbf{B}}_{12})$.
\end{theorem}
\begin{proof}
The proof follows from comparison of Lemma~\ref{Kum_B_dual_12} and Lemma~\ref{Kum_B_12}.
\end{proof}
  
 \subsubsection{Two special pencils of genus-three curves}
Next, we consider the quotient of the symmetric square $\mathcal{D}^{(2)}$ by the \emph{hyperelliptic} involution, i.e., $\mathcal{D}^{(2)} /\langle \imath^\mathcal{D}_h \times  \imath^\mathcal{D}_h \rangle$.
The variety $\mathcal{D}^{(2)} /\langle \imath^\mathcal{D}_h \times  \imath^\mathcal{D}_h \rangle$ is given in terms of the variables
$[W_1:W_2:W_3:\tilde{W}_4] \in \mathbb{WP}(1,1,1,4)$ with
\begin{equation}
 W_1=z^{(1)}z^{(2)}, \; W_2= x^{(1)} z^{(2)} + x^{(2)}z^{(1)}, \; W_3=x^{(1)}x^{(2)}, \; \tilde{W}_4=y^{(1)}y^{(2)}
\end{equation}
by the equation
\begin{equation}
\label{kummer_left_b}
 \tilde{W}_4^2 =  \prod_{i=1}^4 \left( \big(\lambda_i \, W_1+ W_3\big)^2  -  \lambda_i \, W^2_2 \right) \;.
\end{equation}
Using the rational transformation $[W_1:W_2:W_3:\tilde{W}_4]=[S_1z: x: S_2 z: y]$ 
for $[(S_1, S_2) , (x,y, z)] \in \mathbb{F}^{4,4}_{-1}$, Equation~(\ref{kummer_left_b}) becomes a pencil $\pi^{ \mathcal{H}}: \mathcal{H} \to \mathbb{P}^1$
of bielliptic and hyperelliptic genus-three curves given by
\begin{equation}
\label{kummer_left_hyp}
  y^2 =   \prod_{i=1}^4 \left( \lambda_i \, x^2  - \big(\lambda_i \, S_1  +S_2 \big)^2 \, z^2 \right)
\end{equation}
with projection $\pi^{ \mathcal{H}}: ([S_1:S_2],[x:y:z]) \mapsto [S_1:S_2]$. 
The central fiber over $[S_1:S_2]=[0:1]$ is isomorphic to the curve $\mathcal{D}$ in Equation~(\ref{Eq:Rosenhain_g3}), i.e.,
\begin{equation}
   y^2 =   \prod_{i=1}^4 \left( z^2 - \lambda_i \, x^2  \right)\;.
\end{equation}

The rational map
\begin{equation}
\begin{split}
  \pi^{\mathcal{D}^{(2)}_h}_{\mathcal{D}^{(2)}_b}: \quad \mathcal{D}^{(2)}/\langle \imath^\mathcal{D}_h \times  \imath^\mathcal{D}_h \rangle 
   \dashrightarrow  & \; \mathcal{D}^{(2)}/\langle \imath^\mathcal{D}_b \times  \imath^\mathcal{D}_b \rangle \\
  [W_1:W_2:W_3:\tilde{W}_4]  \mapsto & \; [Z_1:Z_2:Z_3:\tilde{Z}_4]=[W_1:W_2^2-2W_1W_3:W_3:W_4] 
\end{split}
\end{equation}
induces a rational double cover between the fibered surfaces in Equation~(\ref{kummer_left_hyp}) and Equation~(\ref{kummer_left}), i.e.,
\begin{equation}
\begin{split}
  \phi^{\mathcal{H}}_{\mathcal{X}_2}: \quad \mathcal{H}  \dashrightarrow  & \; \mathcal{X}_2 \\
  \Big([S_1:S_2], [x:y:z]\Big) \mapsto  & \; \Big([S_1:S_2], [X:Y:Z]=  [x^2-2S_1S_2z^2:y:z^2]\Big).
\end{split}
\end{equation}
The bielliptic involution $\imath_b^{\mathcal{H}}: [x:y:z] \mapsto  [-x:y:z]=[x:y:-z]$ interchanges the sheets of $\phi^{\mathcal{H}}_{\mathcal{X}_2}$.

There is a second double cover of the surface in Equation~(\ref{kummer3}). A surface is defined
for $[Z_1:Z_2:Z_3:\hat{Z}_4] \in \mathbb{WP}(1,2,1,2)$ by the equation
\begin{equation}
\label{kummer_left_c}
  \hat{Z}_4^4 =  \prod_{i=1}^4 \big( \lambda_i^2 \, Z_1^2  -  \lambda_i \, Z_2 +  Z_3^2 \big) \;.
\end{equation}
Using the rational transformation $[Z_1:Z_2:Z_3:\hat{Z}_4]=[S_1Z: XZ: S_2 Z: \hat{Y} Z]$ with 
$[(S_1, S_2) , (X,Y,Z)] \in \mathbb{F}^{2,1}_{-2}$, 
Equation~(\ref{kummer_left_c}) becomes a pencil $\pi^{ \mathcal{B}}:  \mathcal{B} \to \mathbb{P}^1$ 
of plane quartic curves given by
\begin{equation}
\label{kummer_left_bi}
  \hat{Y}^4 =   \prod_{i=1}^4 \left( \lambda_i \, X  - \big(\lambda_i^2 \, S_1^2  +S_2^2 \big) \, Z \right)\;,
\end{equation}
with projection $\pi^{ \mathcal{B}}:  ([S_1:S_2],[X:\hat{Y}:Z]) \mapsto [S_1:S_2]$.
The general member of the pencil is a smooth irreducible plane quartic curve $ \mathcal{B}_s$, hence a genus-three, 
non-hyperelliptic curve.  Four sections are given by $\mathsf{E}_i: [X:\hat{Y}:Z]=[\lambda_i^2 \, S_1^2  +S_2^2:0:\lambda_i]$ for $1\le i \le4$.
The central fiber at $[S_1:S_2]=[0:1]$ is isomorphic to the smooth plane quartic curve $\mathcal{Q}$ in Equation~(\ref{Eq:Rosenhain_g3b}), i.e.,
\begin{equation}
   \hat{Y}^4 =   \prod_{i=1}^4 \left( Z - \lambda_i \, X  \right)\;.
\end{equation}

The map $\tilde{Z}_4=\hat{Z}_4^2$ induces a rational double cover between the fibered surfaces 
in Equation~(\ref{kummer_left_bi}) and Equation~(\ref{kummer_left}), i.e.,
\begin{equation}
\begin{split}
  \phi^{ \mathcal{B}}_{\mathcal{X}_2}: \quad  \mathcal{B}  \dashrightarrow  & \; \mathcal{X}_2 \\
  \Big([S_1:S_2], [X:\hat{Y}:Z]\Big) \mapsto  & \; \Big([S_1:S_2], [X:Y:Z]=  [X:\hat{Y}^2:Z]\Big).
\end{split}
\end{equation}
whose sheets are interchanged by the bielliptic involution $\imath^{ \mathcal{B}}: [X:\hat{Y}:Z] \mapsto  [X:-\hat{Y}:Z]$.
That is, each curve $ \mathcal{B}_s$  admits the involution $\imath^{ \mathcal{B}}$ covering the degree-two morphism 
$\phi^{ \mathcal{B}}_{\mathcal{X}_2}\!\!\mid_s$ onto the fiber $ \mathcal{B}_s /\langle\imath^{ \mathcal{B}}\rangle$ of $\mathcal{X}_2$.
Moreover, the morphism $\phi^{ \mathcal{B}}_{\mathcal{X}_2}$ maps sections according to 
$\phi^{ \mathcal{B}}_{\mathcal{X}_2}\circ \mathsf{E}_i = \mathsf{e}_i$,
and the involution fixes the sections $\mathsf{E}_i$, i.e., $\imath^{ \mathcal{B}}(\mathsf{E}_i )=\mathsf{E}_i$ for $1 \le i \le4$.

Equation~(\ref{kummer1_ell_sing_b}) proves that there is a Jacobian elliptic fibration on $\mathcal{X}_2$ with twelve singular fibers $\mathcal{X}_{2,s^{\pm}_{ij}}$ 
over the base points $s^\pm_{ij}: [1: \pm \sqrt{\lambda_i \lambda_j } ]$ for $1\le i < j \le 4$
which are irreducible curves with one node of geometric genus zero. 
Therefore, the pencil $ \mathcal{B}$ will
have twelve singular members $ \mathcal{B}_{s^\pm_{ij}}$ where the fibers are curves of geometric genus two with a tacnode singularity fixed by the bielliptic involution.
Similarly,  the pencil $\mathcal{H}$ will have twelve singular members $\mathcal{H}_{s^\pm_{ij}}$ where the fibers are curves of geometric genus two with one node 
fixed by the bielliptic involution. In addition, the pencil $\mathcal{H}$ has four more singular fibers of the form $y^2=l^2x^8$ over the points $[1: -\lambda_i ]$ for $1\le i < j \le 4$ at $[x:y:z]=[0:0:1]$. We have the following:

\begin{lemma}
\label{lem:DoublePoints}
Equation~(\ref{kummer_left_bi}) defines a pencil $ \mathcal{B}$ of bielliptic plane quartic curves over $\mathbb{P}^1$
with four sections $\mathsf{E}_i$ for $1 \le i \le 4$ fixed by the bielliptic involution and twelve singular fibers that are irreducible curves of geometric genus two with a node.
$ \mathcal{B}$ is a rational double cover of the K3 surface $\mathcal{X}_2$ in Lemma~\ref{Kum_B_12}
whose sheets are interchanged by the bielliptic involution on $\mathcal{B}$.
\end{lemma}

We have also established the following result:
\begin{proposition}
\label{trigonal}
The surface $\mathcal{X}_2$ defines fiberwise a 2-2-correspondence for the families of bielliptic quartics in Equation~(\ref{kummer_left_c})
and  bielliptic and hyperelliptic curves in Equation~(\ref{kummer_left_hyp})
by the equations
\begin{equation}
\begin{split}
  \mathcal{B}_s   \quad \rightarrow & \quad \mathcal{X}_{2,s}  \quad \leftarrow \quad \mathcal{H}_s \\
[X:\hat{Y}:Z]  \rightarrow [X:\hat{Y}^2:Z]  & =[x^2-2 S_1 S_2 z^2:y:z^2] \leftarrow [x:y:z]\;.
\end{split}
\end{equation}
Moreover, the involutions $\imath^{ \mathcal{B}}$ and $\imath_b^{\mathcal{H}}$
act as deck transformations, i.e., $\phi^{ \mathcal{B}}_{\mathcal{X}_2}\circ \imath^{ \mathcal{B}} =\phi^{ \mathcal{B}}_{\mathcal{X}_2}$ and
$\phi^{\mathcal{H}}_{\mathcal{X}_2}\circ \imath^{\mathcal{H}}_b =\phi^{\mathcal{H}}_{\mathcal{X}_2}$ in the diagram
\par\centerline{
\xymatrix{
{ \mathcal{B}_s} \ar@(ul,dl)[]|{\imath^{ \mathcal{B}}}\ar@{-->}[rd]^{\phi^{ \mathcal{B}}_{\mathcal{X}_2}} &&  {\mathcal{H}_s} \ar@(ur,dr)[]|{\imath_b^{\mathcal{H}}}\ar@{-->}[ld]_{\phi^{\mathcal{H}}_{\mathcal{X}_2}} \\
& \mathcal{X}_2\!\mid_{s=[S_1:S_2]} 
}}
\end{proposition}

\begin{remark}
Proposition~\ref{trigonal} is a special case of the trigonal construction of a Hurwitz scheme parametrizing bielliptics in the special case when it 
puts in correspondence bielliptic and hyperelliptic genus-three curves.
\end{remark}
 
Using the Duality Theorem~\ref{thm:duality} we have established the following result:
 \begin{theorem}
 \label{cor:pencil}
For the genus-three curve $\mathcal{B}_s$ in Equation~(\ref{kummer_left_bi}) with $s\not =s^\pm_{ij}$ and the elliptic curve
$\mathcal{F}_s = \mathcal{B}_s /\langle\imath^{ \mathcal{B}} \rangle$,
it follows that $\mathcal{B}_s$ is embedded into $\operatorname{Prym}(\mathcal{B}_s/\mathcal{F}_s)$ 
with self-intersection $4$. Moreover, $\operatorname{Prym}(\mathcal{B}_s/\mathcal{F}_s)$
carries a natural $(1,2)$-polarization such that 
\begin{equation}
 \operatorname{Prym}\left( \mathcal{B}_s/\mathcal{F}_s\right) \cong \mathbf{B}_{12} \;.
\end{equation}
 \end{theorem}
 \begin{proof}
 The abelian surface $\mathbf{B}_{12}$ carries a line bundle $\mathcal{L}$ of self-intersection $\mathcal{L}\cdot \mathcal{L}=4$
defining a $(1,2)$-polarization such that any element of $|\mathcal{L}|$ is irreducible.
 Let $\widetilde{\mathbf{B}}_{12}$ be the surface obtained by blowing up $\mathbf{B}_{12}$ at the four base points and let
 $\widetilde{\Phi}: \widetilde{\mathbf{B}}\to \mathbb{P}^1$ be the the fibration induced by the pencil $|\mathcal{L}|$.
 We denote by $\mathsf{E}_i$ for $1 \le i \le 4$ the exceptional curves of the blow-up that are sections of $\widetilde{\Phi}$
 \cite[Ex.~10.1.4]{MR2062673}.  The general member $\mathcal{B} \in |\mathcal{L}|$ is a bielliptic curve of genus-three 
 such that $\imath^{\mathcal{B}} = -\mathbb{I} \mid_{\mathbf{B}_{12}}$ and  $\operatorname{Prym}(\mathcal{B}/\mathcal{F}) \cong \mathbf{B}_{12}$.
 
 We proved in Lemma~\ref{Kum_B_12} that the surface $\mathcal{X}_2$  is the Kummer surface $\operatorname{Kum}(\mathbf{B}_{12})$.
 In Lemma~\ref{Kum_B_12}, we proved that the pencil $ \mathcal{B}$ in Equation~(\ref{kummer_left_bi}) is a pencil of bielliptic plane quartic curves
 and forms a rational double cover of the K3 surface $\mathcal{X}_2$. The sheets are interchanged by the bielliptic involution on $\mathcal{B}$.
 Moreover, the preimages of the sections $\mathsf{e}_i$   for $1\le i \le 4$ on $\mathcal{X}_2$ are the sections $\mathsf{E}_i$  on $\mathcal{B}$
 which are fixed by the bielliptic involution. Hence, the sections $\mathsf{E}_i$ correspond to base points of $\mathbf{B}_{12}$ and $\mathcal{B}_s \in  |\mathcal{L}|$
 for each $s$ not in the singular locus, i.e., $s\not =s^\pm_{ij}$. The result then follows from Theorem~\ref{thm:duality}.
 \end{proof}
 
\begin{corollary}
 The double points in Lemma~\ref{lem:DoublePoints} are exactly the order-two points of 
 $\mathbf{B}_{12}$ different from the base points.
\end{corollary}
\begin{proof}
We showed that the twelve singular points $q_{ij}^{\pm}$  for $1\le i < j \le 4$
on $\mathcal{X}_2$ give the twelve singular members $ \mathcal{B}_{s^\pm_{ij}}$ 
where the fibers are curves of geometric genus two with a tacnode singularity fixed by the bielliptic involution, and
are mapped to the six nodes $p_{ij}$ on $\mathcal{Y}_1$. Hence, they are 
exactly the order-two points of $\mathbf{B}_{12}$ different from the base points.
\end{proof}

\subsection{Kummer surface from the genus-one curve $\mathcal{E}$}
We start with two copies of the elliptic curve in Equation~(\ref{Eq:Rosenhain_g1}) and the symmetric product $\mathcal{E}^{(2)} = (\mathcal{E}\times\mathcal{E})/\langle \sigma_{\mathcal{E}^{(2)} } \rangle$.

The variety $\mathcal{E}^{(2)} /\langle \imath^\mathcal{E}_h \times  \imath^\mathcal{E}_h \rangle$ is given in terms of the variables
$w_1=Z^{(1)}Z^{(2)}$, $w_2=X^{(1)} Z^{(2)} +X^{(2)}Z^{(1)}$, $w_3=X^{(1)}X^{(2) }$, and
$\tilde{w}_4=Y^{(1)}Y^{(2)}$ with $[w_1:w_2:w_3:\tilde{w}_4] \in \mathbb{WP}(1,1,1,2)$ by the equation
\begin{equation}
\label{kummer4}
\tilde{w}_4^2 =  \prod_{i=1}^4 \big( \lambda_i^2 \, w_1  -  \lambda_i \, w_2 +  w_3\big) \;.
\end{equation}

\subsubsection{Rational elliptic surfaces}
Using the rational transformation $[w_1:w_2:w_3:\tilde{w}_4]=[T_1 z: x: T_2 z: y]$ for $[(T_1, T_2) , (x,y,z)] \in \mathbb{F}^{2,2}_{-1}$, 
Equation~(\ref{kummer4}) becomes the equation of a genus-one fibration $\mathcal{Z}_1$ over $\mathbb{P}^1$ given by
\begin{equation}
\label{rational1_ell}
    y^2 = \prod_{i=1}^4 \left( \lambda_i \, x  - \big(\lambda_i^2 \, T_1  +T_2 \big) \, z \right)\;.
\end{equation}
The projection $\pi^{\mathcal{Z}_1}: \mathcal{Z}_1\to \mathbb{P}^1$ with $([T_1:T_2],[x:y:z]) \mapsto [T_1:T_2]$ 
defines a rational fibration with six fibers of type $I_2$. Four sections are given by  $\mathsf{E}_i: [x:y:z]=[\lambda_i^2 \, T_1  +T_2:0:\lambda_i]$
for $1\le i \le4$. We choose the point given by $\mathsf{E}_4$ to be the neutral element of the Mordell-Weil group of sections,
turning Equation~(\ref{rational1_ell}) into a Jacobian elliptic fibration. The Mordell-Weil is generated by 4 two-torsion 
sections $\mathsf{E}_i$ for $1\le i \le4$ and two sections of infinite order
$\mathsf{Q}: [x:y:z]=[1:\pm l:0]$ and
\begin{equation}
\mathsf{R}: \left\lbrace \begin{array}{rl}
 x & = (T_1 + T_2) ( T^2_2 -l^2 T^2_1) -  \prod_{i=1}^3 (T_2 - \lambda_i T_1)  \,,\\[0.2em]
 y & =  \pm  \prod_{i=1}^3 (T_2 - \lambda_i T_1) (T_2 - T_1 \! \prod_{k=1, k\not= i }^3 \! \lambda_k)\,,\\[0.4em]
 z & =  T_2^2 -l^2 T_1^2 \;.
\end{array}\right.
\end{equation}
 
 From the cover map $\pi^{\mathcal{D}}_b: \mathcal{D} \to \mathcal{E}$ we obtain a degree-two, rational map given by
\begin{equation}
\begin{split}
  \pi^{\mathcal{D}^{(2)}_b}_{\mathcal{E}^{(2)}}: \quad \mathcal{D}^{(2)}/\langle \imath^\mathcal{D}_b \times  \imath^\mathcal{D}_b \rangle 
   \dashrightarrow  & \; \mathcal{E}^{(2)}/\langle \imath^\mathcal{E}_h \times  \imath^\mathcal{E}_h \rangle \\
  [Z_1:Z_2:Z_3:\tilde{Z}_4]\mapsto & \; [w_1:w_2:w_3:\tilde{w}_4] = [Z^2_1: Z_2: Z^2_3:\tilde{Z}_4] \;,
\end{split}
\end{equation}
that induces a map between the corresponding elliptic surfaces in Equation~(\ref{kummer_left}) and Equation~(\ref{rational1_ell}), i.e., 
\begin{equation}
\begin{split}
  \phi^{\mathcal{X}_2}_{\mathcal{Z}_1}: \quad \mathcal{X}_2 \dashrightarrow  & \; \mathcal{Z}_1 \\
  \Big([S_1:S_2], [X:Y:Z]\Big) \mapsto & \; \Big([T_1:T_2], [x:y:z]\Big)= \Big([S_1^2:S_2^2], [X:Y:Z]\Big).
\end{split}
\end{equation}
The involution $\jmath^{ \mathcal{X}_2}$ interchanges the sheets of $\phi^{\mathcal{X}_2}_{\mathcal{Z}_1}$; the morphism maps sections 
according to $\phi^{\mathcal{X}_2}_{\mathcal{Z}_1} \circ \mathsf{e}_i = \mathsf{E}_i$  for $1 \le i \le4$,
and the same relations holds between $\mathsf{q}$ and $\mathsf{Q}$, and $\mathsf{r}$ and $\mathsf{R}$, respectively.
We have the following:
\begin{lemma}
\label{lem:RatZ1}
Equation~(\ref{rational1_ell}) defines a rational Jacobian elliptic fibration $\pi^{ \mathcal{Z}_1}:  \bar{\mathcal{Z}\,}\!_1\to \mathbb{P}^1$
with six singular fibers of type $I_2$, eight fibers of type $I_1$,
and a Mordell-Weil group of sections $\operatorname{MW}(\pi^{ \mathcal{Z}_1},\mathsf{E}_4)=(\mathbb{Z}/2)^2 \oplus \langle \frac{1}{2} \rangle^{\oplus 2}$.
\end{lemma}

A second surface is obtained as follows: for $[(U_1, U_2) , (X,Y,Z)] \in \mathbb{F}^{2,2}_{-1}$ 
the following equation defines a genus-one fibration $\mathcal{Z}_3$ over $\mathbb{P}^1$ given by
\begin{equation}
\label{rational2_ell}
  Y^2 = (X- U_2 Z) \, \prod_{i=1}^3 \left( X  - \lambda'_i \, l \, U_1 \, Z \right)\;.
\end{equation}
The projection $\pi^{ \mathcal{Z}_3}: \mathcal{Z}_3 \to \mathbb{P}^1$ with $([U_1:U_2],[X:Y:Z]) \mapsto [U_1:U_2]$ 
defines a rational fibration with three fibers of type $I_2$ and one fiber of type $I_0^*$. 
Four sections are given by  $\mathsf{e}'_i: [\lambda_i' \, l \, U_1:0:1]$ for $1\le i \le 3$ and $\mathsf{e}'_4: [U_2:0:1]$.
We choose the point given by $\mathsf{e}'_4$ to be the neutral element of the Mordell-Weil group of sections,
turning Equation~(\ref{rational2_ell}) into a Jacobian elliptic fibration. The Mordell-Weil group is generated by the 4 two-torsion sections 
$\mathsf{e}_i'$ for $1\le i \le 4$ and one section of infinite order given by
$\mathsf{p}': [X:Y:Z]=[1:1:0]$.  
A degree-two map  $\phi^{\mathcal{Y}_1}_{\mathcal{Z}_3}: \mathcal{Y}_1 \dashrightarrow  \mathcal{Z}_3$ between the surfaces in Equation~(\ref{kummer_middle_ell_p}) and Equation~(\ref{rational2_ell}) is given by $ [U_1:U_2] = [  T_1T_2 : T_2^2 + l^2 T_1^2]$ and $ [X:Y:Z] = [x': y': z']$. 
The involution $\imath^{\mathcal{Y}_1}$ interchanges the sheets of $\phi^{\mathcal{Y}_1}_{\mathcal{Z}_3}$;
the morphism maps sections  according to $\phi^{\mathcal{Y}_1}_{\mathcal{Z}_3} \circ \mathsf{E}'_i = \mathsf{e}'_i $  for $1 \le i \le4$;
the same relation holds for $\mathsf{P}'$ and $\mathsf{p}'$. We have the following:

\begin{lemma}
\label{lem:RatZ3}
Equation~(\ref{rational2_ell}) defines a rational Jacobian elliptic fibration $\pi^{ \mathcal{Z}_3}:  \bar{\mathcal{Z}\,}\!_3\to \mathbb{P}^1$ 
with three fibers of type $I_2$ and one fiber of type $I_0^*$, and a Mordell-Weil group of sections 
$\operatorname{MW}(\pi^{ \mathcal{Z}_3}, \mathsf{e}_4')=(\mathbb{Z}/2)^2 \oplus \langle \frac{1}{2} \rangle$.
\end{lemma}

\subsubsection{A K3 elliptic surface}
We consider the quadratic twist of Equation~(\ref{rational2_ell}), i.e., for $[(U_1, U_2) , (X,Y,Z)] \in \mathbb{F}^{3,2}_{-1}$ a genus-one fibration $\mathcal{X}_3$ over $\mathbb{P}^1$ given by
\begin{equation}
\label{K3_2_ell}
   Y^2 = (U_2^2 - 4 \, l^2 U_1^2) \, (X- U_2 Z) \, \prod_{i=1}^3 \left( X  - \lambda'_i \, l \, U_1 \, Z \right)\;.
\end{equation}
The projection $\pi^{ \mathcal{X}_3}: \mathcal{X}_3\to \mathbb{P}^1$ with $([U_1:U_2],[X:Y:Z]) \mapsto [U_1:U_2]$ 
defines a K3 fibration with three fibers of Kodaira type $I_2$ and three fibers of Kodaira type $I_0^*$. 
Four sections are given by  $\mathsf{e}'_i: [\lambda_i' l U_1:0:1]$ for $1\le i \le 3$ and $\mathsf{e}'_4: [U_2:0:1]$.
We move the point given by $\mathsf{e}_4'$ to infinity and convert Equation~(\ref{K3_2_ell})
to Weierstrass form with coefficients in $\mathbb{C}[u]$ using a transformation defined over 
$\mathbb{C}(u)$ with $u=U_2/(lU_1)$. After translation by two-torsion, the Weierstrass fibration is given by
\begin{equation}
\label{kummer_right_ell_p_W}
\begin{split}
y^2  = x \, \big( x-   \left( \lambda'_1-\lambda'_2 \right)  \left( u - \lambda'_3\right)  (u^2 - 4) \big)\,  \big( x-  \left( \lambda'_1 - \lambda'_3 \right)  \left( u-\lambda'_2 \right) (u^2 - 4 ) \big) \;,
\end{split}
\end{equation}
with two-torsion section $(x,y)=(0,0)$, and a distinguished section $\mathsf{s}''_0$ given by the point at infinity.
We will denote the Weierstrass fibration by $y^2= x^3 - (u^2-4) \, q_2(u)  \, x^2+(u^2-4)^2 \, q_1(u)\,  x$ with
\begin{equation}
\label{polynomials_right}
\begin{split}
 q_1(u) &=   \left( \lambda_1'-\lambda_2' \right)  \left( \lambda_1'-\lambda_3' \right) \left( u -  \lambda'_2  \right)  \left( u - \lambda'_3 \right)  ,\\
 q_2(u) & = (2\lambda_1'-\lambda_2'-\lambda_3')  \, u - (\lambda_1'(\lambda_2'+\lambda_3') - 2 \lambda_2' \lambda_3')  \;,
 \end{split}
 \end{equation}
 such that $p_i(t)=t^{3-i} \, q_i(u=(t^2+1)/t)$ for $i=1, 2$ when compared with Equations~(\ref{polynomials_left}).
 The discriminant is given by
 \begin{equation}
 \label{discrim_R}
 \begin{split}
  \Delta_{\mathcal{X}_3} =  \left( \lambda_1'-\lambda_2' \right)^2  \left( \lambda_1'-\lambda_3' \right)^2  \left( \lambda_2'-\lambda_3' \right)^2  (u^2-4)^6  \left( u -  \lambda'_1 \right)^2  \left( u - \lambda'_2  \right)^2 \left( u -  \lambda'_3  \right)^2 \;.
 \end{split} 
 \end{equation}

A degree-two map between the elliptic surfaces in Equation~(\ref{kummer_middle_ell_p_W}) and Equation~(\ref{kummer_right_ell_p_W})
is given by
\begin{equation}
\label{double_cover_rev}
\begin{split}
  \phi^{\mathcal{Y}_1}_{\mathcal{X}_3}: \quad \mathcal{Y}_1 \dashrightarrow  & \; \mathcal{X}_3, \\
  \Big(t, X,Y\Big) \mapsto  & \Big(u, x,y\Big)= \left(t+\frac{1}{t}, (1-\frac{1}{t^2})^2X,(1-\frac{1}{t^2})^3Y\right).
  \end{split}
\end{equation}
The involution $\imath^{\mathcal{Y}_1}$ interchanges the two sheets of $\phi^{\mathcal{Y}_1}_{\mathcal{X}_3}$;
the morphism maps sections  according to $\phi^{\mathcal{Y}_1}_{\mathcal{X}_3} \circ \mathsf{E}'_i = \mathsf{e}'_i$  for $1 \le i \le4$.
We collect the properties of the Jacobian elliptic K3 fibration in Equation~(\ref{kummer_right_ell_p_W}) in the following:
\begin{proposition}
\label{lem:EllRight}
Equation~(\ref{kummer_right_ell_p_W}) defines a Jacobian elliptic K3 fibration $\pi^{ \mathcal{X}_3}:  \bar{\mathcal{X}\,}\!_3\to \mathbb{P}^1$ of Picard rank 17
with three singular fibers of type $I_2$, three fibers of type $I_0^*$, a determinant of the discriminant form equal to $2^5$,
and a Mordell-Weil group of sections $\operatorname{MW}(\pi^{ \mathcal{X}_3}, \mathsf{s}''_0)=(\mathbb{Z}/2)^2$.
\end{proposition}

\subsubsection{A Van Geemen-Sarti involution}
The translation by the two-torsion section $(x,y)=(0,0)$ in Equation~(\ref{kummer_right_ell_p_W}) is a Nikulin involution $\imath_F$, and we obtain a K3 surface $\mathcal{Y}_3$ by resolving the eight nodes of 
$\mathcal{X}_3/\imath_F$. The isogeny  $\phi^{\mathcal{X}_3}_{\mathcal{Y}_3}: \mathcal{X}_3 \dashrightarrow \mathcal{Y}_3$ is the rational quotient map over 
$\mathbb{C}(u)$ with the kernel generated by the section $(x,y)=(0,0)$ and the point at infinity. 
We obtain for $\bar{\mathcal{Y}\,}\!_3$ the Weierstrass model 
\begin{equation}
\label{K3_2_ell_dual_W}
 Y^2 = X \,  \Big( X^2 -2\, (u^2-4) \, q_2(u) \, X  + (u^2-4)^2\, \big(q_2^2(u) - 4 \, q_1(u)\big)  \Big)  \;,
\end{equation}
with a two-torsion section given by $(X,Y)=(0,0)$ and a distinguished section $\mathsf{S}''_0$ given by the point at infinity.
The explicit formulas for the isogeny and the dual isogeny are well known and given by
\begin{equation}
\label{Eq:isogeny_right}
{\phi}^{\mathcal{X}_3}_{\mathcal{Y}_3}: \quad \bar{\mathcal{X}\,}\!_3 \dashrightarrow \bar{\mathcal{Y}\,}\!_3\,, \quad  (x,y) \mapsto \left( \frac{y^2}{x^2}, \frac{y \, \big((u^2-4)^2  q_1(u)  - x^2\big)}{x^2}\right)\,,
\end{equation}
and
\begin{equation}
{\phi}^{\mathcal{Y}_3}_{\mathcal{X}_3}:  \quad \bar{\mathcal{Y}\,}\!_3 \dashrightarrow \bar{\mathcal{X}\,}\!_3\,, \quad  (X,Y) \mapsto \left( \frac{Y^2}{4 \,X^2}, \frac{Y \, \big((u^2-4)^2(q_2^2(u) -4 \, q_1(u))- X^2 \big)}{8 \, X^2}\right)\,.
\end{equation}
We denote the Nikulin involution covering the dual two-isogeny ${\phi}^{\mathcal{Y}_3}_{\mathcal{X}_3}$ by  $\imath^{\mathcal{Y}_3}_F$.
We have the following:
\begin{lemma}
\label{lem:RightTop}
Equation~(\ref{K3_2_ell_dual_W}) defines a  Jacobian elliptic K3 fibration $\pi^{ \mathcal{Y}_3}:  \bar{\mathcal{Y}\,}\!_3\to \mathbb{P}^1$ of Picard rank 17
with a singular fiber of type $I_4$, two fibers of type $I_1$, three fibers of type $I^*_0$, a determinant of the discriminant form equal to $2^6$,
and a Mordell-Weil group of sections $\operatorname{MW}(\pi^{ \mathcal{Y}_3}, \mathsf{S}''_0)=\mathbb{Z}/2$.
\end{lemma}

We also obtain:
 \begin{lemma}
\label{lem:RightTop2}
The K3 surface $\mathcal{Y}_3$  is the Kummer surface $\operatorname{Kum}(\operatorname{Jac} \hat{\mathcal{C}\,}\!_{12})$ of the principally polarized abelian surface
$\operatorname{Jac} \hat{\mathcal{C}\,}\!_{12}$
 where $\hat{\mathcal{C}\,}\!_{12}$ is related by $(2,2)$-isogeny to $\mathcal{C}$ by Equations~(\ref{relations_RosRoots}).
\end{lemma}
\begin{proof}
Lemma~\ref{lem:RightTop} gives the same singular fibers and group of sections as fibration labeled {\tt (4)} 
in the list of all possible elliptic fibrations on a generic Kummer surface determined 
by Kumar in~\cite[Thm.~2]{MR3263663} whence $\mathcal{Y}_3=\operatorname{Kum}(\operatorname{Jac}{\tilde{\mathcal{C}}})$. A computation shows that the cross ratio of the four base points of the fibers of type $I_4$
and $I_0^*$ in Equation~(\ref{K3_2_ell_dual_W}) is given by $(\lambda_1'-2)/(\lambda_1'+2)$. 

The $(2,2)$-isogeny $\tilde{\psi}_{12}: \operatorname{Jac}(\tilde{\mathcal{C}}_{12}) \to \operatorname{Jac}(\hat{\tilde{\mathcal{C}\,}}\!_{12})$
induces an isomorphism $\tilde{\Psi}_{12}$ between $\mathcal{Y}_3$ and $\hat{\mathcal{Y}\,}\!_3=\operatorname{Kum}(\operatorname{Jac}\hat{\tilde{\mathcal{C}\,}}\!_{12})$.
Since the Weierstrass equation~(\ref{kummer3_ell_dual_W}) only depends on $\{\lambda_1', \lambda_2', \lambda_2' \}$ it follows that $\tilde{\Psi}_{12}$
preserves the Jacobian elliptic fibration in Lemma~(\ref{lem:RightTop2}). On the other hand, if we plug Equations~(\ref{relations_RosRoots}) into the cross ratio of the four base points we obtain
\[
 \frac{\lambda_1'-2}{\lambda_1'+2}= \frac{\Lambda'_1-\Lambda'_2}{\Lambda'_1-\Lambda'_3} =\frac{(\Lambda_3-1)(\Lambda_1-\Lambda_2)}{(\Lambda_2-1)(\Lambda_1-\Lambda_3)} \;
 \]
 which is the cross ratio of the fibers of type $I_4$ and $I_0^*$ in the elliptic fibration {\tt (4)} on $\operatorname{Kum}(\operatorname{Jac}{\mathcal{C}})$ by Kumar in~\cite[Thm.~2]{MR3263663}.
  Therefore, we have $\hat{\tilde{\mathcal{C}\,}}\!_{12} = \mathcal{C}$ and $\tilde{\mathcal{C}} = \hat{\mathcal{C}\,}\!_{12}$ by Corollary~\ref{cor:symmetric}.
\end{proof}

We can now prove that the K3 surface $\mathcal{X}_3$ is a $(1,2)$-polarized Kummer surface:
\begin{lemma}
\label{Kum_B_dual_12b}
The K3 surface $\mathcal{X}_3$  is the Kummer surface $\operatorname{Kum}(\hat{\mathbf{B}}_{12})$ of a $(1,2)$-polarized abelian surface
$\hat{\mathbf{B}}_{12}$ that covers $\hat{\varphi}_{12}: \hat{\mathbf{B}}_{12} \to \operatorname{Jac} \hat{\mathcal{C}\,}\!$ by an isogeny of degree two
such that the induced rational map $\hat{\phi}_{12}: \operatorname{Kum}(\hat{\mathbf{B}}_{12}) \dashrightarrow \operatorname{Kum}(\operatorname{Jac} \hat{\mathcal{C}\,}\!)$
associated with the even eight $\hat{\Delta}_{12}$ realizes the fiberwise two-isogeny $\phi^{\mathcal{X}_3}_{\mathcal{Y}_3}: \mathcal{X}_3 \dashrightarrow\mathcal{Y}_3$ 
in Equation~(\ref{Eq:isogeny_right}).
\end{lemma}

\begin{proof}
For a double branched cover to turn the fiber of type $I_4$ into a fiber of type $I_2$, two rational components different from 
the central fiber must be chosen to be in the branch locus.  Results of Kumar~\cite{MR3263663} show that there only two such components in the fiber of type $I_4$
in this fibration, namely $\hat{E}_{16}, \hat{E}_{26}$. To turn a fiber of type $I_0^*$ into another fiber of type $I_0^*$
by double cover, two rational curves different from the central fiber must be chosen to be in the branch locus.  
Results of Kumar~\cite{MR3263663} show that each $I_0^*$-fiber in this fibration contains exactly three nodes, namely
\begin{equation}
\label{list_nodes}
 \hat{E}_{13}, \hat{E}_{23}, \hat{E}_{36},  \quad \hat{E}_{14}, \hat{E}_{24}, \hat{E}_{46} \quad \hat{E}_{15} , \hat{E}_{25}, \hat{E}_{56} \;.
\end{equation}
There is no even eight that contains $\{\hat{E}_{16}, \hat{E}_{26}\}$ and two components of $\{\hat{E}_{13}, \hat{E}_{23}, \hat{E}_{36}\}$ unless the components
$\hat{E}_{13}, \hat{E}_{23}$ are chosen. Similarly, there is no even eight that contains $\{\hat{E}_{16}, \hat{E}_{26}\}$ and two components of 
$\{\hat{E}_{14}, \hat{E}_{24}, \hat{E}_{46}\}$ unless the two components
$\hat{E}_{14}, \hat{E}_{24}$ are chosen. The only overlapping case is in the notation of Mehran the even eight given by the even eight
 \[
 \hat{\Delta}_{12}= \hat{E}_{13} +  \hat{E}_{23} + \hat{E}_{14} + \hat{E}_{24} + \hat{E}_{15} + \hat{E}_{25}+ \hat{E}_{16} + \hat{E}_{26}  \;.
\] 
\end{proof}

We have the following:
\begin{proposition}
\label{Lem:VGSI_Y3}
The van Geemen-Sarti 
involution $\imath^{\mathcal{Y}_3}_F$ covering the two-isogeny ${\phi}^{\mathcal{Y}_3}_{\mathcal{X}_3}$ in Equation~(\ref{Eq:isogeny_right}) descends from a 
translation of $\hat{\mathbf{A}}=\operatorname{Jac} \hat{\mathcal{C}\,}\!_{12}$ by the two-torsion point $\hat{\mathfrak{p}}_{56}=\hat{\mathfrak{e}}_{12} \in \hat{\mathbf{A}}[2]$
to an automorphism on $\mathcal{Y}_3$.
\end{proposition}
\begin{proof}
The involution must map the set of nodes according to $\hat{E}_{1i} \leftrightarrow \hat{E}_{2i}$ for $i=3, 4, 5, 6$.
It follows that the involution is induced by translation by $\hat{e}_{12}=\hat{p}_{56}$.
\end{proof}

The maximal isotropic subgroup $\mathsf{k}_{12}$ of the two-torsion $\mathbf{A}[2]$ of $\mathbf{A}=\operatorname{Jac}(\mathcal{C})$ defines the $(2,2)$-isogeny 
$\psi_{12}: \mathbf{A}= \operatorname{Jac}(\mathcal{C}) \to \hat{\mathbf{A}}_{12} = \operatorname{Jac}(\hat{\mathcal{C}\,}\!_{12})$ by $\hat{\mathbf{A}}_{12}=
\mathbf{A}/\mathsf{k}_{12}$.
The decomposition of the associated algebraic map $\Psi_{12}: \operatorname{Kum}(\operatorname{Jac} \mathcal{C}) 
\to \operatorname{Kum}(\operatorname{Jac} \hat{\mathcal{C}\,}\!_{12})$ of Kummer surfaces as defined Section~\ref{ssec:action_isog} 
into  $\phi^{\mathcal{X}_3}_{\mathcal{Y}_3} \circ \phi^{\mathcal{Y}_1}_{\mathcal{X}_3}$ corresponds to a factorization of $\mathsf{k}_{12}$ 
given by two generators that can be characterized as Nikulin involutions on the corresponding Kummer surfaces:
\begin{theorem}
\label{cor:psi}
The maximal isotropic subgroup $\mathsf{k}_{12}$ of the two-torsion $\mathbf{A}[2]$ 
defining the $(2,2)$-isogeny  $\psi_{12}: \mathbf{A}= \operatorname{Jac}(\mathcal{C}) \to \hat{\mathbf{A}}_{12}=
\mathbf{A}/\mathsf{k}_{12}$ in Equations~(\ref{relations_RosRoots}) is given by
\[ 
 \mathsf{k}_{12}= \{ \mathfrak{p}_0, \mathfrak{p}_{12}, \mathfrak{p}_{56}, \mathfrak{p}_{34}=\mathfrak{p}_{12}+ \mathfrak{p}_{56} \} \;.
\]
The translations by the two-torsion point  $\mathfrak{p}_{12}$ and $\mathfrak{p}_{56}$ descend to automorphisms 
on $\mathcal{Y}_1=\operatorname{Kum}(\operatorname{Jac}\mathcal{C})$ 
that act on the elliptic fibration given by $H - P_0 -P_{56}$ 
as Van Geemen-Sarti involution $\imath_F^{\mathcal{Y}_1}$  and Nikulin involution $\imath^{\mathcal{Y}_1}$ of a rational base transformation, 
respectively.
\end{theorem}
\begin{proof}
The proof follows from Propositions~\ref{Lem:BT_Y1} and~\ref{Lem:VGSI_Y1}, and the fact that by construction the two 
Van Geemen-Sarti involutions $\imath_F^{\mathcal{Y}_1}$ and $\imath_F^{\mathcal{Y}_3}$ act equivariantly
with respect to $\Psi_{12}=\phi^{\mathcal{X}_3}_{\mathcal{Y}_3} \circ \phi^{\mathcal{Y}_1}_{\mathcal{X}_3}$, i.e.,
$\Psi_{12} \circ \imath_F^{\mathcal{Y}_1} = \imath_F^{\mathcal{Y}_3}\circ \Psi_{12}$.
\end{proof}
Using Lemma~\ref{lem:decomp} we can find another maximal isotropic subgroup $\!\hat{\,\mathsf{k}}_{12}$ 
such that  $\!\hat{\,\mathsf{k}}_{12}+\mathsf{k}_{12}=\mathbf{A}[2]$, $\!\hat{\,\mathsf{k}}_{12} \cap \mathsf{k}_{12}=\{ \mathfrak{p}_0\}$.
 Set $\hat{\mathbf{A}}_{12}=\mathbf{A}/\mathsf{k}_{12}$, and denote the image of $\!\hat{\,\mathsf{k}}_{12}$ in $\hat{\mathbf{A}}_{12}$
 by $\mathsf{K}_{12}$. The group $\mathsf{K}_{12}$ of the two-torsion $\hat{\mathbf{A}}_{12}[2]$ defines the dual $(2,2)$-isogeny 
$\hat{\psi}_{12}: \hat{\mathbf{A}}_{12} = \operatorname{Jac}(\hat{\mathcal{C}\,}\!_{12}) \to \mathbf{A}$ since $\mathbf{A} = \hat{\mathbf{A}}_{12} /\mathsf{K}_{12}$.
The decomposition of the associated algebraic map $\hat{\Psi}_{12}: \operatorname{Kum}(\operatorname{Jac} \hat{\mathcal{C}\,}\!_{12})
\to \operatorname{Kum}(\operatorname{Jac} \mathcal{C})$ of Kummer surfaces  into  
$\phi^{\mathcal{X}_2}_{\mathcal{Y}_1} \circ \phi^{\mathcal{Y}_2}_{\mathcal{X}_2}$ corresponds to a factorization of $\mathsf{K}_{12}$ 
given by two generators:
\begin{theorem}
\label{cor:hat_psi}
The maximal isotropic subgroup $\mathsf{K}_{12}$ of the two-torsion $\hat{\mathbf{A}}_{12}[2]$ defining the $(2,2)$-isogeny 
$\hat{\psi}_{12}: \hat{\mathbf{A}}_{12} = \operatorname{Jac}(\hat{\mathcal{C}\,}\!_{12}) \to \mathbf{A} = \hat{\mathbf{A}}_{12} /\mathsf{K}_{12}$ 
in Equations~(\ref{relations_RosRoots})  is given by
\[ 
 \mathsf{K}_{12}= \{ \hat{\mathfrak{p}}_0, \hat{\mathfrak{p}}_{12}, \hat{\mathfrak{p}}_{56}, \hat{\mathfrak{p}}_{34}=\hat{\mathfrak{p}}_{12}+ \hat{\mathfrak{p}}_{56} \} \;.
\]
The translations by the two-torsion point 
$\hat{\mathfrak{p}}_{56}$ and $\hat{\mathfrak{p}}_{12}$ descend to automorphisms on $\mathcal{Y}_2=\operatorname{Kum}(\operatorname{Jac} \hat{\mathcal{C}\,}\!_{12})$
that act on the elliptic fibration given by $\mathsf{E}_4 + \mathsf{Q} + P_{46} + P_{56}$ as Van Geemen-Sarti involution $\imath_F^{\mathcal{Y}_2}$ and Nikulin involution 
$\jmath^{\mathcal{Y}_2}$ of a rational base transformation, respectively.
\end{theorem}
\begin{proof}
The proof follows from Propositions~\ref{Lem:BT_Y2} and~\ref{Lem:VGSI_Y2}, and the fact that by construction the two 
Van Geemen-Sarti involutions $\imath_F^{\mathcal{Y}_2}$ and $\imath_F^{\mathcal{Y}_1}$ act equivariantly
with respect to $\hat{\Psi}_{12}=\phi^{\mathcal{X}_2}_{\mathcal{Y}_1} \circ \phi^{\mathcal{Y}_2}_{\mathcal{X}_2}$, i.e.,
$\hat{\Psi}_{12} \circ \imath_F^{\mathcal{Y}_2} = \imath_F^{\mathcal{Y}_1}\circ \hat{\Psi}_{12}$.
\end{proof}

\subsection{Double covers of the projective plane}
We describe the surface $\mathcal{X}_3$ as branched double cover using the moduli $\lambda_1', \lambda_2', \lambda_3'$:
\begin{lemma}
The surface $\mathcal{X}_3$ is the double branched cover of $\mathbb{P}^2$ given by
\begin{equation}
\label{Eq:SixLines}
 y^2 = \ell_1 \, \ell_2 \, \ell_3 \, \ell_4 \,  \ell_5 \, \ell_6 \;,
\end{equation}
where the six lines $\ell_1 , \dots,  \ell_6$ in $\mathbb{P}^2$ are given by
\begin{equation}
\label{Eq:SixLines_def}
 \ell_i: Z_2 - \lambda_i' Z_3=0 \, (\text{for $i=1,2,3$}), \quad \ell_4: Z_1 -Z_2=0, \quad \ell_{5,6}: Z_1 \pm 2 Z_3=0 \;.
\end{equation} 
In particular, the surface $\mathcal{X}_3$ is the minimal non-singular model for the double 
cover of the projective plane branched along six lines.  
Moreover, the six lines intersect as follows
\begin{align*}
 \ell_i \cap \ell_j = \lbrace [1:0:0] \rbrace, &\quad \ell_i \cap \ell_4 = \lbrace [\lambda_i': \lambda_i':1] \rbrace, \quad
  \ell_i \cap \ell_{5,6}  = \lbrace [\pm 2: \lambda_i':1] \rbrace, \\
  \ell_4 \cap \ell_{5,6} &= \lbrace [2:2:\mp1] \rbrace, \quad  \ell_5 \cap \ell_{6} = \lbrace [0:1:0] \rbrace.
\end{align*}
The parameters are given by
\begin{equation}
\label{moduli}
 \lambda_1'= \frac{\theta_1^4+\theta_2^4-\theta_3^4-\theta_4^4}{\theta_1^2\theta_2^2-\theta_3^3\theta_4^2}, \quad \lambda_2' =\frac{\theta_1^4+\theta_2^4}{\theta_1^2\theta_2^2}, \quad \lambda_3' =\frac{\theta_3^4+\theta_4^4}{\theta_3^2\theta_4^2} \;.
\end{equation}
 \end{lemma}
 \begin{proof}
 The proof follows directly by setting $Z_1=U_2Z$, $Z_2=X$, $Z_3=l \, U_1 Z$ in Equation~(\ref{K3_2_ell}).
 In terms of affine coordinates we have $Z_1=u$, $Z_2=x$, $Z_3=1$.
 \end{proof}
 There is a second way of obtaining  surface $\mathcal{X}_3$ as double branched cover of $\mathbb{P}^2$ using the moduli 
 $\Lambda_1', \Lambda_2', \Lambda_3'$:
\begin{lemma}
\label{lem:double_cover}
The surface $\mathcal{X}_3$ is the double branched cover of $\mathbb{P}^2$ given by
\begin{equation}
\label{Eq:SixLines_b}
 Y^2 = L_1 \, L_2 \, L_3 \, L_4 \,  L_5 \, L_6 \;,
\end{equation}
where the six lines $L_1 , \dots,  L_6$ in $\mathbb{P}^2$ are given by
\begin{equation}
\label{Eq:SixLines_def_b}
 L_i: Z_2 - \Lambda_i' Z_3=0 \, (\text{for $i=1,2,3$}), \quad L_4: Z_1 -Z_2=0, \quad L_{5,6}: Z_1 \pm 2 Z_3=0 \;.
\end{equation} 
In particular, the surface $\mathcal{X}_3$ is the minimal non-singular model for the double 
cover of the projective plane branched along six lines.  
Moreover, the six lines intersect as follows
\begin{align*}
 L_i \cap L_j = \lbrace [1:0:0] \rbrace, &\quad L_i \cap L_4 = \lbrace [\Lambda_i': \Lambda_i':1] \rbrace, \quad
  L_i \cap L_{5,6}  = \lbrace [\pm 2: \Lambda_i':1] \rbrace, \\
  L_4 \cap L_{5,6} &= \lbrace [2:2:\mp1] \rbrace, \quad  L_5 \cap L_{6} = \lbrace [0:1:0] \rbrace.
\end{align*}
The parameters are given by
\begin{equation}
\label{moduli-b}
 \Lambda_1'= \frac{\Theta_1^4+\Theta_2^4-\Theta_3^4-\Theta_4^4}{\Theta_1^2\Theta_2^2-\Theta_3^3\Theta_4^2}, \quad \Lambda_2' =\frac{\Theta_1^4+\Theta_2^4}{\Theta_1^2\Theta_2^2}, \quad \Lambda_3' =\frac{\Theta_3^4+\Theta_4^4}{\Theta_3^2\Theta_4^2} \;.
\end{equation}
 \end{lemma}
 \begin{proof}
In Equation~(\ref{kummer_right_ell_p_W}), we use the transformation in Equations~(\ref{relations_RosRoots}) and set
\begin{align*}
 u & =\lambda_1' + \dfrac{4 (\lambda_1'-\lambda_2') (\lambda_1'-\lambda_3')}{(\lambda_2'-\lambda_3')U+2\lambda_2'+2\lambda_3'-4\lambda_1'},\\
 x & = \frac{256 \, (\Lambda_1'-\Lambda_2')^3 (\Lambda_1'-\Lambda_3')^3 (U-\Lambda_2')\,(U-\Lambda_3')\,(X-U)}{(\Lambda_2'-\Lambda_3')^4(\Lambda_1'-2)(\Lambda_1'+2)(U-\Lambda_1')^3}\;.
 \end{align*}
In terms of these affine coordinates we then have $Z_1=X$, $Z_2=U$, $Z_3=1$.
 \end{proof}

We immediately obtain the following:
\begin{corollary}
\label{lem:branched_double}
The surface $\mathcal{X}_3$ is the minimal non-singular model for 
the double cover of the projective plane branched along six lines, three of which have a common point.
The situation is shown in Figure~\ref{Fig:6Lines}.
\end{corollary}

\subsubsection{Special Elliptic fibration from double cover}
The elliptic pencil on $\mathcal{X}_3$ discussed in Proposition~\ref{lem:EllRight}
corresponds to the line with four points (marked by intersections) as it runs through the pencil determined by the one common point of three line.
Let us describe the situation in more detail:

Let $L_1, L_2, \dots, L_6$ be the six lines in $\mathbb{P}^2$ of Lemma~\ref{lem:double_cover}. Assume that the first three lines meet at a common point $p_0$. The six lines are subject to no other general conditions. We shall construct the surface $\mathcal{X}_3$ obtained as the double cover of $\mathbb{P}^2$ 
branched over the six lines. 

Take the blow-up of $\mathbb{P}^2$ at the point $p_0$ and denote by $E_0$ the resulting exceptional curve. 
The proper transforms of $L_1, L_2$ and $L_3$ meet $E_0$ at 
three points. We blow up these points with resulting exceptional curves $E_{01}$, $E_{02}$ and $E_{03}$.  
Furthermore, we denote by $p_{ij}$ the intersection point $L_i \cap L_j$ and we perform blow-ups at these twelve points with 
resulting exceptional curves $E_{ij}$.  
Denote by $\mathcal{R}$ the rational surface obtained from $\mathbb{P}^2$ after the above sixteen blow-ups. A standard computation 
gives for twice the canonical divisor
\begin{equation}
\label{canonical}
2 K_{\mathcal{R}} \ \sim \ E_0 + \sum_{i=1}^6 L_i \ .
\end{equation}
Here, by abuse of notation, $L_i$ also denoted the proper transform on $\mathcal{R}$ of the line $L_i$. 

Consider $\mathcal{X}_3$ as the surface obtained as double cover of $\mathcal{R}$ branched over the right-hand-side divisor of $(\ref{canonical})$. Since 
the branching locus is smooth, the double-cover surface $\mathcal{X}_3$ is also smooth. Moreover, via $(\ref{canonical})$ and the Hurwitz formula, 
one sees that $\mathcal{X}_3$ is a K3 surface. 
Note that each of the exceptional curves $E_{ij}$ meets the branching locus exactly twice. Therefore, they lift on $\mathcal{X}_3$ to rational curves 
which we shall denote by $G_{ij}$. We also denote by $G_0$ and $S_1, S_2, \dots, S_6$ the proper transforms 
of the curves in the branching locus.  

Consider the pencil of lines through the point $p_0$. This pencil becomes base-point free when lifted to $\mathcal{R}$, i.e., $ \vert H - E_0 \vert $.
It defines a ruling on $\mathcal{R}$ whose generic member is a smooth rational curve that meets $E_0$, $L_4, L_5, L_6$ 
exactly once and does not meet $L_1, L_2, L_3$. Pulling back the pencil to the K3 surface $\mathcal{X}_3$, we obtain an elliptic fibration 
$\pi^{\mathcal{X}_3} \colon \mathcal{X}_3 \rightarrow \mathbb{P}^1$. An analysis of the singular fibers in this fibration reveals the following situation. 
One has three singular fibers of type $I_0^*$ given by
\begin{equation}
\label{I0star-fibers}
\begin{split}
2S_1 + G_{01} + G_{14} + G_{15} + G_{16} , &\qquad 2S_2 + G_{02} + G_{24} + G_{25} + G_{26}, \\
2S_3 + G_{03} + &G_{34} + G_{35} + G_{36} .
\end{split}
\end{equation}
There are additional singular fibers. Consider the three lines in $\mathbb{P}^2$ passing through $p_0$ and $p_{ij}$ where $ \{ i, j \}  \subset \{4,5,6\}$. The proper transforms of these lines in $\mathcal{R}$ meet the branching locus exactly twice and hence they lift to rational curves in $\mathcal{X}_3$ which we denote by 
$U_{ij}$. Then the following make three singular fibers of type $I_2$ in our elliptic fibration:
$$ G_{45}+ U_{45}, \ \ G_{46}+ U_{46}, \ \ G_{56}+ U_{56}  \ .$$
The Euler characteristics of the above singular fibers add up to twenty-four. Hence the above make all the singular fibers of the elliptic fibration $\pi^{\mathcal{X}_3} \colon \mathcal{X}_3 \rightarrow \mathbb{P}^1 $.  We note that the fibration $\pi^{\mathcal{X}_3}$ has four sections given by the rational curves $E_0$, $L_1$, $L_2$, and $L_3$. These generate a Mordell-Weil group $\operatorname{MW}(\pi^{\mathcal{X}_3}, E_0)$ isomorphic to $(\mathbb{Z}/2)^2$. 
\par We also note that a second elliptic fibration on $\mathcal{X}_3$ with a similar set of singular fibers/sections can be obtained by considering the pencil of conics through the four points $p_0$, $p_{45}, p_{46}, p_{56}$.

\begin{figure}[ht]
  $$
  \begin{xy}
    <0cm,0cm>;<1cm,0cm>:
    (-1,0.5);(5.5,0.5)**@{-}, 
    (-1,2.8);(5.5,-0.5)**@{-},   
    (-1,2);(5.5,1.1)**@{-},    
    (-1,-0.5);(5,3.5)**@{-},
    (4,3.5);(4,-0.5)**@{-},
    (4.35,3.5);(2.2,-0.5)**@{-},
    (4.31,2.6)*++!\hbox{\tiny{$p_0$}},
    (4.31,-0.5)*++!\hbox{\tiny{$L_1$}},
    (2.6,-0.5)*++!\hbox{\tiny{$L_2$}},
    (-.4,-0.5)*++!\hbox{\tiny{$L_3$}},
  \end{xy}
  $$
\caption{\label{Fig:6Lines}}
\end{figure}

In the context of the above double cover model we also have:
\begin{theorem}
\label{prop:double_cover}
The minimal non-singular model for  the double cover of the projective plane branched along six lines, 
three of which have a common point, is the $(1,2)$-polarized Kummer surface $\operatorname{Kum}(\mathbf{B}_{12})$.
\end{theorem}
\begin{proof}
The proof follows from applying Lemma~\ref{lem:branched_double} and Theorem~\ref{cor:equiv}
to $\mathcal{X}_3=\operatorname{Kum}(\hat{\mathbf{B}}_{12})$.
\end{proof}

\begin{remark}
\label{Gab2}
As pointed out in Remark~\ref{Gab1}, there is a description of $\operatorname{Kum}(\mathbf{B}_{12})$ as specialization within a family of K3 surface with Picard rank fourteen
whose general member is the double cover of the projective plane with a branch locus that is the 
union of four divisors of bidegree $(1,1)$ in $\mathbb{P}^1 \times \mathbb{P}^1$. In \cite[Sec.~2.4.5]{MR2363136orig} Garbagnati gives another description of the branch locus
as reducible sextic: there is a projection of $\mathbb{P}^1 \times \mathbb{P}^1$ to $\mathbb{P}^2$ 
such that the union of the four divisors of bidegree $(1,1)$ is sent in the union of two conics and two lines. This is achieved by embedding 
$\mathbb{P}^1 \times \mathbb{P}^1 \hookrightarrow \mathbb{P}^3$ such that the projection of $\mathbb{P}^1 \times \mathbb{P}^1$ 
from one of the points of intersections of two divisor is a plane. The four divisors are then mapped to two lines and two conics.
Garbagnati also proves that the general member admits a genus one curve fibration (without section) with three fibers of type $I_0^*$ \cite[Sec.~2.4.5]{MR2363136orig}.
The specialization to $\operatorname{Kum}(\mathbf{B}_{12})$ is given by requiring that the branch locus is tangent to three conics.
\end{remark}

Similarly, we have the following:
\begin{lemma}
The surface $\mathcal{Y}_3=\operatorname{Kum}(\operatorname{Jac} \hat{\mathcal{C}\,}\!)$ 
is the double branched cover of $\mathbb{P}^2$ given by
\begin{equation}
\label{Eq:SixLinesB}
 Y^2 = X \, \big(X^2 + U \, X + 1\big)\, (U-\Lambda_1') \, (U-\Lambda_2') \, (U-\Lambda_3')\;,
\end{equation}
In particular, $\mathcal{Y}_3$ is the double cover of the projective plane branched  
along four lines, exactly three of which have a common point, and a conic tangent to the fourth
line which does not meet the common point. The situation is shown in Figure~\ref{Fig:4Lines1Conic}. 
\end{lemma}
\begin{proof}
In Equation~(\ref{K3_2_ell_dual_W}), we rescale $X\mapsto (u^2-4) (u-\lambda_1')(\lambda_2'-\lambda_3')X$,
use the transformation in Equations~(\ref{relations_RosRoots}), and introduce
the new variable 
\[
 U=\Lambda_1' + \dfrac{4 (\Lambda_1'-\Lambda_2') (\Lambda_1'-\Lambda_3')}{(\Lambda_2'-\Lambda_3')u+2\Lambda_2'+2\Lambda_3'-4\Lambda_1'} \;.
\]
After rescaling $Y$ we obtain Equation~(\ref{Eq:SixLinesB}).
\end{proof}

\begin{figure}[ht]
  $$
  \begin{xy}
    <0cm,0cm>;<1cm,0cm>:
    (-1,0);(5.5,0)**@{-},
    (-1,-0.5);(5,3.5)**@{-},
    (4,3.5);(4,-0.5)**@{-},
    (4.35,3.5);(2.2,-0.5)**@{-},
    (2,0.9)*\xycircle(2.2,.9){},
    (5.2,0.1)*++!\hbox{\tiny{$X$}},
    (4.51,-0.5)*++!\hbox{\tiny{$U-\Lambda_1'$}},
    (2.8,-0.5)*++!\hbox{\tiny{$U-\Lambda_2'$}},
    (-.2,-0.5)*++!\hbox{\tiny{$U-\Lambda_3'$}},
  \end{xy}
  $$
\caption{\label{Fig:4Lines1Conic}}
\end{figure}

\subsection{Degeneration to higher Picard number}
In this section we describe the limit where the Picard number of the constructed Jacobian elliptic K3 surfaces increases to $18$.
If we set $\Lambda_1=K_1$, $\Lambda_2=K_2 \epsilon$, $\Lambda_3=\epsilon$, and $R^2= K_1 K_2$ we obtain
\begin{equation}
 \lim_{\epsilon \to 0} \Lambda'_1= \infty, \quad    \Lambda'_2= \frac{K_1+K_2}{R}, \quad \Lambda'_3= \frac{1+K_1K_2}{R} \;.
\end{equation}

The limit is realized by replacing $(U-\Lambda_1') \, (U-\Lambda_2') \, (U-\Lambda_3')$ by $(U-\Lambda_2') \, (U-\Lambda_3')$ in
Equation~(\ref{Eq:SixLinesB}). Equation~(\ref{Eq:SixLines}) is then easily shown to
coincide with the elliptic fibration $\mathfrak{J}_7$ in the list of all possible elliptic fibrations on the Kummer surface 
of two non-isogenous elliptic curves determined by Kuwata and Shioda in~\cite{MR2409557} .
The following lemma then follows immediately:
\begin{lemma}
\label{lem:degen_Y3}
In the limit $\Lambda_1' \to \infty$ and
\begin{equation}
 \Lambda'_2 = \frac{K_1+K_2}{R} \,, \quad \Lambda'_3= \frac{1+K_1 K_2}{R} \,, \quad R^2= K_1 K_2 \;,
\end{equation} 
the elliptic surface $\mathcal{Y}_3$ is the Kummer surface $\operatorname{Kum}(\hat{\mathcal{E}\,}\!_1 \times \hat{\mathcal{E}\,}\!_2)$
of two non-isogenous elliptic curves where each elliptic curve is given in Legendre normal form by
\begin{equation}
\label{Eq:ECcurve_dual}
 \hat{\mathcal{E}\,}\!_i=\mathcal{E}(K_i): \quad y^2 = x \, (x-1) \, (x-K_i) 
\end{equation}
for $i=1,2$.
\end{lemma}
\begin{remark}
The same result applies to the fibration on $\mathcal{Y}_2$ given by Equation~(\ref{kummer3_ell_dual_W}):
in the limit $\Lambda_1' \to \infty$ and after rescaling it coincides with the elliptic fibration $\mathfrak{J}_1$ in the list of all possible elliptic fibrations on the Kummer surface 
of two non-isogenous elliptic curves determined by Kuwata and Shioda in~\cite{MR2409557}.
\end{remark}

Taking the limit $\lambda_1'\to \infty$ and rescaling is achieved for the surface $\mathcal{Y}_1$ by replacing Equation~(\ref{kummer_middle_ell_p_W}) by 
\begin{equation}
\label{kummer_middle_ell_p_W_b}
\begin{split}
Y^2  = \; X \, \big( X- t \, \left( t^2 - t \, \lambda'_3  +1 \right)  \big) \, \big( X- t \,\left( t^2 - t \, \lambda'_2  +1 \right)  \big) \;.
\end{split}
\end{equation}
Equation~(\ref{kummer_middle_ell_p_W_b}) is easily shown to
coincide with the elliptic fibration $\mathfrak{J}_6$ in the list of all possible elliptic fibrations on the Kummer surface 
of two non-isogenous elliptic curves determined by Kuwata and Shioda in~\cite{MR2409557} .
We have the following:
\begin{lemma}
\label{lem:degen_Y1}
In the limit $\lambda_1' \to \infty$ with
\begin{equation}
 \lambda'_2 = \frac{k_1+k_2}{r} \,, \quad \lambda'_3= \frac{1+k_1 k_2}{r} \,, \quad r^2= k_1 k_2 \;,
\end{equation} 
the elliptic surface $\mathcal{Y}_1$ is the Kummer surface $\operatorname{Kum}(\mathcal{E}_1 \times \mathcal{E}_2)$
of two non-isogenous elliptic curves where each elliptic curve is given in Legendre normal form by
\begin{equation}
\label{Eq:ECcurve}
 \mathcal{E}_i=\mathcal{E}(k_i): \quad y^2 = x \, (x-1) \, (x-k_i) 
\end{equation}
for $i=1,2$.
\end{lemma}

\begin{remark}
In addition, sending $\Lambda'_2 \to 2$, that is $K_2 \to K_1$,
or, equivalently, $\Lambda'_3 \to 2$, that is $K_2 \to 1/K_1$, we obtain a one-parameter family of Jacobian elliptic K3 surfaces of Picard number 19.
For example, Equation~(\ref{Eq:SixLines}) and Equation~(\ref{Eq:SixLinesB}) then describe Jacobian elliptic surfaces with singular fibers
$2\, I_2^* + I_2 + I_0^*$ and $I_4^* + I_1^* + I_1 + I_0^*$, respectively. These families were considered by Hoyt in \cite{MR894512}.
In this limit, the surfaces $\mathcal{Y}_i$ for $i=1, \dots, 3$ are all Kummer surfaces of two isogenous elliptic curves.
\end{remark}

\subsubsection{Elliptic curve two-isogenies and moduli}
The elliptic curve given by
\begin{equation}
\label{EllC1}
\mathcal{E}(k): \quad y^2 = x \, (x-1) \, (x-k),
\end{equation}
has the two-torsion point $P: (x,y)=(0,0)$. The two-isogenous elliptic curve $\hat{\mathcal{E}\,}\!$ obtained by translation by the order-two point $P$ 
is given by
\begin{equation}
\label{EllC1_dual}
\hat{\mathcal{E}\,}\!=\mathcal{E}(k)/\langle P \rangle: \quad y^2 = x \, \left( x + \big(\sqrt{k} + 1\big)^2 \right) \, \left( x + \big(\sqrt{k} - 1\big)^2 \right) \;.
\end{equation}
The six values of the cross-ratio of the four ramification points $\{0, (\sqrt{k} \pm 1)^2, \infty\}$ are given by
\[
 \left\lbrace K, \; 1-K, \; \frac{1}{K}, \; 1-\frac{1}{K}, \; \frac{1}{1-K}, \; \dfrac{1}{1-\frac{1}{K}} \right \rbrace
\]
with
\begin{equation}
 K= \left(\frac{1-\sqrt{k}}{1+\sqrt{k}} \right)^2 \;,
\end{equation}
and we set $\hat{\mathcal{E}\,}\! = \mathcal{E}(K)$. If we introduce the Jacobi moduli $m$ and $M$ given by $k=m^2$ and $K=M^2$, respectively, we
obtain
\begin{equation}
\label{def_M}
 M=  \frac{1-m}{1+m}  , \quad  m=  \frac{1-M}{1+M}   \;.
\end{equation}

In terms of $\tau \in \mathbb{H}$ and $q=\exp{(2\pi i \tau)}$ the Jacobi-theta-functions are defined by
\begin{equation}
 \theta_2 = \sum_{n\in\mathbb{Z}} q^{\frac{(2n+1)^2}{8}} , \quad
 \theta_3 = \sum_{n\in\mathbb{Z}} q^{\frac{n^2}{2}} , \quad
 \theta_4 = \sum_{n\in\mathbb{Z}} (-1)^n \, q^{\frac{n^2}{2}},
\end{equation}
such that $\theta_3^4=\theta_2^4+ \theta_4^4$. In terms of theta-functions
expressions for the elliptic curve modulus and Jacobi modulus  are given by $k=  \theta_4^4/\theta_3^4$,
and $m =  \theta_4^2/\theta_2^4$, respectively.
We will denote by $\Theta_k$ for $k=2,3,4$ the functions
with doubled modular parameter. In the theory of elliptic theta-functions, we have the well known relations \cite[Sec.~4]{MR2485477}
\begin{equation}
\label{Eq:2isogEC}
\begin{split}
 \Theta_2^2 = \frac{1}{2} \, \left( \theta_3^2 - \theta_4^2 \right) , \quad \Theta_3^2 = \frac{1}{2} \, \left( \theta_3^2 + \theta_4^2 \right) , \quad \Theta_4^2 = \theta_3 \, \theta_4  \;,
\end{split} 
\end{equation}
and
 \begin{equation}
\label{Eq:2isogEC_b}
\begin{split}
 \theta_4^2 =  \Theta_3^2 - \Theta_2^2  , \quad \theta_3^2 = \Theta_3^2 + \Theta_2^2 , \quad \theta_2^2 = 2 \,\Theta_2 \Theta_3  \;.
\end{split} 
\end{equation}

Expressions for the elliptic curve modulus and Jacobi modulus compatible with Equation~(\ref{def_M}) and Equation~(\ref{Eq:2isogEC}) are given by
$K=  \Theta_2^4/\Theta_3^4$ and $M=  \Theta_2^2/\Theta_3^2$. Observe that the change corresponding to $\theta_4 \leftrightarrow \Theta_2$
was -- for genus-two theta-functions -- compensated by switching the roles of $\underline{a}^{(i)}$ and $\underline{b}^{(i)}$ 
in Equation~(\ref{Eqn:theta_short}) and Equation~(\ref{Eqn:Theta_short}).
We then have the following:
\begin{proposition}
In the limit $\epsilon \to 0$ with $\lambda_1=k_1$, $\lambda_2=k_2\, \epsilon$, $\lambda_3=\epsilon$ in Equation~(\ref{Eq:Rosenhain}), 
the principally polarized abelian surface $\operatorname{Jac}(\mathcal{C})$ is isomorphic 
to $\mathcal{E}_1 \times \mathcal{E}_2$ where each elliptic curve is given in Legendre normal form by
\begin{equation}
\label{Eq:ECcurve_b}
 \mathcal{E}_i=\mathcal{E}(k_i): \quad y^2 = x \, (x-1) \, (x-k_i) 
\end{equation}
for $i=1,2$. Additionally, in the limit $\epsilon \to 0$ the isogenous abelian surface 
$\operatorname{Jac}(\hat{\mathcal{C}\,}\!_{12})$ is isomorphic to $\hat{\mathcal{E}\,}\!_1 \times \hat{\mathcal{E}\,}\!_2$
where each elliptic curve is given in Legendre normal form by
\begin{equation}
\label{Eq:ECcurve_dual_b}
 \hat{\mathcal{E}\,}\!_i=\mathcal{E}(K_i): \quad y^2 = x \, (x-1) \, (x-K_i) 
\end{equation}
for $i=1,2$ and the Jacobi moduli $m_i, M_i$ with $k_i=m_i^2$ and $K_i=M_i^2$, respectively, are related by two-isogeny, i.e.,
\begin{equation}
\label{2iso_ec}
m_1 =  \pm \frac{1-M_1}{1+M_1} , \qquad m_2 =  \mp \frac{1+M_2}{1-M_2} \;.
\end{equation}
In the limit $\epsilon \to 0$, the $(2,2)$-isogeny $\psi_{12}: \operatorname{Jac}(\mathcal{C}) \to \operatorname{Jac}(\hat{\mathcal{C}\,}\!_{12})$
is the product of the two-isogenies of elliptic curves $\mathcal{E}_i \to \hat{\mathcal{E}\,}\!_i$.
\end{proposition}
\begin{proof}
We already proved in Lemma~\ref{lem:degen_Y1} that the surface $\mathcal{Y}_1 = \operatorname{Kum}(\operatorname{Jac} \mathcal{C})$ degenerates
to $\operatorname{Kum}(\mathcal{E}_1 \times \mathcal{E}_2)$ for $\lambda_1=k_1$, $\lambda_2=k_2 \epsilon$, $\lambda_3=\epsilon$ in the limit $\epsilon \to 0$.
Similarly, we proved in Lemma~\ref{lem:degen_Y3} that the surface $\mathcal{Y}_3 = \operatorname{Kum}(\operatorname{Jac} \hat{\mathcal{C}\,}\!)$ degenerates
to $\operatorname{Kum}(\hat{\mathcal{E}\,}\!_1 \times \hat{\mathcal{E}\,}\!_2)$ for $\Lambda_1=K_1$, $\Lambda_2=K_2 \epsilon'$, $\Lambda_3=\epsilon'$ in the limit $\epsilon' \to 0$
such that 
\begin{equation}
  \Lambda'_1= \frac{M_1}{M_2 \, \epsilon'} + \frac{M_2 \, \epsilon'}{M_1}, \quad    \Lambda'_2= \frac{M^2_1+M^2_2}{M_1 M_2}, \quad \Lambda'_3= \frac{1+M^2_1M^2_2}{M_1M_2} \;.
\end{equation}
We plug $\lambda_1=m^2_1$, $\lambda_2=m^2_2 \epsilon$, $\lambda_3=\epsilon$ into Equations~(\ref{relations_RosRoots}), use Equation~(\ref{2iso_ec}), and obtain
\begin{align*}
 \Lambda_1' & = \frac{(1\pm M_1)^2(1\pm M_2)^2}{4 \, M_1  M_2 \, \epsilon} + O(1),\\
 \Lambda_2' & = \frac{M_1^2+M_2^2}{M_1 \, M_2} + O(\epsilon),\\
 \Lambda_3' & = \frac{1+M_1^2 M_2^2}{M_1 \, M_2} + O(\epsilon).
\end{align*}
\end{proof}

\section{Summary of surfaces and morphisms}
\label{sec:summary}
Figure~\ref{Relations_Fibrations} depicts all constructed Jacobian elliptic surfaces and maps. 
Table~\ref{tab:surfaces} lists their singular fibers and Mordell-Weil groups. We highlight some of the key features:

\begin{figure}[ht]
\scalebox{0.9}{
\centerline{
\xymatrix{
 &	*+[F-,]{\mathcal{Y}_2 = \operatorname{Kum}(\operatorname{Jac} \hat{\mathcal{C}\,}\!)} 	\ar[ld] 	\ar@{-->}[d] 	
 \ar[rd]  \ar@{->}[rdd]^{\;\hat{\Psi}}|{\phantom{\text{\Huge{W}}_{\text{\Huge{W}}}}} |\hole |!{"2,3";"3,2"}\hole \\
	*+[F.]{\mathcal{Z}_2 = \operatorname{Rat}} 	\ar@{-->}[d]	
 &	*+[F--]{\mathcal{X}_2 =\operatorname{Kum}(\hat{\mathbf{B}})  \cong  \operatorname{Kum}(\mathbf{B})}\ar[ld]^{\phi^{\mathcal{X}_2}_{\mathcal{Z}_1}} 	\ar[rd]_{\phi^{\mathcal{X}_2}_{\mathcal{Y}_1}}	\ar@<1ex>[u]^{\phi^{\mathcal{X}_2}_{\mathcal{Y}_2}}
 &	*+[F--]{\mathcal{X}_1 = \operatorname{Kum}(\mathbf{B}')}	\ar[rd] \ar[dl]|!{"2,2";"3,3"}\hole  	\ar[d]^{\phi^{\mathcal{X}_1}_{\mathcal{Y}_1}} \\
	*+[F.]{\mathcal{Z}_1 = \operatorname{Rat}}  	\ar@{->}[u]<1ex>
& 	*+[F.]{\mathcal{Z}_4 = \operatorname{Rat}} 	\ar[d] 		
& 	*+[F-,]{\mathcal{Y}_1 = \operatorname{Kum}(\operatorname{Jac} \mathcal{C})}	\ar@{-->}[u]<1ex>	\ar[rd]_{\phi^{\mathcal{Y}_1}_{\mathcal{X}_3}}	\ar[ld]^{\phi^{\mathcal{Y}_1}_{\mathcal{Z}_3}}  	
\ar[r]^{\Psi}
&	*+[F-,]{\mathcal{Y}_3 = \operatorname{Kum}(\operatorname{Jac} \hat{\mathcal{C}\,}\!)} 	\ar@{-->}[d] \\
&	*+[F.]{\mathcal{Z}_3 = \operatorname{Rat}}		\ar@{-->}[u]<1ex>		
&										
&	*+[F--]{\mathcal{X}_3 = \operatorname{Kum}(\hat{\mathbf{B}})} \ar@<1ex>[u]^{\phi^{\mathcal{X}_3}_{\mathcal{Y}_3}}
}}}
\caption{\label{Relations_Fibrations}}
\end{figure}

\begin{table}
\scalebox{0.9}{
\begin{tabular}{|c|c|c|c|}
$\#$ & type & fibration & $\operatorname{MW}(\pi)$ \\
\hline
&&&\\[-0.9em]
$\mathcal{Y}_2$ & $\operatorname{Kum}(\operatorname{Jac} \hat{\mathcal{C}\,}\!)$ & $4 \, I_4 + 8 \, I_1$ & $\mathbb{Z}/2\mathbb{Z} \oplus \langle 1 \rangle^{\oplus3}$  \\[0.2em]
$\mathcal{X}_2$ & $\begin{array}{r}\operatorname{Kum}(\hat{\mathbf{B}}_{12}) \\\cong \operatorname{Kum}(\mathbf{B}_{12}) \end{array}$ & $12 \, I_2$ & $(\mathbb{Z}/2\mathbb{Z})^2 \oplus \langle \frac{1}{2} \rangle_{\mathsf{p}} \oplus   \langle \frac{1}{2} \rangle_{\mathsf{q}} \oplus   \langle \frac{1}{2} \rangle_{\mathsf{r}}$  \\[1em]
$\mathcal{X}_1$ & $\operatorname{Kum}(\mathbf{B}_{34})$ & $2 \, I_4 + 4 \, I_1 + 2\, I_0^*$ & $\mathbb{Z}/2\mathbb{Z} \oplus \langle \frac{1}{2} \rangle$ \\[0.2em]
$\mathcal{Y}_1$ & $\operatorname{Kum}(\operatorname{Jac} \mathcal{C})$ & $6 \, I_2 + 2 \, I_0^*$ & $(\mathbb{Z}/2\mathbb{Z})^2 \oplus \langle 1 \rangle_{\mathsf{P}}$  \\[0.2em]
$\mathcal{Y}_3$ & $\operatorname{Kum}(\operatorname{Jac} \hat{\mathcal{C}\,}\!)$ & $I_4 + 2 \, I_1 + 3 \, I_0^*$ & $\mathbb{Z}/2\mathbb{Z}$  \\[0.1em]
$\mathcal{X}_3$ & $\operatorname{Kum}(\hat{\mathbf{B}}_{12})$ & $3 \, I_2 + 3\, I_0^*$ & $(\mathbb{Z}/2\mathbb{Z})^2 $  \\[0.2em]
$\mathcal{Z}_1$ & rational & $6 \, I_2$ & $(\mathbb{Z}/2\mathbb{Z})^2  \oplus \langle \frac{1}{2} \rangle_{\mathsf{Q}} \oplus \langle \frac{1}{2} \rangle_{\mathsf{R}}$ \\[0.2em]
$\mathcal{Z}_2$ & rational & $2 \, I_4 + 4 \, I_1$ & $\mathbb{Z}/2\mathbb{Z}  \oplus \langle \frac{1}{2} \rangle^{\oplus 2}$ \\[0.2em]
$\mathcal{Z}_3$ & rational & $3 \, I_2 + I_0^*$ & $(\mathbb{Z}/2\mathbb{Z})^2  \oplus \langle \frac{1}{2} \rangle_{\mathsf{p}}$ \\[0.2em]
$\mathcal{Z}_4$ & rational & $I_4 + 2 \, I_1 + I_0^*$ & $\mathbb{Z}/2\mathbb{Z}  \oplus \langle \frac{1}{4} \rangle$ \\[0.2em]
\end{tabular}}
\medskip
\caption{\label{tab:surfaces}}
\end{table}

\begin{enumerate}
\item The vertical arrows indicate fiberwise two-isogenies, and the dual isogenies are the dashed vertical arrows. The corresponding involutions are the Van Geemen-Sarti 
involutions obtained as fiberwise translation by a two-torsion section in $\operatorname{MW}(\pi)$.

\item  For the Kummer surface $\operatorname{Kum}(\operatorname{Jac} \mathcal{C})$,
the Van Geemen-Sarti  involution $\imath^{\mathcal{Y}_1}_F$
covering $\phi_{\mathcal{Y}_1}^{\mathcal{X}_1}$ descended from the translation of $\mathbf{A}=\operatorname{Jac}(\mathcal{C})$ by the two-torsion 
point $\mathfrak{e}_{34} \in \mathbf{A}[2]$. Similarly, the involutions $\imath^{\mathcal{Y}_2}_F$ and $\imath^{\mathcal{Y}_3}_F$ covering the dual two-isogeny 
${\phi}^{\mathcal{Y}_2}_{\mathcal{X}_2}$ and ${\phi}^{\mathcal{Y}_3}_{\mathcal{X}_3}$, respectively, descend from a  translation of 
$\hat{\mathbf{A}}=\operatorname{Jac} \hat{\mathcal{C}\,}\!_{12}$ by the two-torsion point $\hat{\mathfrak{e}}_{12} \in \hat{\mathbf{A}}[2]$.

\item The diagonal arrows pointing from the right to the left are degree-two rational maps relating a K3 surface to a rational surface by doubling
the singular fibers.

\item The diagonal arrows pointing from $\mathcal{X}_i$ to $\mathcal{Y}_{i'}$ or vice versa are rational double covers such that the rational base curve is mapped into
another rational curve with two ramification points located at the base of two fibers of type $I_0^*$. The Jacobian elliptic fibration in the domain is the pull-back
of the fibration in the range along this base transformation. In particular, the maps $\phi^{\mathcal{X}_2}_{\mathcal{Y}_1}$ and $\phi^{\mathcal{Y}_1}_{\mathcal{X}_3}$
lift the base transformations $s \mapsto t=s^2$ and $t \mapsto u = t + 1/t$ whose sheets are interchanged by the Nikulin involutions $\jmath^{\mathcal{X}_2}: s \to -s$
and $\imath^{\mathcal{Y}_1}: t \to 1/t$, respectively.

\item The rational elliptic surfaces $\mathcal{Z}_2$ and $\mathcal{Z}_4$ were not discussed in detail. They play the same role for $\mathcal{Y}_2$ and $\mathcal{X}_1$ 
as the rational surfaces $\mathcal{Z}_1$ and $\mathcal{Z}_3$ do for $\mathcal{X}_2$ and $\mathcal{Y}_1$. Moreover, they are related by fiberwise two-isogeny
to $\mathcal{Z}_1$ and $\mathcal{Z}_3$.

\item The two sheets of the morphism $\phi^{\mathcal{X}_2}_{\mathcal{Y}_1}$ and $\phi^{\mathcal{X}_2}_{\mathcal{Z}_1}$ are interchanged by  the
Nikulin involutions $\jmath^{\mathcal{X}_2}$, i.e.,
$\phi^{\mathcal{X}_2}_{\mathcal{Y}_1} \circ \jmath^{\mathcal{X}_2}  = \phi^{\mathcal{X}_2}_{\mathcal{Y}_1}$ and 
$\phi^{\mathcal{X}_2}_{\mathcal{Z}_1} \circ \jmath^{\mathcal{X}_2}  = \phi^{\mathcal{X}_2}_{\mathcal{Z}_1}$.
The involution $\jmath^{\mathcal{X}_2}$ lifts to an involution
$\jmath^{\mathcal{Y}_2}$ such that $\jmath^{\mathcal{X}_2} \circ \phi^{\mathcal{Y}_2}_{\mathcal{X}_2}=\phi^{\mathcal{Y}_2}_{\mathcal{X}_2} \circ \jmath^{\mathcal{Y}_2}$.

\item The two sheets of the morphism $\phi^{\mathcal{Y}_1}_{\mathcal{X}_3}$ and $\phi^{\mathcal{Y}_1}_{\mathcal{Z}_3}$ are interchanged by  the 
Nikulin involutions $\imath^{\mathcal{Y}_1}$, i.e.,
$\phi^{\mathcal{Y}_1}_{\mathcal{X}_3} \circ \imath^{\mathcal{Y}_1}  =\phi^{\mathcal{Y}_1}_{\mathcal{X}_3}$ and 
$\phi^{\mathcal{Y}_1}_{\mathcal{Z}_3} \circ \imath^{\mathcal{Y}_1}  =\phi^{\mathcal{Y}_1}_{\mathcal{Z}_3}$.

\item The involution $\imath^{\mathcal{Y}_1}$ descends from a 
translation of $\mathbf{A}=\operatorname{Jac}(\mathcal{C})$ by the two-torsion point $\mathfrak{e}_{12} \in \mathbf{A}[2]$.
Similarly, the involution $\jmath^{\mathcal{Y}_2}$ descends from a 
translation of $\hat{\mathbf{A}}=\operatorname{Jac}(\hat{\mathcal{C}\,}\!_{12})$ by the two-torsion point $\hat{\mathfrak{e}}_{34} \in \hat{\mathbf{A}}[2]$.

\item The involutions $\imath^{\mathcal{X}_2}$ and $\imath^{\mathcal{Y}_1}$ are related by 
$\phi^{\mathcal{X}_2}_{\mathcal{Y}_1} \circ \imath^{\mathcal{X}_2} =\imath^{\mathcal{Y}_1} \circ \phi^{\mathcal{X}_2}_{\mathcal{Y}_1}$.

\item The non-torsion sections of $\mathcal{X}_2$ are related to the sections of $\mathcal{Y}_1$ and $\mathcal{Z}_1$ by $\phi^{\mathcal{X}_2}_{\mathcal{Y}_1} \circ  \mathsf{p} = \mathsf{P}$
and $\phi^{\mathcal{X}_2}_{\mathcal{Z}_1} \circ  \mathsf{q}, \mathsf{r} = \mathsf{Q}, \mathsf{R}$, respectively. Similar relations hold for the torsion sections.

\item The non-torsion section of $\mathcal{Y}_1$ is related to the section of $\mathcal{Z}_3$ by $\phi^{\mathcal{Y}_1}_{\mathcal{Z}_3} \circ  \mathsf{P}' = \mathsf{p}'$.  
Similar relations hold for the torsion sections, and the relation between torsion sections on $\mathcal{Y}_1$ and $\mathcal{X}_3$.

\item The maximal isotropic subgroup $\mathsf{k}_{12}$ of the two-torsion of $\operatorname{Jac}(\mathcal{C})$ defining the $(2,2)$-isogeny 
$\psi_{12}: \operatorname{Jac}(\mathcal{C}) \to \operatorname{Jac}(\hat{\mathcal{C}\,}\!_{12})$ is isomorphic to $( \mathbb{Z}/2\mathbb{Z})^2$.
The decomposition of the associated rational map $\Psi_{12}: \operatorname{Kum}(\operatorname{Jac} \mathcal{C}) 
\to \operatorname{Kum}(\operatorname{Jac} \hat{\mathcal{C}\,}\!_{12})$ of Kummer surfaces  into 
$\phi^{\mathcal{X}_3}_{\mathcal{Y}_3} \circ \phi^{\mathcal{Y}_1}_{\mathcal{X}_3}$
corresponds to a factorization of $\mathsf{k}_{12}$ generated by the involution $\imath^{\mathcal{Y}_1}$ and the 
Van Geemen-Sarti involution $\imath_F^{\mathcal{Y}_1}$.

\item Similarly, the composition $\phi^{\mathcal{X}_2}_{\mathcal{Y}_1} \circ \phi^{\mathcal{Y}_2}_{\mathcal{X}_2}$ gives the rational map $\hat{\Psi}_{12}$ 
between Kummer surfaces descending from the dual $(2,2)$-isogeny $\hat{\psi\,}\!_{12}: \hat{\mathcal{C}\,}\!_{12} \to \mathcal{C}$ associated
with the dual maximal isotropic subgroup $\mathsf{K}_{12}$ of $\operatorname{Jac}(\hat{\mathcal{C}\,}\!_{12})$.
The decomposition of $\hat{\Psi}_{12}$ into  $\phi^{\mathcal{X}_2}_{\mathcal{Y}_1} \circ \phi^{\mathcal{Y}_2}_{\mathcal{X}_2}$
corresponds to a factorization of $\mathsf{K}_{12}$ generated by the involution $\imath^{\mathcal{Y}_2}$ and the 
Van Geemen-Sarti involution $\imath_F^{\mathcal{Y}_2}$.

\item The Van Geemen-Sarti involutions $\imath_F^{\mathcal{Y}_1}$, $\imath_F^{\mathcal{Y}_2}$ and $\imath_F^{\mathcal{Y}_3}$ act equivariantly
with respect to $\Psi_{12}=\phi^{\mathcal{X}_3}_{\mathcal{Y}_3} \circ \phi^{\mathcal{Y}_1}_{\mathcal{X}_3}$ and 
$\hat{\Psi}_{12}=\phi^{\mathcal{X}_2}_{\mathcal{Y}_1} \circ \phi^{\mathcal{Y}_2}_{\mathcal{X}_2}$, i.e.,
$\Psi_{12} \circ \imath_F^{\mathcal{Y}_1} = \imath_F^{\mathcal{Y}_3}\circ \Psi_{12}$ and
$\hat{\Psi}_{12} \circ \imath_F^{\mathcal{Y}_2} = \imath_F^{\mathcal{Y}_1}\circ \hat{\Psi}_{12}$.

\item The Kummer surface $\operatorname{Kum}(\hat{\mathbf{B}}_{12})$ dominates and is dominated by 
$\operatorname{Kum}(\operatorname{Jac} \hat{\mathcal{C}\,}\!)$ by the rational maps of degree two which are the
fibrewise two-isogenies  $\phi^{\mathcal{X}_2}_{\mathcal{Y}_2}$ and $\phi^{\mathcal{Y}_2}_{\mathcal{X}_2}$
as well as  $\phi^{\mathcal{X}_3}_{\mathcal{Y}_3}$ and $\phi^{\mathcal{Y}_3}_{\mathcal{X}_3}$.
Similarly, $\operatorname{Kum}(\hat{\mathbf{B}}_{12})$ dominates and is dominated by 
$\operatorname{Kum}(\operatorname{Jac} \mathcal{C})$ by the rational maps of degree two which are the
rational base transformations $\phi^{\mathcal{X}_2}_{\mathcal{Y}_1}$ and $\phi^{\mathcal{Y}_1}_{\mathcal{X}_3}$.
This establishes three Kummer sandwich theorems for $(1,2)$-polarized Kummer surfaces similar
to what was done in \cite{MR2279280} for principally polarized Kummer surfaces.
\end{enumerate}

\bibliography{CM16} 
\bibliographystyle{amsplain}

\end{document}